\documentclass[12pt, 3p
			  ]{elsarticle}
\usepackage{amsmath,amsbsy,amsfonts}
\allowdisplaybreaks[4]
\usepackage{bm}
\usepackage{amssymb}

\usepackage{booktabs}	
\usepackage[colorlinks]{hyperref}
\usepackage{enumerate}
\usepackage{lineno}
\usepackage{stmaryrd} 

\usepackage{caption,subcaption}
\usepackage{graphicx}
\usepackage{graphics}
\usepackage{epstopdf}
\usepackage{pstricks-add}
\usepackage{CJKutf8,CJKnumb}
\AtBeginDocument{\begin{CJK}{UTF8}{gbsn}}
\AtEndDocument{\clearpage \end{CJK}}
\usepackage{cases}
\usepackage{float}
\usepackage{diagbox}
\usepackage{geometry}

\newfam\msbfam
\font\tenmsb=msbm10    \textfont\msbfam=\tenmsb
\font\sevenmsb=msbm7 \scriptfont\msbfam=\sevenmsb
\font\fivemsb=msbm5 \scriptscriptfont\msbfam=\fivemsb

\newfam\bigfam
\font\tenbig=msbm10 scaled \magstep2   \textfont\bigfam=\tenbig
\font\sevenbig=msbm7 scaled \magstep2 \scriptfont\bigfam=\sevenbig
\font\fivebig=msbm5 scaled \magstep2
\scriptscriptfont\bigfam=\fivebig

\newtheorem{theorem}{Theorem}[section]
\newtheorem{lemma}[theorem]{Lemma}
\newtheorem{corollary}[theorem]{Corollary}

\newdefinition{remark}{Remark}[section]
\newdefinition{example}{Example}[section]
\newenvironment{proof}[1][Proof]{\noindent\textbf{#1. }}{\hfill $\Box$}

\definecolor{mygray}{gray}{.85}

\numberwithin{equation}{section}

\journal{}

\begin{document}
\begin{frontmatter}

\title{Decoupled Modified Characteristic Finite Element Method with Different Subdomain Time Steps for Nonstationary Dual-Porosity-Navier-Stokes Model
\tnoteref{mytitlenote}}

\author[add1]{Luling Cao}\ead{lulingcao@163.com}
\author[add1]{Yinnian He\corref{correspondingauthor}}\ead{heyn@mail.xjtu.edu.cn}
\author[add2]{Jian Li\corref{correspondingauthor}}\ead{jiaaanli@gmail.com}
\cortext[correspondingauthor]{Corresponding author.}

\address[add1]{School of Mathematics and Statistics, Xi'an Jiaotong University, Xi'an 710049, China}
\address[add2]{Department of Mathematics, Shaanxi University of Science and Technology, Xi'an 710021, China}
~\\
~\\
~\\
~\\

\begin{abstract}
\par In this paper, we develop the numerical theory of decoupled modified characteristic finite element method with different subdomain time steps for the mixed stabilized formulation of nonstationary dual-porosity-Navier-Stokes model. Based on partitioned time-stepping methods, the system is decoupled, which means that the Navier-Stokes equations and two different Darcy equations are solved independently at each time step of subdomain. In particular, the Navier-Stokes equations are solved by the modified characteristic finite element method, which overcome the computational difficulties caused by the nonlinear term. In order to increase the efficiency, different time steps are used to different subdomains. The stability of this method is proved. In addition, we verify the optimal~$L^2$-norm error convergence order of the solutions by mathematical induction, whose proof implies the uniform~$L^{\infty}$-boundedness of the fully discrete velocity solution. Finally, some numerical tests are presented to show efficiency of the proposed method.

\end{abstract}
\begin{keyword} nonstationary dual-porosity-Navier-Stokes model; decoupled method; different subdomain time steps; mixed finite element method; modified characteristic finite element method; stability; convergence analysis
\end{keyword}

\end{frontmatter}


\section{Introduction}
Coupled free flow and porous medium flow systems play an important role in many fields. For example, the flood simulation of arid areas in geological science~\cite{discacciati2004domain}, filtration treatment in industrial production~\cite{nassehi1998modelling, hanspal2006numerical}, petroleum exploitation in mining and blood penetration between vessels and organs in life science~\cite{d2011robust}.

Usually, the system can be described by a Stokes~(Navier-Stokes)~coupled Darcy equation. There are a great deal of achievements~\cite{Yong2013Decoupling, Mu2010Decoupled, Chen2012Efficient, Girault2009DG, Cao2011Robin, Cai2009NUMERICAL, zhao2016two-grid}. However, the Darcy equation is a single porosity model, which is not accurate to deal with the complicated multiple porous media similar to naturally fractured reservoir. Actually, the naturally fractured reservoir is comprised of low permeable rock matrix blocks surrounded by an irregular network of natural microfractures. And they have different fluid storage and conductivity properties~\cite{The1963, Naturally1980, New1983}. In 2016, Hou et al. proposed and numerically solved a coupled dual-porosity-Stokes multi-physics interface system~\cite{hou2016a}~where dual-porosity equations were used to describe the multiple porous media flow. At present, the research on this model can be found in the literature~\cite{shan2019partitioned, mahbub2019coupled, mahbub2020mixed, he2020an}. To our best knowledge, up till now, there has been no research on the dual-porosity-Navier-Stokes model.

For the nonstationary dual-porosity-Navier-Stokes model, it has some features in physical and some difficulties in numerical analysis. As we all know, the fluid velocity in the free flow domain is usually much higher than that in the porous medium. Therefore, it is reasonable to apply different time steps in different subdomains. For coupled free flow and porous media flow with different subdomain time steps, see~\cite{kirby2003on, shan2013a, rybak2014a, 2017A, li2018a}~and the references therein. And for the difficulties,
it is a coupling problem. It will be considered to take direct method or decoupled method. In order to reduce the scale of solving problems and increase the reusability of software packages, we choose the decoupled method. When using partitioned time-stepping method, as a simple decoupled strategy, we have to face a difficulty that the artificial energy transfers generated by the interface time-splitting. This results in numerical instabilities~\cite{Leopold2008}. To this point,  Nitsche’s interface method is used to control the artificial energy transfers~\cite{Nitsche1971, becker2003a, burman2007a, fernandez2014explicit, mahbub2019coupled}.
On the other hand, we need to consider how to solve the nonlinear term in Navier-Stokes equation. Since the modified characteristic finite element methods could link the nonlinear term to the time term, it can greatly improve efficiency~\cite{Si2014Decoupled, Si2016Unconditional, cao2020JCAM}. The method is preferred.

In this paper, we propose and develop the numerical theory of decoupled modified characteristic finite element method with different subdomain time steps for the mixed stabilized formulation of this model. In order to hold the numerical stability, a mesh dependent stabilization term is introduced. The partitioned time-stepping method decomposes the original problem into Navier-Stokes equations and two different Darcy equations. For Navier-Stokes equations, the modified characteristic finite element method is employed to deal with the time and nonlinear terms. And the other Darcy equations are used by mixed finite element method. The stability of this method is proved. In the error analysis framework, proposed in~\cite{Si2016Unconditional, cao2020JCAM}, we prove the optimal~$L^2$-norm error convergence order by mathematical induction, whose proof implies the uniform~$L^{\infty}$-boundedness of the fully discrete velocity solution. Finally, some numerical tests are presented to show the validity of our theoretical results, especially high efficiency of the proposed method.

The rest of this paper is organized as follows: In Section 2, we introduce the nonstationary dual-porosity-Navier-Stokes model, and construct the fully discrete mixed stabilized decoupled modified characteristic scheme with different subdomain time steps. Stability of the method is presented in Section 3. In Section 4, some preliminaries and convergence analysis are shown. Section 5 reports some numerical examples, and the conclusions are given in Section 6.
\\

\section{Modified characteristic finite element method for the model problem}
\subsection{The nonstationary dual-porosity-Navier-Stokes model}
We consider a coupled dual-porosity-Navier-Stokes system on a bounded domain~$\Omega=\Omega_c \cup \Omega_d \subset \mathbb{R}^D, D=2,3$,  where~$\Omega_c$~and~$\Omega_d$~denote disjoint nonoverlapping bounded open convex regions with common boundary~$\mathbb{I}=\overline{\Omega}_c \cap \overline{\Omega}_d$, i.e. $\Omega_c \cap \Omega_d =\varnothing$. See Figure~1.

\begin{figure}[htbp]
  \centering
  \includegraphics[width=10cm]{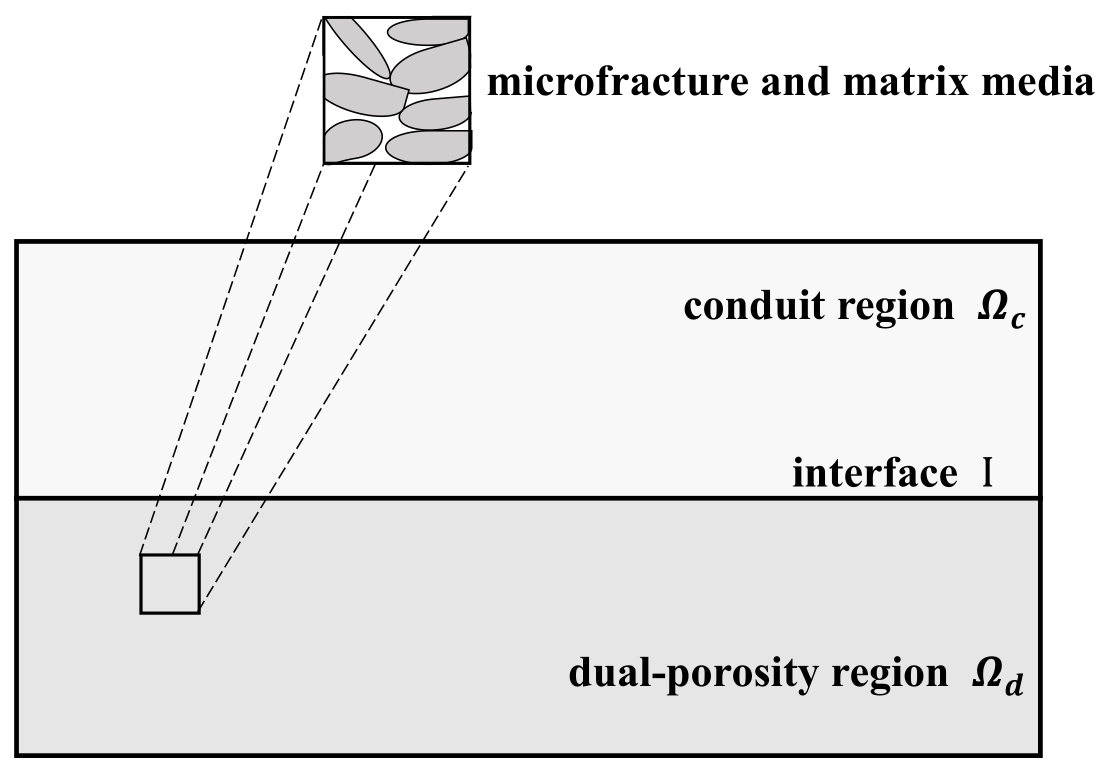}\\
  \caption{A sketch of the conduit region~$\Omega_c$, the dual-porosity region~$\Omega_d$~and the interface~$\mathbb{I}$.}
\end{figure}

In the conduit region~$\Omega_c$, let~$\boldsymbol{u}_c$~denote the fluid velocity, $p_c$~denote the kinematic pressure, $\boldsymbol{f}_c$~denote the external body force density, and $\nu >0$~is the kinematic viscosity of the fluid. The conduit flow in~$\Omega_c$~is assumed to satisfy, for~$t \in (0,T]$, the Navier-Stokes system
\begin{equation*}
\begin{split}
\frac{\partial \boldsymbol{u}_c}{\partial t}-\nu \Delta \boldsymbol{u}_c + \nabla p_c + (\boldsymbol{u}_c \cdot \nabla)\boldsymbol{u}_c &= \boldsymbol{f}_c,~~~~~~~~~~~~\text{in}~\Omega_c \times (0,T],\\
\nabla \cdot \boldsymbol{u}_c &= 0,~~~~~~~~~~~~~\text{in}~\Omega_c \times (0,T].
\end{split}
\end{equation*}

In the dual-porosity region~$\Omega_d$, let~$\boldsymbol{u}_f$~denote the velocity of microfracture flow,  $\phi_f$~denote the pressure. Accordingly, $\boldsymbol{u}_m$~denotes the velocity in matrix flow and~$\phi_m$~denotes the pressure. Then, the flow in the dual-porosity region is assumed to satisfy, for~$t \in (0,T]$, the dual-porosity system
\begin{equation*}
\begin{split}
\eta_f C_{ft} \frac{\partial \phi_f}{\partial t} + \nabla \cdot \boldsymbol{u}_f + \frac{\sigma k_m}{\mu} (\phi_f - \phi_m) &= f_d,~~~~~~~~~~~~~~~~\text{in}~\Omega_d \times (0,T],\\
\boldsymbol{u}_f &= \frac{-k_f}{\mu} \nabla \phi_f,~~~~~~~~\text{in}~\Omega_d \times (0,T],\\
\eta_m C_{mt} \frac{\partial \phi_m}{\partial t} + \nabla \cdot \boldsymbol{u}_m + \frac{\sigma k_m}{\mu}(\phi_m - \phi_f) &=0,~~~~~~~~~~~~~~~~~~\text{in}~\Omega_d \times (0,T],\\
\boldsymbol{u}_m &= \frac{-k_m}{\mu}\nabla \phi_m,~~~~~~\text{in}~\Omega_d \times (0,T],
\end{split}
\end{equation*}
where the porosity of the microfracture and matrix region are denoted by~$\eta_f$~and~$\eta_m$,  $C_{ft}$~and~$C_{mt}$~are the total compressibility for the matrix and microfractures system, $k_f$~and~$k_m$
denote the intrinsic permeability. Additionally, $\sigma$~is the shape factor characterizing the morphology and dimension of the microfractures, $\mu$~is the dynamic viscosity and~$f_d$~is a source/sink term. The term $\frac{\sigma k_m}{\mu}(\phi_m - \phi_f)$~describes the mass exchange between the matrix and the microfractures.

Along the interface~$\mathbb{I}$, there is a no-exchange situation between the matrix and conduits/microfractures:
$$\boldsymbol{u}_m \cdot \boldsymbol{n}_d = 0,$$
where~$\boldsymbol{n}_d$~denotes the unit outer normal on the interface edges from~$\Omega_d$~to~$\Omega_c$~and~$\boldsymbol{n}_c=-\boldsymbol{n}_d$.
Similar to the Navier-Stokes-Darcy model, the three well-accepted interface conditions are imposed:
\begin{itemize}
  \item The mass conservation between conduit flow and microfracture flow,
  $$\boldsymbol{u}_c \cdot \boldsymbol{n}_c + \boldsymbol{u}_f \cdot \boldsymbol{n}_d = 0.$$
  \item The balance of forces normal,
  $$\frac{\phi_f}{\rho}=p_c - \nu \boldsymbol{n}_c \cdot \nabla \boldsymbol{u}_c \cdot \boldsymbol{n}_c.$$
  \item Beavers-Joseph-Saffman (BJS) interface condition,
  $$-\nu \boldsymbol{\tau_i} \cdot \nabla \boldsymbol{u}_c \cdot \boldsymbol{n}_c =
  \frac{\alpha \nu \sqrt{D}}{\sqrt{\mathrm{trace}(\boldsymbol{\Pi})}}(\boldsymbol{u}_c \cdot \boldsymbol{\tau_i}).$$
\end{itemize}
Here~$\boldsymbol{\tau}_i (i=1,2,3,...,D-1)$~denote mutually orthogonal unit tangential vectors along the interface. In addition, $\rho$~is the density of fluid, $\alpha$~is constant parameters and~$D$~is the spatial dimension. $\boldsymbol{\Pi}=k_f \boldsymbol{\mathrm{I}}$~is the intrinsic permeability of microfractures.

For simplicity,  except on~$\mathbb{I}$, homogeneous Dirichlet boundary conditions are imposed on the boundaries~$\partial \Omega_c$~and~$\partial \Omega_d$, respectively. We have
\begin{equation*}
\begin{split}
\boldsymbol{u}_c &= 0,~~~~~~~~~~~~~\text{on}~\partial \Omega_c \backslash \mathbb{I},\\
\boldsymbol{u}_f \cdot \boldsymbol{n}_d &=0,~~~~~~~~~~~~~\text{on}~\partial \Omega_d \backslash \mathbb{I},\\
\boldsymbol{u}_m \cdot \boldsymbol{n}_d &= 0,~~~~~~~~~~~~~\text{on}~\partial \Omega_d \backslash \mathbb{I},
\end{split}
\end{equation*}
where boundaries~$\partial \Omega_c \backslash \mathbb{I}$~and~$\partial \Omega_d \backslash \mathbb{I}$~are smooth enough and Lipschitzian continuous.

Finally, initial conditions are imposed
\begin{equation*}
\begin{split}
\boldsymbol{u}_c(0,x) &= \boldsymbol{u}_{c0}(x),~~~~~~~\text{in}~\Omega_c\\
\phi_f(0,x) &= \phi_{f0}(x),~~~~~~~\text{in}~\Omega_d,\\
\phi_m(0,x) &= \phi_{m0}(x),~~~~~~\text{in}~\Omega_d.
\end{split}
\end{equation*}

\subsection{The weak formulation of model problem}
To begin with, we introduce some notations. For the Sobolev space~$W^{k,p}(\Omega_{\Lambda}),
\Lambda=c~\text{or}~d$, the integer~$k \geq 0$~and~$1 \leq p \leq \infty$. The scalar value function~$\psi \in W^{k,p}(\Omega_{\Lambda})$~is equipped with the following norms:
\begin{equation*}
\|\psi\|_{W^{k,p}}=
\begin{cases}
(\sum_{|\beta|\leq k} \int_{\Omega_{\Lambda}} |D^{\beta}\psi(x)|^p \mathrm{d}x)^{\frac{1}{p}},~~~~1 \leq p \leq \infty,\\
\sum_{|\beta|\leq k} \mathrm{ess} \sup _{x\in \Omega_{\Lambda}} |D^{\beta}\psi(x)|,~~~~p=\infty,
\end{cases}
\end{equation*}
where
\begin{equation*}
D^{\beta}=\frac{\partial^{|\beta|}}{\partial x_1^{\beta_1}...\partial x_D^{\beta_D}}.
\end{equation*}
For the multi-index~$\beta=(\beta_1,...,\beta_D)$, $\beta_i \geq 0, i=1,...,D$~and~$|\beta|=\beta_1+...+\beta_D$. The vector value function~$\boldsymbol{v}=\{v_j\} \in W^{k,p}(\Omega_{\Lambda})^D$~is equipped with the norm
\begin{equation*}
\|\boldsymbol{v}\|_{W^{k,p}}=
(\sum_{j=1}^D\|v_j\|^2_{W^{k,p}(\Omega_{\Lambda})})^{\frac{1}{2}}.
\end{equation*}
When~$p=2$, $W^{k,2}$~is denoted by~$H^k$. Define
\begin{equation*}
\|\psi\|_k := \|\psi\|_{H^k},~~~~\|\boldsymbol{v}\|_k := \|\boldsymbol{v}\|_{H^k}.
\end{equation*}
When~$k=0$, $W^{0,p}$~is denoted by~$L^p$. Especially, when~$p=2$,
\begin{equation*}
\|\psi\|_0:= \|\psi\|_{L^2},~~~~\|\boldsymbol{v}\|_0 :=\|\boldsymbol{v}\|_{L^2}.
\end{equation*}
In addition, $(\cdot,\cdot)_{\Omega_{\Lambda}}$~denotes the inner product of~$L^2(\Omega_{\Lambda})$.
Define
\begin{equation*}
H(\mathrm{div},\Omega_d):=\{\boldsymbol{v} \in L^2(\Omega_d)^D, \nabla \cdot \boldsymbol{v} \in L^2(\Omega_d)\},
\end{equation*}
and equip the space with norm
\begin{equation*}
\|\boldsymbol{v}\|_{H(\mathrm{div})}=(\|\boldsymbol{v}\|_0^2 + \|\nabla \cdot \boldsymbol{v}\|_0^2)^{\frac{1}{2}}.
\end{equation*}

Next, we recall some inequalities~\cite{brenner2007mathematical}~that are useful in the analysis.\\
$A_1$.(Poincar$\acute{\mathrm{e}}$-Friedrichs~inequality): For all~$\boldsymbol{v}_c \in Y_c$, there exists a positive constant~$C_{PF}$~which only depends on the area of~$\Omega_c$~such that
$$\|\boldsymbol{v}_c\|_0 \leq C_{PF} \|\nabla \boldsymbol{v}_c\|_0.$$
$A_2$.(Trace inequality): For all~$\boldsymbol{v}_c \in Y_c$, there exists a positive constant~$C_T$~which only depends on the area of~$\Omega_c$~such that
$$\|\boldsymbol{v}_c\|_{L^2(\mathbb{I})} \leq C_T \|\boldsymbol{v}_c\|_0^{1/2} \|\nabla \boldsymbol{v}_c\|_0^{1/2}.$$
(General trace inequality): Assume that~$\Omega_c$~is a bounded area with Lipschitz boundary and $\boldsymbol{v}_c \in H^m(\Omega_c)$. Define the trace function~$\gamma_0 \boldsymbol{v}_c, \gamma_1 \boldsymbol{v}_c,...,\gamma_{m-1} \boldsymbol{v}_c$~on~$\partial \Omega_c$, where~$\gamma_j(0 \leq j \leq m-1)$~is a linear continuous map from~$H^m(\Omega_c)$~to~$H^{m-j-\frac{1}{2}}(\partial \Omega_c)$. There exists a constant~$C_T$~which only depends on~$\Omega_c$, such that
\begin{equation*}
\|\boldsymbol{v}_c\|_{H^{m-j-\frac{1}{2}}(\partial \Omega_c)} \leq C_T \|\boldsymbol{v}_c\|_{m}.
\end{equation*}
$A_3$.(Properties of~$H(\mathrm{div})$~space) For all~$\boldsymbol{v}_c \in H(\mathrm{div},\Omega_c)$~and~$\boldsymbol{v}_c \cdot \boldsymbol{n}_d \in H^{-1/2}(\partial \Omega_c)$, there exists a positive constant~$C_{div}$~satisfying
$$\|\boldsymbol{v}_c \cdot \boldsymbol{n}_d\|_{H^{-\frac{1}{2}}(\partial \Omega_c)} \leq C_{div} \|\boldsymbol{v}_c\|_{H(\mathrm{div})}.$$
$A_4$.(Sobolev~interpolation inequality)
\begin{equation*}
\begin{split}
&\|\psi\|_{L^q} \leq C_S \|\psi\|_1,~~~~q \leq 6,\\
&\|\psi\|_{L^{\infty}} \leq C_s \|\psi\|_{W^{1,q}},~~~~q > d^{\ast},\\
&\|\psi\|_{L^q} \leq C_q \|\psi\|_0^{\beta} \|\psi\|_1^{1-\beta},~~~~2 \leq q \leq 6, \beta=\frac{6-q}{2q}.
\end{split}
\end{equation*}

In order to deduce the weak formulation of~dual-porosity-Navier-Stokes~equations, we introduce the following spaces~$(Y_c,Q_c;Y_f,Q_f;Y_m,Q_m)$:
\begin{equation*}
\begin{split}
Y_c &:= \{\boldsymbol{v}_c \in H^1(\Omega_c)^D : \boldsymbol{v}_c=0~~\text{on}~\partial \Omega_c \backslash \mathbb{I}\},\\
Q_c &:= L_0^2(\Omega_c) := \{q \in L^2(\Omega_c): \int_{\Omega_c} q \mathrm{d}x=0 \},\\
Y_f &:= \{\boldsymbol{v}_f \in H(\mathrm{div},\Omega_d): \boldsymbol{v}_f \cdot \boldsymbol{n}_d =0~~\text{on}~\partial \Omega_d \backslash \mathbb{I}\},\\
Q_f &:= L_0^2(\Omega_d) := \{\psi_f \in L^2(\Omega_d): \int_{\Omega_d} \psi_f \mathrm{d}x=0\},\\
Y_m &:= \{\boldsymbol{v}_m \in H(\mathrm{div},\Omega_d): \boldsymbol{v}_m \cdot \boldsymbol{n}_d =0~~\text{on}~\partial \Omega_d \},\\
Q_m &:= L_0^2(\Omega_d) := \{\psi_m \in L^2(\Omega_d): \int_{\Omega_d} \psi_m \mathrm{d}x=0\}.
\end{split}
\end{equation*}
For simplicity, define the product space as follows:
$$\mathcal{X}:=Y_c \times Q_c \times Y_f \times Q_f \times Y_m \times Q_m.$$
Furthermore, define the involving time Sobolev space
\begin{equation*}
\mathcal{X}_T :=L^2(0,T;\mathcal{X}).
\end{equation*}

In weak formulation, the interface term~$\frac{1}{\rho} \int_{\mathbb{I}} \phi_f(\boldsymbol{v}_c-\boldsymbol{v}_f)\cdot \boldsymbol{n}_d \mathrm{d}s$~is difficult to control in numerical calculation. According to~\cite{burman2007a, mahbub2019coupled}, in order to overcome this difficulty of numerical instability which generated by the interface time-splitting, take the Nitsche's interface method and introduce a mesh dependent stabilization term~$\frac{\gamma}{\rho h} \int_{\mathbb{I}} (\boldsymbol{u}_c-\boldsymbol{u}_f) \cdot \boldsymbol{n}_d (\boldsymbol{v}_c-\boldsymbol{v}_f) \cdot \boldsymbol{n}_d \mathrm{d}s$, where~$\gamma >0$~is a penalty parameter and have been nondimensionalize. Based on the second interface condition, the stabilization term is zero in the continuous sense. It conforms that the finite element scheme is well-posedness. We will consider the effect of penalty parameter on the scheme in numerical experiment.

The~dual-porosity-Navier-Stokes~weak formulation is as follows. Assuming that ~$\boldsymbol{f}_c \in L^2(0,T;H^{-1}(\Omega_c)^{D}),$ $f_d \in L^2(0,T;L^2(\Omega_d))$, for all
$(\boldsymbol{v}_c,q;\boldsymbol{v}_f,\psi_f;\boldsymbol{v}_m,\psi_m) \in \mathcal{X}$, find~$(\boldsymbol{u}_c,p_c;\boldsymbol{u}_f,$\\
$\phi_f;\boldsymbol{u}_m,\phi_m) \in \mathcal{X}_T$~satisfying
\begin{equation}\label{yuanshi}
\begin{split}
&\bigg(\frac{\partial \boldsymbol{u}_c}{\partial t},\boldsymbol{v}_c\bigg)_{\Omega_c}
+\frac{\eta_d C_{dt}}{\rho}\bigg(\frac{\partial \phi_d}{\partial t},\psi_d\bigg)_{\Omega_d}
+a_{\Omega_c}(\boldsymbol{u}_c,\boldsymbol{v}_c)
+((\boldsymbol{u}_c \cdot \nabla) \boldsymbol{u}_c,\boldsymbol{v}_c)_{\Omega_c}
-b(\boldsymbol{v_c},p_c)+b(\boldsymbol{u}_c,q)\\
&+a_{\phi_d}(\phi_d,\psi_d)
+a_{\boldsymbol{u}_d}(\boldsymbol{u}_d,\boldsymbol{v}_d)
+b_{\phi_d}(\boldsymbol{u}_d,\psi_d)
-b_{\phi_d}(\boldsymbol{v}_d,\phi_d)\\
&-\frac{1}{\rho}\int_{\mathbb{I}} \phi_f (\boldsymbol{v}_c - \boldsymbol{v}_f) \cdot \boldsymbol{n}_d \mathrm{d}s
+\frac{\gamma}{\rho h}\int_{\mathbb{I}}(\boldsymbol{u}_c-\boldsymbol{u}_f)\cdot \boldsymbol{n}_d (\boldsymbol{v}_c-\boldsymbol{v}_f)\cdot \boldsymbol{n}_d \mathrm{d}s\\
&=(\boldsymbol{f}_c,\boldsymbol{v}_c)_{\Omega_c}
+\frac{1}{\rho}(f_d,\psi_f)_{\Omega_d},
\end{split}
\end{equation}
where
\begin{equation*}
\begin{split}
&[\boldsymbol{u}_f,\boldsymbol{u}_m]^T=\boldsymbol{u}_d,~~~~
[\boldsymbol{v}_f,\boldsymbol{v}_m]^T=\boldsymbol{v}_d,~~~~
[\phi_f,\phi_m]^T=\phi_d,~~~~
[\psi_f,\psi_m]^T=\psi_d,\\
&[\eta_f,\eta_m]^T=\eta_d,~~~~
[C_{ft},C_{mt}]^T=C_{dt},\\
&\frac{\eta_d C_{dt}}{\rho}(\frac{\partial \phi_d}{\partial t},\psi_d)_{\Omega_d}
=\frac{\eta_f C_{ft}}{\rho}(\frac{\partial \phi_f}{\partial t},\psi_f)_{\Omega_d}
+\frac{\eta_m C_{mt}}{\rho}(\frac{\partial \phi_m}{\partial t},\psi_m)_{\Omega_d},\\
&a_{\Omega_c}(\boldsymbol{u}_c,\boldsymbol{v}_c)=\nu (\nabla \boldsymbol{u}_c, \nabla \boldsymbol{v}_c)_{\Omega_c}
+\sum_{i=1}^{D-1}\int_{\mathbb{I}}\frac{\alpha \nu \sqrt{D}}{\sqrt{\mathrm{trace}(\Pi)}} (\boldsymbol{u}_c \cdot \boldsymbol{\tau}_i) (\boldsymbol{v}_c \cdot \boldsymbol{\tau}_i) \mathrm{d}s,\\
&a_{\phi_d}(\phi_d,\psi_d)=\frac{\sigma k_m}{\rho \mu}(\phi_f-\phi_m,\psi_f)_{\Omega_d}
+\frac{\sigma k_m}{\rho \mu}(\phi_m-\phi_f,\psi_m)_{\Omega_d},\\
&a_{\boldsymbol{u}_d}(\boldsymbol{u}_d,\boldsymbol{v}_d)=
\frac{1}{\rho}(\mu k_f^{-1} \boldsymbol{u}_f,\boldsymbol{v}_f)_{\Omega_d}
+\frac{1}{\rho}(\mu k_m^{-1} \boldsymbol{u}_m,\boldsymbol{v}_m)_{\Omega_d},\\
&b_{\phi_d}(\boldsymbol{u}_d,\psi_d)=\frac{1}{\rho}(\nabla \cdot \boldsymbol{u}_f,\psi_f)_{\Omega_d}
+\frac{1}{\rho}(\nabla \cdot \boldsymbol{u}_m,\psi_m)_{\Omega_d},\\
&b(\boldsymbol{v}_c,p_c)=(p_c,\nabla \cdot \boldsymbol{v}_c)_{\Omega_c}.
\end{split}
\end{equation*}

As we all know, the velocity in the conduit flow is faster than that in the dual-porosity flow, therefore it is reasonable to apply different time steps in different subdomain. For~$\Omega_c$, divide the time interval~$[0,T]$~into~$N >0$~averagely. For the segments~$[t_n,t_{n+1}](n=0,1,...,N-1)$~satisfying
\begin{equation*}
0=t_0 \leq t_1 \leq ... \leq t_{N-1} \leq t_N=T,~~~t_n=n\Delta t,
\end{equation*}
and the time step is~$\Delta t=\frac{T}{N}$.

\begin{figure}[htbp]
  \centering
  \includegraphics[width=10cm]{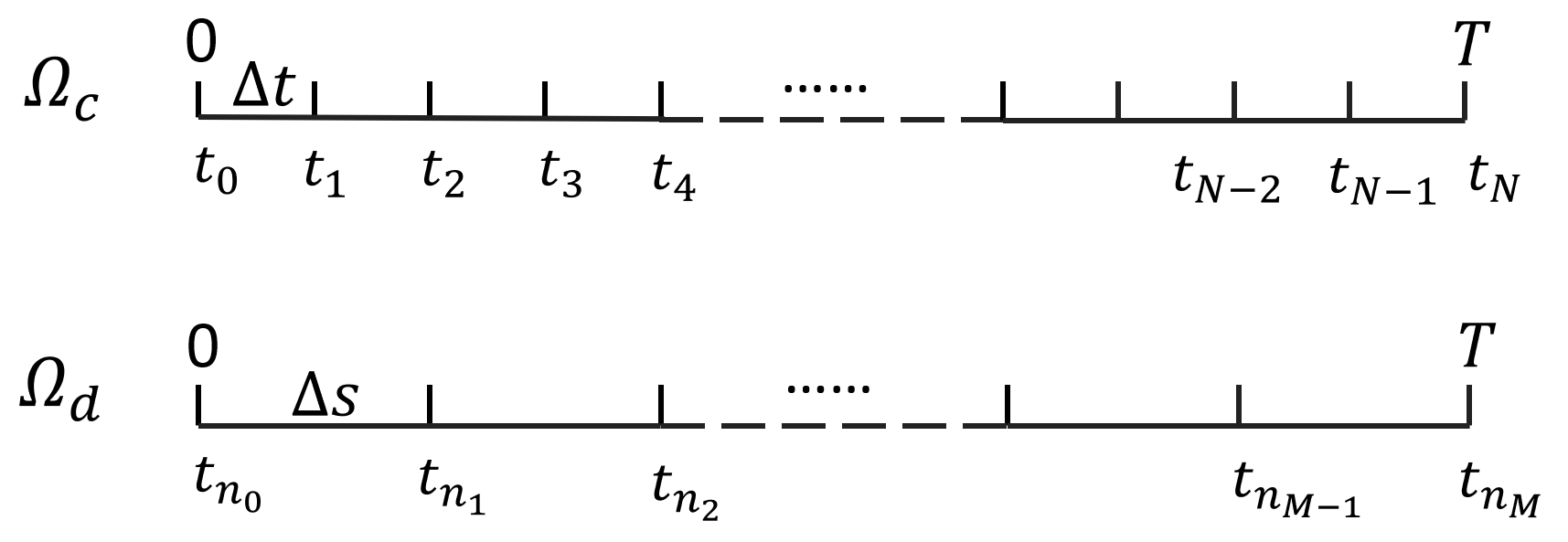}\\
  \caption{Different time steps in different subdomian}
\end{figure}

For~$\Omega_d$, divide the time interval~$[0,T]$~into~$M >0$~averagely. For the segments~$[t_{n_k},t_{n_{k+1}}](k=0,1,...,M-1)$~satisfying
\begin{equation*}
0=t_{n_0} \leq t_{n_1} \leq ... \leq t_{n_{M-1}} \leq t_{n_M}=T,~~~~t_{n_k}=k \Delta s,
\end{equation*}
and the time step is~$\Delta s=\frac{T}{M}$.
Note that time level in~$\Omega_c$~and that in~$\Omega_d$~are nested, i.e.~$N=rM$, where integer time step ratio~$r=\frac{\Delta s}{\Delta t}$.

Following~\cite{Si2016Unconditional, cao2020JCAM}, we have
\begin{equation*}
\bigg(\frac{\partial \boldsymbol{u}_c}{\partial t} + \boldsymbol{u}_c \cdot \nabla \boldsymbol{u}_c \bigg)\bigg|_{t=t_{n+1}}\approx \frac{\boldsymbol{u}_c(t_{n+1})-\bar{\boldsymbol{u}}_c(t_n)}{\Delta t},
\end{equation*}
where~$\bar{\boldsymbol{u}}_c(t_n)=\boldsymbol{u}_c(\bar{x},t_n), \bar{x}=x-\boldsymbol{u}_c(t_n)\Delta t$.

Therefore, we obtain the characteristic version of variational formulation.\\
\textbf{A. The characteristic version of variational formulation}:
\begin{equation}\label{CBF}
\begin{split}
&(d_t \boldsymbol{u}_c(t_{n+1}),\boldsymbol{v}_c)_{\Omega_c}
+\frac{\eta_d C_{dt}}{\rho}(d_t \phi_d(t_{n+1}),\psi_d)_{\Omega_d}
+a_{\Omega_c}(\boldsymbol{u}_c(t_{n+1}),\boldsymbol{v}_c)
-b(\boldsymbol{v}_c,p_c(t_{n+1}))\\
&+b(\boldsymbol{u}_c(t_{n+1}),q)
+a_{\phi_d}(\phi_d(t_{n+1}),\psi_d)
+a_{\boldsymbol{u}_d}(\boldsymbol{u}_d(t_{n+1}),\boldsymbol{v}_d)
+b_{\phi_d}(\boldsymbol{u}_d(t_{n+1}),\psi_d)\\
&-b_{\phi_d}(\boldsymbol{v}_d,\phi_d(t_{n+1}))
-\frac{1}{\rho}\int_{\mathbb{I}}\phi_f(t_{n+1})(\boldsymbol{v}_c-\boldsymbol{v}_f)\cdot \boldsymbol{n}_d \mathrm{d}s\\
&+\frac{\gamma}{\rho h}\int_{\mathbb{I}}(\boldsymbol{u}_c(t_{n+1})-\boldsymbol{u}_f(t_{n+1}))\cdot \boldsymbol{n}_d (\boldsymbol{v}_c-\boldsymbol{v}_f)\cdot \boldsymbol{n}_d \mathrm{d}s\\
&=(\boldsymbol{f}_c(t_{n+1}),\boldsymbol{v}_c)_{\Omega_c}
+\frac{1}{\rho}(f_{d}(t_{n+1}),\psi_f)_{\Omega_d}
-\bigg(\frac{\boldsymbol{u}_c(t_n)-\bar{\boldsymbol{u}}_c(t_n)}{\Delta t},\boldsymbol{v}_c\bigg)_{\Omega_c}
+(R_{tr}^{n+1},\boldsymbol{v}_c)_{\Omega_c}\\
&+\frac{\eta_d C_{dt}}{\rho}\bigg(d_t \phi_d(t_{n+1})- \frac{\partial \phi_d}{\partial t}(t_{n+1}),\psi_d\bigg)_{\Omega_d},
\end{split}
\end{equation}
where~$d_t \boldsymbol{u}_c(t_{n+1})=\frac{\boldsymbol{u}_c(t_{n+1})-\boldsymbol{u}_c(t_n)}{\Delta t},
R_{tr}^{n+1}=\frac{\boldsymbol{u}_c(t_{n+1})-\bar{\boldsymbol{u}}_c(t_n)}{\Delta t}
-\frac{\partial \boldsymbol{u}_c}{\partial t}(t_{n+1})-(\boldsymbol{u}_c(t_{n+1}) \cdot \nabla )\boldsymbol{u}_c(t_{n+1})$.
\\

We also introduce the discrete~Gronwall~lemma~\cite{heywood1990finite}.
\begin{lemma} Assume that~$E \geq 0$, for any integer~$M \geq 0$, $\kappa_m, A_m, B_m, C_m \geq 0$~satisfying
\begin{equation*}
A_M + \Delta t \sum_{m=0}^M B_m \leq \Delta t \sum_{m=0}^M \kappa_m A_m + \Delta t \sum_{m=0}^M C_m +E.
\end{equation*}
For all~$m$, assume that
\begin{equation*}
\kappa_m \Delta t < 1,
\end{equation*}
and set~$g_m=(1-\kappa_m\Delta t)^{-1}$, then
\begin{equation*}
A_M + \Delta t\sum_{m=0}^M B_m \leq \exp(\Delta t \sum_{m=0}^M g_m \kappa_m)(\Delta t \sum_{m=0}^M C_m + E).
\end{equation*}
\end{lemma}

\subsection{The discrete formulation of model problem}
Let~$T_h$~be a uniform simplex partition of~$\overline{\Omega}_c \cup \overline{\Omega}_d$, and~$h:=\{\max_{K \in \mathcal{T}_h}h_K: h_K=diam(K)\}$. $T_h^c$~and~$T_h^d$~denote the partition of subdomain~$\Omega_c$~and~$\Omega_d$, respectively. Furthermore, the partition matches along the interface~$\mathbb{I}$, that is to say, there is no hanging point on~$\mathbb{I}$. If~$D=2$, adjacent elements share the same edge; If~$D=3$, adjacent elements share the same face.

Choose the finite element space~$(Y_c^h,Q_c^h)$~for Navier-Stokes model satisfying the velocity-pressure inf-sup condition: there exists~$\chi_c$~independent of~$h$, such that
\begin{equation*}
\begin{split}
Y_c^h \subset Y_c,Q_c^h \subset Q_c &,\\
\inf_{0 \neq q^h \in Q_c^h} \sup_{0 \neq \boldsymbol{v}_c^h \in Y_c^h} \frac{(q^h,\nabla \cdot \boldsymbol{v}_c^h)_{\Omega_c}}{\|\nabla \boldsymbol{v}_c^h\|_0 \|q^h\|_0} \geq \chi_c &.
\end{split}
\end{equation*}
Define a discrete divergence-free velocity space
$$V_c^h=\{\boldsymbol{v}_c \in Y_c^h: (q,\nabla \cdot \boldsymbol{v}_c)_{\Omega_c}=0, \forall q \in Q_c^h \}.$$

For dual-porosity model, choose the finite element space~$(Y_f^h,Q_f^h)$~also satisfying the velocity-pressure inf-sup condition: there exists a positive constant~$\chi_f$, for all~$\psi_f^h \in Q_f^h$, we have
\begin{equation*}
\begin{split}
Y_f^h \subset Y_f, Q_f^h \subset Q_f,&\\
\sup_{0\neq \boldsymbol{v}_f \in Y_f^h} \frac{(\psi_f^h,\nabla \cdot \boldsymbol{v}_f^h)_{\Omega_d}}{\|\boldsymbol{v}_f^h\|_{H(\mathrm{div},\Omega_d)}} \geq \chi_f \|\psi_f^h\|_0.&
\end{split}
\end{equation*}
Similarly, for all~$\psi_m^h \in Q_m^h$, there exists a positive constant~$\chi_m$, we have
\begin{equation*}
\begin{split}
Y_m^h \subset Y_m, Q_m^h \subset Q_m,&\\
\sup_{0 \neq \boldsymbol{v}_m \in Y_m^h} \frac{(\psi_m^h,\nabla \cdot \boldsymbol{v}_m^h)_{\Omega_d}}{\|\boldsymbol{v}_m^h\|_{H(\mathrm{div},\Omega_d)}} \geq \chi_m \|\psi_m^h\|_0.&
\end{split}
\end{equation*}

In the next place, we introduce some inequalities~\cite{brenner2007mathematical} and lemmas~\cite{chen2002error} that may be used in discrete spaces.\\
$B_1.$(The inverse inequality) When~$1 \leq p,q \leq \infty, 0 \leq l \leq k$,
\begin{equation}\label{inverse2}
\|\boldsymbol{u}_c^h\|_{W^{k,p}} \leq C_{inv} h^{-\max\{0,\frac{d}{q}-\frac{d}{p}\}}h^{l-k} \|\boldsymbol{u}_c^h\|_{W^{l,q}},~~~~\forall \boldsymbol{u}_c^h \in Y_c^h.
\end{equation}
$B_2.$(The trace inverse inequality) For all~$\psi_f^h \in Q_f^h$, we have
\begin{equation}\label{inverseTrace}
\|\psi_f^h\|_{L^2(\mathbb{I})} \leq \tilde{C}_{inv} h^{-1/2} \|\psi_f^h\|_0,~~~~\forall \psi_f^h \in Q_f^h.
\end{equation}
$B_3.$(The discrete Sobolev inequality~\cite{achdou2000convergence,brenner2007mathematical}) There exists~$C_{DS} > 0$~such that for all~$\boldsymbol{u}_c^h \in Y_c^h$, the following inequalities hold:
\begin{equation}\label{DS}
\begin{split}
\|\boldsymbol{u}_c^h\|_{0,\infty} &\leq C_{DS} (1+|\ln(h)|)^{1/2} \|\boldsymbol{u}_c^h\|_{1},~~~~\text{in}~d=2,\\
\|\boldsymbol{u}_c^h\|_{0,\infty} &\leq C_{DS} h^{-\frac{1}{2}} \|\boldsymbol{u}_c^h\|_{1},~~~~\text{in}~d=3.
\end{split}
\end{equation}

\begin{lemma}[\cite{chen2002error, zhang2010a}]\label{LEM}
Assume that~$\boldsymbol{u}_c^n \in W^{1,\infty}$, for all~$\boldsymbol{u}_c^n \in H_0^1(\Omega_c)$, we have
\begin{equation*}
(\hat{\boldsymbol{u}}_c^n,\hat{\boldsymbol{u}}_c^n)_{\Omega_c}-(\boldsymbol{u}_c^n,\boldsymbol{u}_c^n)_{\Omega_c}
\leq \hat{C} \Delta t (\boldsymbol{u}_c^n,\boldsymbol{u}_c^n)_{\Omega_c},
\end{equation*}
where~$\hat{\boldsymbol{u}}_c^n=\boldsymbol{u}_c^n(x-\boldsymbol{u}_c^n\Delta t)$~and~$\hat{C}$~is a constant which is independent of the spatial and temporal grid sizes~$h$~and~$\Delta t$.
\end{lemma}

For simplicity, for~$t_{n+1}, t_{n_{k+1}} \in [0,T]$, $n_{k+1}=k r$, we use~$(\boldsymbol{u}_c^{n+1},\boldsymbol{p}^{n+1};\boldsymbol{u}_f^{n_{k+1}},\phi_f^{n_{k+1}};\boldsymbol{u}_m^{n_{k+1}},$\\
$\phi_m^{n_{k+1}})$ denotes~$(\boldsymbol{u}_c^{h,n+1},\boldsymbol{p}^{h,n+1};\boldsymbol{u}_f^{h,n_{k+1}},\phi_f^{h,n_{k+1}};\boldsymbol{u}_m^{h,n_{k+1}},\phi_m^{h,n_{k+1}})$. Define the following product space
$$\mathcal{X}^h := Y_c^h \times Q_c^h \times Y_f^h \times Q_f^h \times Y_m^h \times Q_m^h \subset \mathcal{X}.$$
And define the~$L^2$~bounded linear projection operator~$P^h: \mathcal{X}\rightarrow \mathcal{X}^h$.

For the formulation~A, we take the mixed finite element method for space, the backward-Euler discretization in time and the decoupled approach by partition time-stepping method. These yield that\\
~\\
\textbf{B. The fully discrete decoupled modified characteristic scheme}\\
Given~$\boldsymbol{u}_c^0=P^h \boldsymbol{u}_{c0}, \phi_m^0 = P^h \phi_{m0}, \phi_f^0 =P^h \phi_{f0}$, for all~$(\boldsymbol{v}_c^h,q^h;\boldsymbol{v}_f^h,\psi_f^h;\boldsymbol{v}_m^h,\psi_m^h) \in \mathcal{X}^h$, find $(\boldsymbol{u}_c^{n+1},p_c^{n+1};\boldsymbol{u}_f^{n_{k+1}},\phi_f^{n_{k+1}};\boldsymbol{u}_m^{n_{k+1}},\phi_m^{n_{k+1}}) \in \mathcal{X}^h, n=n_k, n_k+1,...,n_{k+1}-1$~such that the following formulations are established.
\begin{itemize}
  \item Step~1\\
  \begin{align}
  &(d_t \boldsymbol{u}_c^{n+1},\boldsymbol{v}_c^h)_{\Omega_c}
  +a_{\Omega_c}(\boldsymbol{u}_c^{n+1},\boldsymbol{v}_c^h)
  -b(\boldsymbol{v}_c^h,p_c^{n+1})
  +b(\boldsymbol{u}_c^{n+1},q^h)
  -\frac{1}{\rho}\int_{\mathbb{I}}\phi_f^{n_k} \boldsymbol{v}_c^h \cdot \boldsymbol{n}_d\mathrm{d}s \nonumber \\
  &+\frac{\gamma}{\rho h}\int_{\mathbb{I}} ((\boldsymbol{u}_c^{n+1}-\boldsymbol{u}_f^{n_k})\cdot \boldsymbol{n}_d)(\boldsymbol{v}_c^h \cdot \boldsymbol{n}_d) \mathrm{d}s
  =(\boldsymbol{f}_c(t_{n+1}),\boldsymbol{v}_c^h)_{\Omega_c}
  -\bigg(\frac{\boldsymbol{u}_c^n-\hat{\boldsymbol{u}}_c^n}{\Delta t},\boldsymbol{v}_c^h\bigg)_{\Omega_c},\label{Step1}
  \end{align}
  where~$d_t \boldsymbol{u}_c^{n+1}=\frac{\boldsymbol{u}_c^{n+1}-\boldsymbol{u}_c^n}{\Delta t}, \hat{\boldsymbol{u}}_c^n=\boldsymbol{u}_c^n(\hat{x}), \hat{x}=x-\boldsymbol{u}_c^n\Delta t$,
  and~$\Delta t$~is a small step in~$\Omega_c$,
  \begin{equation*}
  \begin{split}
  a_{\Omega_c}(\boldsymbol{u}_c^{n+1},\boldsymbol{v}_c^h)
  =\nu (\nabla \boldsymbol{u}_c^{n+1},\nabla \boldsymbol{v}_c^h)_{\Omega_c}
  +\sum_{i=1}^{D-1} \int_{\mathbb{I}}\frac{\alpha \nu \sqrt{D}}{\sqrt{\mathrm{trace}(\Pi)}}(\boldsymbol{u}_c^{n+1} \cdot \boldsymbol{\tau}_i)(\boldsymbol{v}_c^h \cdot \boldsymbol{\tau}_i)\mathrm{d}s.
  \end{split}
  \end{equation*}
  \item Step~2\\
  \begin{align}
  &\frac{\eta_m C_{mt}}{\rho}(d_s \phi_m^{n_{k+1}},\psi_m^h)_{\Omega_d}
  + \frac{1}{\rho}(\nabla \cdot \boldsymbol{u}_m^{n_{k+1}},\psi_m^h)_{\Omega_d}
  + \frac{1}{\rho}(\mu k_m^{-1} \boldsymbol{u}_m^{n_{k+1}},\boldsymbol{v}_m^h)_{\Omega_d}\nonumber \\
  &-\frac{1}{\rho}(\phi_m^{n_{k+1}},\nabla \cdot \boldsymbol{v}_m^h)_{\Omega_d}
  + \frac{\sigma k_m}{\rho \mu}(\phi_m^{n_{k+1}}-\phi_f^{n_k},\psi_m^h)_{\Omega_d}=0,\label{Step2}
  \end{align}
  where~$d_s \phi_m^{n_{k+1}}=\frac{\phi_m^{n_{k+1}}-\phi_m^{n_k}}{\Delta s}$~and~$\Delta s$~is a large step in~$\Omega_d$.\\

  \item Step~3\\
  \begin{align}
  &\frac{\eta_f C_{ft}}{\rho}(d_s \phi_f^{n_{k+1}},\psi_f^h)_{\Omega_d}
  +\frac{1}{\rho}(\nabla \cdot \boldsymbol{u}_f^{n_{k+1}},\psi_f^h)_{\Omega_d}
  +\frac{1}{\rho}(\mu k_f^{-1} \boldsymbol{u}_f^{n_{k+1}},\boldsymbol{v}_f^h)_{\Omega_d} \nonumber\\
  &-\frac{1}{\rho}(\phi_f^{n_{k+1}},\nabla \cdot \boldsymbol{v}_f^h)_{\Omega_d}
  +\frac{\sigma k_m}{\rho \mu}(\phi_f^{n_{k+1}}-\phi_m^{n_k},\psi_f^h)_{\Omega_d}
  +\frac{1}{\rho}\int_{\mathbb{I}} \phi_f^{n_k} \boldsymbol{v}_f^h \cdot \boldsymbol{n}_d \mathrm{d}s \nonumber\\
  &-\frac{\gamma}{\rho h}\int_{\mathbb{I}}((S^{n_{k+1}}-\boldsymbol{u}_f^{n_{k+1}})\cdot \boldsymbol{n}_d)(\boldsymbol{v}_f^h\cdot \boldsymbol{n}_d)\mathrm{d}s
  = \frac{1}{\rho}(f_d(t_{n_{k+1}}),\psi_f^h)_{\Omega_d},\label{Step3}
  \end{align}
  where~$d_s \phi_f^{n_{k+1}}=\frac{\phi_f^{n_{k+1}}-\phi_f^{n_k}}{\Delta s}$~and $S^{n_{k+1}}=\frac{1}{r}\sum_{n=n_k}^{n_{k+1}-1} \boldsymbol{u}_c^{n+1}$.\\
\end{itemize}

\begin{remark}
Note that Step 1 and Step 2 can be calculated at the same time. In the following numerical experiments, we take parallel algorithm for Step 1 and Step 2. By this means, we can improve the computing efficiency.
\end{remark}

\section{Stability of the method}
Hereafter, $C > 0$~denotes a generic constant whose value may be different from place to place, but which is independent of the spatial and temporal grid sizes~$h$~and~$\Delta t$, respectively.
\begin{theorem}
If~$\boldsymbol{f}_c \in L^2(0,T;H^{-1}(\Omega_c)), f_d \in L^2(0,T;L^2(\Omega_d))$. Assume that~$\hat{C} \Delta t <1$, we get the stability for the fluid velocity in the first large time interval~$[0,t_{n_1}]$, for any~$0 \leq J \leq r-1$
\begin{equation*}
\begin{split}
&\|\boldsymbol{u}_c^{J+1}\|_0^2
+\frac{\nu \Delta t}{2} \sum_{n=0}^{J}\|\nabla \boldsymbol{u}_c^{n+1}\|_0^2
+\frac{2\alpha \nu \Delta t}{\sqrt{k}}\sum_{n=0}^{J}\|P_{\tau}(\boldsymbol{u}_c^{n+1})\|_{L^2(\mathbb{I})}^2\\
&+\frac{\gamma \Delta t}{\rho h}\sum_{n=0}^{J}\|(\boldsymbol{u}_c^{n+1}-\boldsymbol{u}_f^{n_k})\cdot \boldsymbol{n}_d\|_{L^2(\mathbb{I})}^2
+\frac{\gamma \Delta t}{\rho h}\sum_{n=0}^{J} \|\boldsymbol{u}_c^{n+1} \cdot \boldsymbol{n}_d\|_{L^2(\mathbb{I})}^2\\
&\leq C\bigg(\frac{2 C_{PF}^2 \Delta t}{\nu}\sum_{n=0}^{J}\|\boldsymbol{f}_c(t_{n+1})\|_{H^{-1}}^2
+\frac{C_T^2 C_{inv}^2 C_{PF}\Delta t}{\rho^2 \nu h}\sum_{n=0}^{J} \|\phi_f^{0}\|_0^2
+\frac{\gamma \Delta t}{\rho h}\sum_{n=0}^{J}\|\boldsymbol{u}_f^{0} \cdot \boldsymbol{n}_d\|_{L^2(\mathbb{I})}^2
+\|\boldsymbol{u}_c^{0}\|_0^2 \bigg).
\end{split}
\end{equation*}
On the other hand, under the condition
$$\gamma \geq \frac{4 C_{inv}^2}{\eta_f C_{ft}}\bigg( \frac{\Delta t}{1-\Delta t -2 \hat{C} \Delta t }\bigg),$$
we obtain the stability for the time interval~$[0,t_{n_M}]$
\begin{align}
&\|\boldsymbol{u}_c^{n_{M}}\|_0^2 + \nu \Delta t \sum_{k=0}^{M-1}\sum_{n=n_k}^{n_{k+1}-1}\|\nabla \boldsymbol{u}_c^{n+1}\|_0^2
+\frac{2\alpha \nu \Delta t}{\sqrt{k}}\sum_{k=0}^{M-1} \sum_{n=n_k}^{n_{k+1}-1}\|P_{\tau} (\boldsymbol{u}_c^{n+1})\|_{L^2(\mathbb{I})}^2
+\frac{2\mu \Delta s}{\rho k_m}\sum_{k=0}^{M-1}\|\boldsymbol{u}_m^{n_{k+1}}\|_0^2 \nonumber \\
&+\frac{\gamma \Delta t}{2\rho h}\sum_{k=0}^{M-1} \sum_{n=n_k}^{n_{k+1}-1}\|(\boldsymbol{u}_c^{n+1}-\boldsymbol{u}_f^{n_{k+1}})\cdot \boldsymbol{n}_d\|_{L^2(\mathbb{I})}^2
+\frac{\gamma \Delta s}{\rho h} \|\boldsymbol{u}_f^{n_{M}} \cdot \boldsymbol{n}_d\|_{L^2(\mathbb{I})}^2
+\frac{2\mu \Delta s}{\rho k_f}\sum_{k=0}^{M-1}\|\boldsymbol{u}_f^{n_{k+1}}\|_0^2 \nonumber \\
&+\frac{\eta_m C_{mt}}{\rho}\bigg(\|\phi_m^{n_{M}}\|_0^2 + \sum_{k=0}^{M-1} \|\phi_m^{n_{k+1}}-\phi_m^{n_k}\|_0^2 \bigg)
+\frac{\eta_f C_{ft}}{\rho}\bigg(\|\phi_f^{n_{M}}\|_0^2 + \sum_{k=0}^{M-1} \|\phi_f^{n_{k+1}}-\phi_f^{n_k}\|_0^2 \bigg) \nonumber \\
&+\frac{\sigma k_m \Delta s}{\rho \mu}\bigg( \|\phi_m^{n_M}\|_0^2 + \sum_{k=0}^{M-1} \|\phi_m^{n_{k+1}}-\phi_f^{n_k}\|_0^2
+\|\phi_f^{n_M}\|_0^2
+\sum_{k=0}^{M-1} \|\phi_f^{n_{k+1}}-\phi_m^{n_k}\|_0^2 \bigg) \\
&\leq  C \bigg( \frac{C_{PF}^2 \Delta t}{\nu}\sum_{k=0}^{M-1} \sum_{n=n_k}^{n_{k+1}-1}\|\boldsymbol{f}_c(t_{n+1})\|_{H^{-1}}^2
+ \frac{2\Delta t}{\rho \eta_f C_{ft}}\sum_{k=0}^{M-1}\sum_{n=n_k}^{n_{k+1}-1}\|f_d(t_{n_{k+1}})\|_0^2
+ \frac{\gamma \Delta s}{\rho h} \|\boldsymbol{u}_f^0 \cdot \boldsymbol{n}_d\|_{L^2(\mathbb{I})}^2 \nonumber\\
&+ \|\boldsymbol{u}_c^0\|_0^2
+ \frac{\eta_m C_{mt}}{\rho}\|\phi_m^0\|_0^2
+ \frac{\eta_f C_{ft}}{\rho} \|\phi_f^0\|_0^2
+ \frac{\sigma k_m \Delta s}{\rho \mu} (\|\phi_m^0\|_0^2 + \|\phi_f^0\|_0^2 ) \bigg). \nonumber
\end{align}
\end{theorem}

~\\
\begin{proof}
Taking~$\boldsymbol{v}_c^h=2 \Delta t \boldsymbol{u}_c^{n+1}$~and~$q^h=2 \Delta t p_c^{n+1}$~in~\eqref{Step1}, sum over~$n=n_k,...,n_{k+1}-1$,
\begin{equation}\label{WStep1}
\begin{split}
&2\nu \Delta t \sum_{n=n_k}^{n_{k+1}-1}\|\nabla \boldsymbol{u}_c^{n+1}\|_0^2
+\frac{2\alpha \nu \Delta t}{\sqrt{k}}\|P_{\tau}(\boldsymbol{u}_c^{n+1})\|_{L^2(\mathbb{I})}^2\\
&=2 \Delta t \sum_{n=n_k}^{n_{k+1}-1} (\boldsymbol{f}_c(t_{n+1}),\boldsymbol{u}_c^{n+1})_{\Omega_c}
-2\Delta t \sum_{n=n_k}^{n_{k+1}-1} \bigg( \frac{\boldsymbol{u}_c^{n+1}-\hat{\boldsymbol{u}}_c^n}{\Delta t},\boldsymbol{u}_c^{n+1} \bigg)_{\Omega_c}\\
&+\frac{2 \Delta t}{\rho} \sum_{n=n_k}^{n_{k+1}-1} \int_{\mathbb{I}} \phi_f^{n_k} \boldsymbol{u}_c^{n+1} \cdot \boldsymbol{n}_d \mathrm{d}s
-\frac{2\gamma \Delta t}{\rho h} \sum_{n=n_k}^{n_{k+1}-1} \int_{\mathbb{I}} ((\boldsymbol{u}_c^{n+1}-\boldsymbol{u}_f^{n_k}) \cdot \boldsymbol{n}_d)(\boldsymbol{u}_c^{n+1} \cdot \boldsymbol{n}_d) \mathrm{d}s.
\end{split}
\end{equation}
Taking~$\boldsymbol{v}_m^h=2\Delta s \boldsymbol{u}_m^{n_{k+1}}$~and~$\psi_m^h=2\Delta s \phi_m^{n_{k+1}}$~in \eqref{Step2}, we obtain
\begin{equation}\label{WStep2}
\begin{split}
&\frac{\eta_m C_{mt}}{\rho} (\|\phi_m^{n_{k+1}}\|_0^2-\|\phi_m^{n_k}\|_0^2 + \|\phi_m^{n_{k+1}}-\phi_m^{n_k}\|_0^2)
+\frac{2 \mu \Delta s}{\rho k_m} \|\boldsymbol{u}_m^{n_{k+1}}\|_0^2\\
&+\frac{\sigma k_m \Delta s}{\rho \mu}(\|\phi_m^{n_{k+1}}\|_0^2 - \|\phi_f^{n_k}\|_0^2 + \|\phi_m^{n_{k+1}}-\phi_f^{n_k}\|_0^2)
=0.
\end{split}
\end{equation}
Taking~$\boldsymbol{v}_f^h=2\Delta s \boldsymbol{u}_f^{n_{k+1}}$~and~$\psi_f^h=2 \Delta s \phi_f^{n_{k+1}}$~in~\eqref{Step3}, we have
\begin{equation}\label{WStep3}
\begin{split}
&\frac{\eta_f C_{ft}}{\rho}(\|\phi_f^{n_{k+1}}\|_0^2 -\|\phi_f^{n_k}\|_0^2 + \|\phi_f^{n_{k+1}}-\phi_f^{n_k}\|_0^2)
+\frac{2\mu \Delta s}{\rho k_f} \|\boldsymbol{u}_f^{n_{k+1}}\|_0^2\\
&+ \frac{\sigma k_m \Delta s}{\rho \mu}(\|\phi_f^{n_{k+1}}\|_0^2 - \|\phi_m^{n_k}\|_0^2 + \|\phi_f^{n_{k+1}}-\phi_m^{n_k}\|_0^2)
+ \frac{2\Delta s}{\rho} \int_{\mathbb{I}} \phi_f^{n_k} \boldsymbol{u}_f^{n_{k+1}} \cdot \boldsymbol{n}_d \mathrm{d}s\\
&-\frac{2\gamma \Delta t}{\rho h}\sum_{n=n_k}^{n_{k+1}-1}\int_{\mathbb{I}}((\boldsymbol{u}_c^{n+1}-\boldsymbol{u}_f^{n_{k+1}})\cdot \boldsymbol{n}_d)(\boldsymbol{u}_f^{n_{k+1}} \cdot \boldsymbol{n}_d) \mathrm{d}s
=\frac{2\Delta s}{\rho}(f_d(t_{n_{k+1}}),\phi_f^{n_{k+1}})_{\Omega_d}.
\end{split}
\end{equation}
Combining~\eqref{WStep1}~\eqref{WStep2}~and~\eqref{WStep3}, we obtain
\begin{align}
&2\nu \Delta t \sum_{n=n_k}^{n_{k+1}-1}\|\nabla \boldsymbol{u}_c^{n+1}\|_0^2
+\frac{2\alpha \nu \Delta t}{\sqrt{k}}\|P_{\tau}(\boldsymbol{u}_c^{n+1})\|_{L^2(\mathbb{I})}^2
+ \frac{\eta_m C_{mt}}{\rho} (\|\phi_m^{n_{k+1}}\|_0^2-\|\phi_m^{n_k}\|_0^2 + \|\phi_m^{n_{k+1}}-\phi_m^{n_k}\|_0^2) \nonumber\\
&+ \frac{\eta_f C_{ft}}{\rho}(\|\phi_f^{n_{k+1}}\|_0^2 -\|\phi_f^{n_k}\|_0^2 + \|\phi_f^{n_{k+1}}-\phi_f^{n_k}\|_0^2)
+\frac{2\mu \Delta s}{\rho k_f} \|\boldsymbol{u}_f^{n_{k+1}}\|_0^2
+\frac{2 \mu \Delta s}{\rho k_m} \|\boldsymbol{u}_m^{n_{k+1}}\|_0^2 \nonumber\\
&+ \frac{\sigma k_m \Delta s}{\rho \mu}(\|\phi_m^{n_{k+1}}\|_0^2 - \|\phi_f^{n_k}\|_0^2
+ \|\phi_m^{n_{k+1}}-\phi_f^{n_k}\|_0^2
+ \|\phi_f^{n_{k+1}}\|_0^2 - \|\phi_m^{n_k}\|_0^2 + \|\phi_f^{n_{k+1}}-\phi_m^{n_k}\|_0^2) \nonumber\\
&= \bigg[ 2 \Delta t \sum_{n=n_k}^{n_{k+1}-1} (\boldsymbol{f}_c(t_{n+1}),\boldsymbol{u}_c^{n+1})_{\Omega_c}
+ \frac{2\Delta s}{\rho}(f_d(t_{n_{k+1}}),\phi_f^{n_{k+1}})_{\Omega_d} \bigg]
-2\Delta t \sum_{n=n_k}^{n_{k+1}-1} \bigg( \frac{\boldsymbol{u}_c^{n+1}-\hat{\boldsymbol{u}}_c^n}{\Delta t},\boldsymbol{u}_c^{n+1} \bigg)_{\Omega_c} \nonumber\\
&+ \frac{2 \Delta t}{\rho} \sum_{n=n_k}^{n_{k+1}-1} \int_{\mathbb{I}} \phi_f^{n_k} (\boldsymbol{u}_c^{n+1} -\boldsymbol{u}_f^{n_{k+1}})\cdot \boldsymbol{n}_d \mathrm{d}s
+ \bigg[ \frac{2\gamma \Delta t}{\rho h}\sum_{n=n_k}^{n_{k+1}-1}\int_{\mathbb{I}}((\boldsymbol{u}_c^{n+1}-\boldsymbol{u}_f^{n_{k+1}})\cdot \boldsymbol{n}_d)(\boldsymbol{u}_f^{n_{k+1}} \cdot \boldsymbol{n}_d) \mathrm{d}s \nonumber\\
&- \frac{2\gamma \Delta t}{\rho h} \sum_{n=n_k}^{n_{k+1}-1} \int_{\mathbb{I}} ((\boldsymbol{u}_c^{n+1}-\boldsymbol{u}_f^{n_k}) \cdot \boldsymbol{n}_d)(\boldsymbol{u}_c^{n+1} \cdot \boldsymbol{n}_d) \mathrm{d}s \bigg] \nonumber\\
&:= \sum_{i=1}^4 I_i. \label{Sum}
\end{align}
For~$I_1$, we can split it into the following form.
\begin{align*}
I_1&=2\Delta t \sum_{n=n_k}^{n_{k+1}-1}(\boldsymbol{f}_c(t_{n+1}),\boldsymbol{u}_c^{n+1})_{\Omega_c}
+\frac{2\Delta s}{\rho}(f_d(t_{n_{k+1}}),\phi_f^{n_{k+1}})_{\Omega_d}\\
&:= I_{11} + I_{12}.
\end{align*}
The two terms are bounded by H$\mathrm{\ddot{o}}$lder inequality, Poincar$\mathrm{\acute{e}}$-Friedriches~inequality and Young inequality,
\begin{align*}
I_{11} &\leq 2 C_{PF}\Delta t\sum_{n=n_k}^{n_{k+1}-1} \|\boldsymbol{f}_c(t_{n+1})\|_{H^{-1}} \|\nabla \boldsymbol{u}_c^{n+1}\|_0\\
&\leq \frac{C_{PF}^2 \Delta t}{\nu}\sum_{n=n_k}^{n_{k+1}-1}\|\boldsymbol{f}_c(t_{n+1})\|_{H^{-1}}^2
+ \nu \Delta t \sum_{n=n_k}^{n_{k+1}-1} \|\nabla \boldsymbol{u}_c^{n+1}\|_0^2,
\end{align*}
and
\begin{align*}
I_{12} &\leq \frac{2\Delta s}{\rho \eta_f C_{ft} } \|f_d(t_{n_{k+1}})\|_0^2
+ \frac{\eta_f C_{ft} \Delta s}{2 \rho} \|\phi_f^{n_{k+1}}\|_0^2.
\end{align*}
For the first term on the right-hand side in~\eqref{Sum},
\begin{equation*}
I_1 \leq \nu \Delta t \sum_{n=n_k}^{n_{k+1}-1}\|\nabla \boldsymbol{u}_c^{n+1}\|_0^2
+ \frac{C_{PF}^2 \Delta t}{\nu}\sum_{n=n_k}^{n_{k+1}-1}\|\boldsymbol{f}_c(t_{n+1})\|_{H^{-1}}^2
+ \frac{2\Delta s}{\rho \eta_f C_{ft} } \|f_d(t_{n_{k+1}})\|_0^2
+ \frac{\eta_f C_{ft} \Delta s}{2 \rho} \|\phi_f^{n_{k+1}}\|_0^2.
\end{equation*}
Applying the lemma~\ref{LEM}, the term
\begin{equation*}
\begin{split}
& -2\Delta t \bigg(\frac{\boldsymbol{u}_c^{n+1}-\hat{\boldsymbol{u}}_c^n}{\Delta t},\boldsymbol{u}_c^{n+1}\bigg)_{\Omega_c}\\
&= -(\boldsymbol{u}_c^{n+1}-\hat{\boldsymbol{u}}_c^n,\boldsymbol{u}_c^{n+1}+\hat{\boldsymbol{u}}_c^n
+\boldsymbol{u}_c^{n+1}-\hat{\boldsymbol{u}}_c^n)_{\Omega_c}\\
&\leq -\bigg[ (\boldsymbol{u}_c^{n+1},\boldsymbol{u}_c^{n+1})_{\Omega_c}- (\hat{\boldsymbol{u}}_c^n,\hat{\boldsymbol{u}}_c^n)_{\Omega_c} \bigg]\\
&= -\bigg[ \|\boldsymbol{u}_c^{n+1}\|_0^2 - \|\boldsymbol{u}_c^{n}\|_0^2
+ (\boldsymbol{u}_c^n,\boldsymbol{u}_c^n)_{\Omega_c}-(\hat{\boldsymbol{u}}_c^n,\hat{\boldsymbol{u}}_c^n)_{\Omega_c} \bigg]\\
&\leq -\bigg[ \|\boldsymbol{u}_c^{n+1}\|_0^2 - \|\boldsymbol{u}_c^{n}\|_0^2 \bigg]
+ \hat{C} \Delta t \|\boldsymbol{u}_c^n\|_0^2.
\end{split}
\end{equation*}
The second term on the right-hand in~\eqref{Sum}~is bounded by
\begin{equation*}
\begin{split}
I_2 &\leq -\sum_{n=n_k}^{n_{k+1}-1} \bigg[ \|\boldsymbol{u}_c^{n+1}\|_0^2 - \|\boldsymbol{u}_c^{n}\|_0^2 \bigg]
+ \hat{C} \Delta t \sum_{n=n_k}^{n_{k+1}-1} \|\boldsymbol{u}_c^n\|_0^2.
\end{split}
\end{equation*}
Using the Cauchy-Schwarz inequality, trace inverse inequality and the Young inequality, we show that
\begin{equation*}
\begin{split}
I_3 \leq \frac{2 \tilde{C}_{inv}^2 \Delta s}{\rho \gamma} \|\phi_f^{n_k}\|_0^2
+\frac{\gamma \Delta t}{2 \rho h}\sum_{n=n_k}^{n_{k+1}-1} \|(\boldsymbol{u}_c^{n+1}-\boldsymbol{u}_f^{n_{k+1}})\cdot \boldsymbol{n}_d\|_{L^2(\mathbb{I})}^2.
\end{split}
\end{equation*}
For~$I_4$, the identity transformation is carried out.
\begin{align}
I_4 &= \frac{2\gamma \Delta t}{\rho h}\sum_{n=n_k}^{n_{k+1}-1}\int_{\mathbb{I}}((\boldsymbol{u}_c^{n+1}-\boldsymbol{u}_f^{n_{k+1}})\cdot \boldsymbol{n}_d)(\boldsymbol{u}_f^{n_{k+1}} \cdot \boldsymbol{n}_d) \mathrm{d}s \nonumber\\
&~~~~- \frac{2\gamma \Delta t}{\rho h} \sum_{n=n_k}^{n_{k+1}-1} \int_{\mathbb{I}} ((\boldsymbol{u}_c^{n+1}-\boldsymbol{u}_f^{n_k}) \cdot \boldsymbol{n}_d)(\boldsymbol{u}_c^{n+1} \cdot \boldsymbol{n}_d) \mathrm{d}s \nonumber\\
&= -\frac{2\gamma \Delta t}{\rho h} \sum_{n=n_k}^{n_{k+1}-1} \|(\boldsymbol{u}_c^{n+1} - \boldsymbol{u}_f^{n_{k+1}}) \cdot \boldsymbol{n}_d\|_{L^2(\mathbb{I})}^2
-\frac{2\gamma \Delta t}{\rho h}\sum_{n=n_k}^{n_{k+1}-1}\int_{\mathbb{I}} ((\boldsymbol{u}_f^{n_{k+1}}-\boldsymbol{u}_f^{n_k}) \cdot \boldsymbol{n}_d)(\boldsymbol{u}_c^{n+1} \cdot \boldsymbol{n}_d)\mathrm{d}s \nonumber\\
&=-\frac{2\gamma \Delta t}{\rho h} \sum_{n=n_k}^{n_{k+1}-1} \|(\boldsymbol{u}_c^{n+1}-\boldsymbol{u}_f^{n_{k+1}})\cdot \boldsymbol{n}_d\|_{L^2(\mathbb{I})}^2
-\frac{2\gamma \Delta t}{\rho h}\sum_{n=n_k}^{n_{k+1}-1} \bigg[ \int_{\mathbb{I}} ((\boldsymbol{u}_f^{n_{k+1}}-\boldsymbol{u}_f^{n_k}) \cdot \boldsymbol{n}_d)(\boldsymbol{u}_f^{n_{k+1}} \cdot \boldsymbol{n}_d)\mathrm{d}s \nonumber\\
&~~~~- \int_{\mathbb{I}} ((\boldsymbol{u}_f^{n_{k+1}}-\boldsymbol{u}_f^{n_k})\cdot \boldsymbol{n}_d)((\boldsymbol{u}_f^{n_{k+1}}-\boldsymbol{u}_c^{n+1})\cdot \boldsymbol{n}_d) \mathrm{d}s \bigg]. \nonumber
\end{align}
Since
\begin{equation*}
\begin{split}
&\int_{\mathbb{I}} ((\boldsymbol{u}_f^{n_{k+1}}-\boldsymbol{u}_f^{n_k})\cdot \boldsymbol{n}_d)((\boldsymbol{u}_f^{n_{k+1}}-\boldsymbol{u}_c^{n+1})\cdot \boldsymbol{n}_d) \mathrm{d}s\\
&\leq \frac{1}{2}\|(\boldsymbol{u}_f^{n_{k+1}}-\boldsymbol{u}_f^{n_k})\cdot \boldsymbol{n}_d\|_{L^2(\mathbb{I})}^2
+ \frac{1}{2}\|(\boldsymbol{u}_f^{n_{k+1}}-\boldsymbol{u}_c^{n+1})\cdot \boldsymbol{n}_d\|_{L^2(\mathbb{I})}^2\\
&\leq \frac{1}{2}\|\boldsymbol{u}_f^{n_{k+1}}\cdot \boldsymbol{n}_d\|_{L^2(\mathbb{I})}^2
+\frac{1}{2} \|\boldsymbol{u}_f^{n_k} \cdot \boldsymbol{n}_d\|_{L^2(\mathbb{I})}^2
-\int_{\mathbb{I}} \boldsymbol{u}_f^{n_k} \cdot \boldsymbol{n}_d \boldsymbol{u}_f^{n_{k+1}} \cdot \boldsymbol{n}_d \mathrm{d}s
+ \frac{1}{2} \|(\boldsymbol{u}_f^{n_{k+1}}-\boldsymbol{u}_c^{n+1})\cdot \boldsymbol{n}_d\|_{L^2(\mathbb{I})}^2,
\end{split}
\end{equation*}
then the fourth term~$I_4$~on the right-hand in~\eqref{Sum}~is yielded that
\begin{equation*}
\begin{split}
I_4 & \leq -\frac{2\gamma \Delta t}{\rho h} \sum_{n=n_k}^{n_{k+1}-1} \|(\boldsymbol{u}_c^{n+1}-\boldsymbol{u}_f^{n_{k+1}})\cdot \boldsymbol{n}_d\|_{L^2(\mathbb{I})}^2
- \frac{\gamma \Delta t}{\rho h} \sum_{n=n_k}^{n_{k+1}-1} \big(\|\boldsymbol{u}_f^{n_{k+1}}\cdot \boldsymbol{n}_d\|_{L^2(\mathbb{I})}^2 \\
&- \|\boldsymbol{u}_f^{n_k} \cdot \boldsymbol{n}_d\|_{L^2(\mathbb{I})}^2-\|(\boldsymbol{u}_f^{n_{k+1}}-\boldsymbol{u}_c^{n+1})\cdot \boldsymbol{n}_d\|_{L^2(\mathbb{I})}^2 \big).
\end{split}
\end{equation*}
Combining above inequalities, summing over~$k=0,...,M-1$~and using the discrete Growall lemma, we get the stability in the time interval~$[0,t_{n_M}]$
\begin{align*}
&\|\boldsymbol{u}_c^{n_{M}}\|_0^2 + \nu \Delta t \sum_{k=0}^{M-1}\sum_{n=n_k}^{n_{k+1}-1}\|\nabla \boldsymbol{u}_c^{n+1}\|_0^2
+\frac{2\alpha \nu \Delta t}{\sqrt{k}}\sum_{k=0}^{M-1} \sum_{n=n_k}^{n_{k+1}-1}\|P_{\tau} (\boldsymbol{u}_c^{n+1})\|_{L^2(\mathbb{I})}^2
+\frac{2\mu \Delta s}{\rho k_m}\sum_{k=0}^{M-1}\|\boldsymbol{u}_m^{n_{k+1}}\|_0^2\\
&+\frac{\gamma \Delta t}{2\rho h}\sum_{k=0}^{M-1} \sum_{n=n_k}^{n_{k+1}-1}\|(\boldsymbol{u}_c^{n+1}-\boldsymbol{u}_f^{n_{k+1}})\cdot \boldsymbol{n}_d\|_{L^2(\mathbb{I})}^2
+\frac{\gamma \Delta s}{\rho h} \|\boldsymbol{u}_f^{n_{M}} \cdot \boldsymbol{n}_d\|_{L^2(\mathbb{I})}^2
+\frac{2\mu \Delta s}{\rho k_f}\sum_{k=0}^{M-1}\|\boldsymbol{u}_f^{n_{k+1}}\|_0^2\\
&+\frac{\eta_m C_{mt}}{\rho}\bigg(\|\phi_m^{n_{M}}\|_0^2 + \sum_{k=0}^{M-1} \|\phi_m^{n_{k+1}}-\phi_m^{n_k}\|_0^2 \bigg)
+\frac{\eta_f C_{ft}}{\rho}\bigg(\|\phi_f^{n_{M}}\|_0^2 + \sum_{k=0}^{M-1} \|\phi_f^{n_{k+1}}-\phi_f^{n_k}\|_0^2 \bigg)\\
&+\frac{\sigma k_m \Delta s}{\rho \mu}\bigg( \|\phi_m^{n_M}\|_0^2 + \sum_{k=0}^{M-1} \|\phi_m^{n_{k+1}}-\phi_f^{n_k}\|_0^2
+\|\phi_f^{n_M}\|_0^2
+\sum_{k=0}^{M-1} \|\phi_f^{n_{k+1}}-\phi_m^{n_k}\|_0^2 \bigg)\\
&\leq  C \bigg( \frac{C_{PF}^2 \Delta t}{\nu}\sum_{k=0}^{M-1} \sum_{n=n_k}^{n_{k+1}-1}\|\boldsymbol{f}_c(t_{n+1})\|_{H^{-1}}^2
+ \frac{2\Delta t}{\rho \eta_f C_{ft}}\sum_{k=0}^{M-1}\sum_{n=n_k}^{n_{k+1}-1}\|f_d(t_{n_{k+1}})\|_0^2
+ \frac{\gamma \Delta s}{\rho h} \|\boldsymbol{u}_f^0 \cdot \boldsymbol{n}_d\|_{L^2(\mathbb{I})}^2\\
&+ \|\boldsymbol{u}_c^0\|_0^2
+ \frac{\eta_m C_{mt}}{\rho}\|\phi_m^0\|_0^2
+ \frac{\eta_f C_{ft}}{\rho} \|\phi_f^0\|_0^2
+ \frac{\sigma k_m \Delta s}{\rho \mu} (\|\phi_m^0\|_0^2 + \|\phi_f^0\|_0^2 ) \bigg).
\end{align*}
with the condition~$\bigg( \frac{\frac{\eta_f C_{ft}}{2\rho} + \frac{2 C_{inv}^2}{\rho \gamma}}{\frac{\eta_f C_{ft}}{\rho}} + \hat{C} \bigg) \Delta t \leq \frac{1}{2}$, which leads to
\begin{equation*}
\gamma \geq \frac{4 C_{inv}^2}{\eta_f C_{ft}}\bigg( \frac{\Delta t}{1-\Delta t -2 \hat{C} \Delta t }\bigg).
\end{equation*}

Taking~$\boldsymbol{v}_c^h=2 \Delta t \boldsymbol{u}_c^{n+1}$~and~$q^h=2 \Delta t p_c^{n+1}$~in~\eqref{Step1}, sum over~$n=n_k,...,n_{k}+J~(0\leq J \leq r-1)$,
\begin{align}
&2\nu \Delta t \sum_{n=n_k}^{n_{k}+J}\|\nabla \boldsymbol{u}_c^{n+1}\|_0^2
+\frac{2\alpha \nu \Delta t}{\sqrt{k}}\sum_{n=n_k}^{n_k+J}\|P_{\tau}(\boldsymbol{u}_c^{n+1})\|_{L^2(\mathbb{I})}^2 \nonumber \\
&=2 \Delta t \sum_{n=n_k}^{n_{k}+J} (\boldsymbol{f}_c(t_{n+1}),\boldsymbol{u}_c^{n+1})_{\Omega_c}
-2\Delta t \sum_{n=n_k}^{n_{k}+J} \bigg( \frac{\boldsymbol{u}_c^{n+1}-\hat{\boldsymbol{u}}_c^n}{\Delta t},\boldsymbol{u}_c^{n+1} \bigg)_{\Omega_c} \label{WStep11} \\
&+\frac{2 \Delta t}{\rho} \sum_{n=n_k}^{n_{k}+J} \int_{\mathbb{I}} \phi_f^{n_k} \boldsymbol{u}_c^{n+1} \cdot \boldsymbol{n}_d \mathrm{d}s
-\frac{2\gamma \Delta t}{\rho h} \sum_{n=n_k}^{n_{k}+J} \int_{\mathbb{I}} ((\boldsymbol{u}_c^{n+1}-\boldsymbol{u}_f^{n_k}) \cdot \boldsymbol{n}_d)(\boldsymbol{u}_c^{n+1} \cdot \boldsymbol{n}_d) \mathrm{d}s \nonumber \\
&:=\sum_{i=1}^4 II_i. \nonumber
\end{align}
Similarly, we obtain
\begin{equation*}
\begin{split}
II_1 &\leq \frac{2 C_{PF}^2 \Delta t}{\nu}\sum_{n=n_k}^{n_k+J}\|\boldsymbol{f}_c(t_{n+1})\|_{H^{-1}}^2
+\frac{\nu \Delta t}{2}\sum_{n=n_k}^{n_k+J}\|\nabla \boldsymbol{u}_c^{n+1}\|_0^2,\\
II_2 &\leq -\sum_{n=n_k}^{n_{k}+J} \bigg[ \|\boldsymbol{u}_c^{n+1}\|_0^2 - \|\boldsymbol{u}_c^{n}\|_0^2 \bigg]
+ \hat{C} \Delta t \sum_{n=n_k}^{n_{k}+J} \|\boldsymbol{u}_c^n\|_0^2,\\
II_3 &\leq \frac{C_T^2 \tilde{C}_{inv}^2 C_{PF}\Delta t}{\rho^2 \nu h}\sum_{n=n_k}^{n_{k}+J} \|\phi_f^{n_k}\|_0^2
+ \nu \Delta t \sum_{n=n_k}^{n_k+J} \|\nabla \boldsymbol{u}_c^{n+1}\|_0^2,\\
II_4 &= -\frac{\gamma \Delta t}{\rho h}\sum_{n=n_k}^{n_k+J}\|(\boldsymbol{u}_c^{n+1}-\boldsymbol{u}_f^{n_k})\cdot \boldsymbol{n}_d\|_{L^2(\mathbb{I})}^2
+\frac{\gamma \Delta t}{\rho h}\sum_{n=n_k}^{n_k+J} (\|\boldsymbol{u}_f^{n_k} \cdot \boldsymbol{n}_d\|_{L^2(\mathbb{I})}^2 - \|\boldsymbol{u}_c^{n+1} \cdot \boldsymbol{n}_d\|_{L^2(\mathbb{I})}^2).
\end{split}
\end{equation*}
Using the discrete~Growall~lemma, when~$\hat{C} \Delta t < 1$, then
\begin{align*}
&\|\boldsymbol{u}_c^{n_{k}+J+1}\|_0^2
+\frac{\nu \Delta t}{2} \sum_{n=n_k}^{n_{k}+J}\|\nabla \boldsymbol{u}_c^{n+1}\|_0^2
+\frac{2\alpha \nu \Delta t}{\sqrt{k}}\sum_{n=n_k}^{n_k+J}\|P_{\tau}(\boldsymbol{u}_c^{n+1})\|_{L^2(\mathbb{I})}^2\\
&+\frac{\gamma \Delta t}{\rho h}\sum_{n=n_k}^{n_k+J}\|(\boldsymbol{u}_c^{n+1}-\boldsymbol{u}_f^{n_k})\cdot \boldsymbol{n}_d\|_{L^2(\mathbb{I})}^2
+\frac{\gamma \Delta t}{\rho h}\sum_{n=n_k}^{n_k+J} \|\boldsymbol{u}_c^{n+1} \cdot \boldsymbol{n}_d\|_{L^2(\mathbb{I})}^2\\
&\leq C \bigg(\frac{2 C_{PF}^2 \Delta t}{\nu}\sum_{n=n_k}^{n_k+J}\|\boldsymbol{f}_c(t_{n+1})\|_{H^{-1}}^2
+\frac{C_T^2 \tilde{C}_{inv}^2 C_{PF} \Delta t}{\rho^2 \nu h}\sum_{n=n_k}^{n_{k}+J} \|\phi_f^{n_k}\|_0^2
+\frac{\gamma \Delta t}{\rho h}\sum_{n=n_k}^{n_k+J}\|\boldsymbol{u}_f^{n_k} \cdot \boldsymbol{n}_d\|_{L^2(\mathbb{I})}^2
+\|\boldsymbol{u}_c^{n_k}\|_0^2 \bigg).
\end{align*}
In particular, taking~$k=0$, we have the stability in first large time~$[0,t_{n_1}]$,
\begin{equation*}
\begin{split}
&\|\boldsymbol{u}_c^{J+1}\|_0^2
+\frac{\nu \Delta t}{2} \sum_{n=0}^{J}\|\nabla \boldsymbol{u}_c^{n+1}\|_0^2
+\frac{2\alpha \nu \Delta t}{\sqrt{k}}\sum_{n=0}^{J}\|P_{\tau}(\boldsymbol{u}_c^{n+1})\|_{L^2(\mathbb{I})}^2\\
&+\frac{\gamma \Delta t}{\rho h}\sum_{n=0}^{J}\|(\boldsymbol{u}_c^{n+1}-\boldsymbol{u}_f^{n_k})\cdot \boldsymbol{n}_d\|_{L^2(\mathbb{I})}^2
+\frac{\gamma \Delta t}{\rho h}\sum_{n=0}^{J} \|\boldsymbol{u}_c^{n+1} \cdot \boldsymbol{n}_d\|_{L^2(\mathbb{I})}^2\\
&\leq C \bigg( \frac{2 C_{PF}^2 \Delta t}{\nu}\sum_{n=0}^{J}\|\boldsymbol{f}_c(t_{n+1})\|_{H^{-1}}^2
+\frac{C_T^2 \tilde{C}_{inv}^2 C_{PF} \Delta t}{\rho^2 \nu h}\sum_{n=0}^{J} \|\phi_f^{0}\|_0^2
+\frac{\gamma \Delta t}{\rho h}\sum_{n=0}^{J}\|\boldsymbol{u}_f^{0} \cdot \boldsymbol{n}_d\|_{L^2(\mathbb{I})}^2
+\|\boldsymbol{u}_c^{0}\|_0^2 \bigg).
\end{split}
\end{equation*}

\end{proof}

\section{Error analysis}
\subsection{Preliminaries}
In this subsection, we will introduce a projection operator, some lemmas and techniques that need to be used.

\noindent \textbf{A. Some lemmas}\\
Define a bounded linear projection operator~$P^h$, for all~$t \in [0,T]$,
\begin{equation*}
\begin{split}
&P^h=(P_{c1}^h,P_{c2}^h;P_{f1}^h,P_{f2}^h;P_{m1}^h,P_{m2}^h): (\boldsymbol{u}_c(t),p_c(t);\boldsymbol{u}_f(t),\phi_f(t);\boldsymbol{u}_m(t),\phi_m(t)) \in \mathcal{X} \mapsto\\
&(P_{c1}^h \boldsymbol{u}_c(t),P_{c2}^h p_c(t); P_{f1}^h \boldsymbol{u}_f(t), P_{f2}^h \phi_f(t); P_{m1}^h \boldsymbol{u}_m(t), P_{m2}^h \phi_m(t)) \in \mathcal{X}_h,
\end{split}
\end{equation*}
satisfying
\begin{equation}\label{projection}
\begin{split}
&a_{\Omega_c}(\boldsymbol{u}_c-P^h \boldsymbol{u}_c,\boldsymbol{v}_c^h)
-b(\boldsymbol{v}_c^h,p_c-P^h p_c)
+b(\boldsymbol{u}_c-P^h \boldsymbol{u}_c,q^h)
+a_{\phi_d}(\phi_d-P^h \phi_d,\psi_d^h)\\
&+a_{\boldsymbol{u}_d}(\boldsymbol{u}_d-P^h \boldsymbol{u}_d,\boldsymbol{v}_d^h)s
-b_{\phi_d}(\boldsymbol{v}_d^h,\phi_d-P^h \phi_d)
+b_{\phi_d}(\boldsymbol{u}_d-P^h \boldsymbol{u}_d,\psi_d^h)\\
&-\frac{1}{\rho}\int_{\mathbb{I}}(\phi_f-P^h \phi_f)(\boldsymbol{v}_c^h-\boldsymbol{v}_f^h)\cdot \boldsymbol{n}_d \mathrm{d}s
+\frac{\gamma}{\rho h}\int_{\mathbb{I}} ((\boldsymbol{u}_c-P^h \boldsymbol{u}_c)-(\boldsymbol{u}_f-P^h \boldsymbol{u}_f))\cdot \boldsymbol{n}_d
(\boldsymbol{v}_c^h-\boldsymbol{v}_f^h)\cdot \boldsymbol{n}_d \mathrm{d}s\\
&=0, \forall (\boldsymbol{v}_c^h,q^h;\boldsymbol{v}_f^h,\psi_f^h;\boldsymbol{v}_m^h,\psi_m^h)\in (Y_c^h,Q_c^h,Y_f^h,Q_f^h,Y_m^h,Q_m^h).
\end{split}
\end{equation}

Suppose that the solution~$(\boldsymbol{u}_c(t),p_c(t);\boldsymbol{u}_f(t),\phi_f(t);\boldsymbol{u}_m(t),\phi_m(t))$~to the variational formulation~\eqref{yuanshi}~satisfy
\begin{equation}\label{suppose}
\begin{split}
&\|\boldsymbol{u}_c\|_{L^{\infty}(0,T;H^2)}+\|\boldsymbol{u}_c\|_{L^{\infty}(0,T;W^{2,d^{\ast}})}
+\|p_c\|_{L^{\infty}(0,T;H^1)}
+\|\boldsymbol{u}_f\|_{L^{\infty}(0,T;H^2)}+\|\phi_f\|_{L^{\infty}(0,T;H^1)}\\
&+\|\boldsymbol{u}_m\|_{L^{\infty}(0,T;H^2)}
+\|\phi_m\|_{L^{\infty}(0,T;H^1)}+\|\boldsymbol{u}_{ct}\|_{L^2(0,T;H^1)}+\|\boldsymbol{u}_{ft}\|_{L^2(0,T;H^1)}
+\|\phi_{ft}\|_{L^2(0,T;L^2)}\\
&+\|\phi_{ftt}\|_{L^2(0,T;L^2)}
+\|\phi_{mt}\|_{L^2(0,T;L^2)}+\|\phi_{mtt}\|_{L^2(0,T;L^2)}
\leq C_H,
\end{split}
\end{equation}
where~$d^{\ast} > D$.

And assume that the solutions satisfy the following approximation properties~\cite{Layton2002, Franco1991}
\begin{align}
\|\boldsymbol{u}_c - P^h \boldsymbol{u}_c\|_0 + h\|\nabla (\boldsymbol{u}_c-P^h \boldsymbol{u}_c)\|_0 + h\|p_c-P^h p_c\|_0 \leq C_{pc} h^2,\label{pro1}\\
\|\boldsymbol{u}_f - P^h \boldsymbol{u}_f\|_0 + h\|\nabla \cdot (\boldsymbol{u}_f - P^h \boldsymbol{u}_f)\|_0 + h \|\phi_f -P^h \phi_f\|_0 \leq C_{pf} h^2,\label{pro2}\\
\|\boldsymbol{u}_m - P^h \boldsymbol{u}_m\|_0 + h\|\nabla \cdot (\boldsymbol{u}_m - P^h \boldsymbol{u}_m)\|_0 + h \|\phi_m - P^h \phi_m\|_0 \leq C_{pm} h^2. \label{pro3}
\end{align}

Furthermore, we suppose the projection operator~$P^h$~satisfying
\begin{equation}\label{Assume}
\|P^h \boldsymbol{u}_c(t_n)\|_{L^{\infty}} \leq C \|\boldsymbol{u}_c(t_n)\|_2,~~~~
\|P^h \boldsymbol{u}_c(t_n)\|_{W^{1,\infty}} \leq C \|\boldsymbol{u}_c(t_n)\|_{W^{1,\infty}}.
\end{equation}

\begin{lemma}[\cite{suli1988convergence}]\label{Lem1}
Let~$R_{tr}^{n+1}=\frac{\boldsymbol{u}_c(t_{n+1})-\bar{\boldsymbol{u}}_c(t_n)}{\Delta t}
-\frac{\partial \boldsymbol{u}_c}{\partial t}(t_{n+1})-(\boldsymbol{u}_c(t_{n+1}) \cdot \nabla )\boldsymbol{u}_c(t_{n+1})$. It holds that
\begin{equation*}
\Delta t \sum_{n=0}^{N-1} \|R_{tr}^{n+1}\|_0^2 \leq C \Delta t^2.
\end{equation*}
\end{lemma}

\begin{lemma}[\cite{Si2016Unconditional}]\label{Lem2}
Assume that~$g_1, g_2$~and~$\rho$~are three functions defined in~$\Omega$~and vanish on~$\partial \Omega$. If
\begin{equation*}
\Delta t (\|g_1\|_{W^{1,\infty}} + \|g_2\|_{W^{1,\infty}}) \leq \frac{1}{2},
\end{equation*}
then
\begin{equation*}
\begin{split}
\|\rho(x-g_1(x)\Delta t) - \rho(x-g_2(x)\Delta t)\|_{L^q} &\leq C_N \Delta t \|\rho\|_{W^{1,q_1}} \|g_1-g_2\|_{L^{q_2}},\\
\|\rho(x)-\rho(x-g_2(x)\Delta t)\|_{-1} &\leq C_n \Delta t \|\rho\|_0 \|g_2\|_{W^{1,4}},
\end{split}
\end{equation*}
where~$1/q_1 + 1/q_2 = 1/q$~and~$1 \leq q < \infty$.
\end{lemma}

For~$(\boldsymbol{v}_c,q;\boldsymbol{v}_f,\psi_f;\boldsymbol{v}_m,\psi_m) \in \mathcal{X}^h$, using the characteristic version of the variational formulation and the definition of projection operator~$P^h$, we have
\begin{equation}\label{S1}
\begin{split}
&(d_t \boldsymbol{u}_c(t_{n+1}),\boldsymbol{v}_c)_{\Omega_c}
+a_{\Omega_c}(P^h \boldsymbol{u}_c(t_{n+1}),\boldsymbol{v}_c)
-b(\boldsymbol{v}_c,P^h p_c(t_{n+1}))+b(P^h \boldsymbol{u}_c(t_{n+1}),q)\\
&-\frac{1}{\rho}\int_{\mathbb{I}}P^h \phi_f(t_{n+1})\boldsymbol{v}_c \cdot \boldsymbol{n}_d \mathrm{d}s
+\frac{\gamma}{\rho h}\int_{\mathbb{I}}(P^h \boldsymbol{u}_c(t_{n+1})-P^h \boldsymbol{u}_f(t_{n+1}))\cdot \boldsymbol{n}_d \boldsymbol{v}_c \cdot \boldsymbol{n}_d \mathrm{d}s\\
&=(\boldsymbol{f}_c(t_{n+1}),\boldsymbol{v}_c)_{\Omega_c}
-\bigg(\frac{\boldsymbol{u}_c(t_n)-\bar{\boldsymbol{u}}_c(t_n)}{\Delta t},\boldsymbol{v}_c\bigg)_{\Omega_c}
+(R_{tr}^{n+1},\boldsymbol{v}_c)_{\Omega_c},
\end{split}
\end{equation}

\begin{equation}\label{S2}
\begin{split}
&\frac{\eta_m C_{mt}}{\rho}(d_s \phi_m(t_{n_{k+1}}),\psi_m)_{\Omega_d}
+\frac{\sigma k_m}{\rho \mu}(P^h \phi_m(t_{n_{k+1}}) - P^h \phi_f(t_{n_{k+1}}),\psi_m)_{\Omega_d}\\
&+\frac{1}{\rho}(\mu k_m^{-1} P^h \boldsymbol{u}_m(t_{n_{k+1}}), \boldsymbol{v}_m)_{\Omega_d}
+\frac{1}{\rho}(\nabla \cdot P^h \boldsymbol{u}_m(t_{n_{k+1}}),\psi_m)_{\Omega_d}
-\frac{1}{\rho}(P^h \phi_m(t_{n_{k+1}}), \nabla \cdot \boldsymbol{v}_m)_{\Omega_d}\\
&=\frac{\eta_m C_{mt}}{\rho}\bigg(d_s \phi_m(t_{n_{k+1}})- \frac{\partial \phi_m}{\partial t}(t_{n_{k+1}}),\psi_m\bigg)_{\Omega_d},
\end{split}
\end{equation}
and
\begin{equation}\label{S3}
\begin{split}
&\frac{\eta_f C_{ft}}{\rho}(d_s \phi_f(t_{n_{k+1}}),\psi_f)_{\Omega_d}
+\frac{\sigma k_m}{\rho \mu}(P^h \phi_f(t_{n_{k+1}})-P^h \phi_m(t_{n_{k+1}}),\psi_f)_{\Omega_d}\\
&+\frac{1}{\rho}(\mu k_f^{-1} P^h \boldsymbol{u}_f(t_{n_{k+1}}),\boldsymbol{v}_f)_{\Omega_d}
+\frac{1}{\rho}(\nabla \cdot P^h \boldsymbol{u}_f(t_{n_{k+1}}),\psi_f)_{\Omega_d}
-\frac{1}{\rho}(P^h \phi_f(t_{n_{k+1}}),\nabla \cdot \boldsymbol{v}_f)_{\Omega_d}\\
&+\frac{1}{\rho}\int_{\mathbb{I}}P^h \phi_f(t_{n_{k+1}})\boldsymbol{v}_f \cdot \boldsymbol{n}_d \mathrm{d}s
-\frac{\gamma}{\rho h}\int_{\mathbb{I}}(P^h \boldsymbol{u}_c(t_{n+1})- P^h \boldsymbol{u}_f(t_{n+1}))\cdot \boldsymbol{n}_d \boldsymbol{v}_f \cdot \boldsymbol{n}_d \mathrm{d}s\\
&=\frac{1}{\rho}(f_{d}(t_{n_{k+1}}),\psi_f)_{\Omega_d}
+\frac{\eta_f C_{ft}}{\rho}\bigg(d_s \phi_f(t_{n_{k+1}})- \frac{\partial \phi_f}{\partial t}(t_{n_{k+1}}),\psi_f\bigg)_{\Omega_d}.
\end{split}
\end{equation}

\noindent \textbf{B. Some techniques}\\
For simplicity, we show some techniques in estimation. Let~$\Psi_{d,t}^{n+1}=\Psi_{d,t,1}^{n+1}+\Psi_{d,t,2}^{n+1}$, where~$\Psi_{d,t,1}^{n+1}=d_t \phi_d(t_{n+1})-\frac{\partial \phi_d}{\partial t}(t_{n+1})$~and~$\Psi_{d,t,2}^{n+1}= d_t(P^h \phi_d(t_{n+1})-\phi_d(t_{n+1})), d=f~\text{and}~m$. Let~$\Psi_{c,t}^{n+1}=d_t(P^h \boldsymbol{u}_c(t_{n+1})-\boldsymbol{u}_c(t_{n+1}))$. For
\begin{equation*}
\begin{split}
\Delta t \Psi_{f,t,1}^{n+1}&=\phi_f(t_{n+1})- \phi_f(t_n) - \Delta t \frac{\partial \phi_f}{\partial t}(t_{n+1})
=\int_{t_n}^{t_{n+1}} (t_{n+1}-t) \frac{\partial^2 \phi_f}{\partial t^2}(t) \mathrm{d}t,
\end{split}
\end{equation*}
then
\begin{equation*}
\begin{split}
\|\Psi_{f,t,1}^{n+1}\|_0^2 &\leq \frac{1}{\Delta t^2} \int_{\Omega_d} \{\int_{t_n}^{t_{n+1}} \bigg(\frac{\partial^2 \phi_f}{\partial t^2}(t)\bigg)^2 \mathrm{d}t \int_{t_n}^{t_{n+1}} (t_{n+1}-t)^2\mathrm{d}t\} \mathrm{d}x\\
&\leq \frac{\Delta t}{3}\int_{t_n}^{t_{n+1}} \bigg\|\frac{\partial^2 \phi_f}{\partial t^2} \bigg\|_0^2 \mathrm{d}t.
\end{split}
\end{equation*}
For
\begin{equation*}
\begin{split}
\Psi_{f,t,2}^{n+1} &=\frac{P^h \phi_f(t_{n+1})-P^h \phi_f(t_n)}{\Delta t}
-\frac{\phi_f(t_{n+1})- \phi_f(t_n)}{\Delta t}\\
&=(P^h - I)\frac{\phi_f(t_{n+1})- \phi_f(t_n)}{\Delta t}\\
&=\frac{1}{\Delta t}\int_{t_n}^{t_{n+1}} (P^h -I)\frac{\partial \phi_f}{\partial t}(t) \mathrm{d}t,
\end{split}
\end{equation*}
then
\begin{equation*}
\begin{split}
\|\Psi_{f,t,2}^{n+1}\|_0^2 &= \frac{1}{\Delta t^2} \int_{\Omega_d} \bigg( \int_{t_n}^{t_{n+1}} (P^h - I) \frac{\partial \phi_f}{\partial t}(t) \mathrm{d}t \bigg)^2 \mathrm{d}x\\
&\leq \frac{1}{\Delta t^2} \int_{\Omega_d} \{\int_{t_n}^{t_{n+1}} ((P^h - I) \frac{\partial \phi_f}{\partial t}(t))^2 \mathrm{d}t \int_{t_n}^{t_{n+1}} 1^2 \mathrm{d}t \} \mathrm{d}x\\
&\leq \frac{1}{\Delta t} \int_{t_n}^{t_{n+1}} \bigg\|(P^h - I) \frac{\partial \phi_f}{\partial t}(t)\bigg\|_0^2 \mathrm{d}t.
\end{split}
\end{equation*}
Similarly,
\begin{equation*}
\begin{split}
\|\Psi_{m,t,1}^{n+1}\|_0^2 &\leq \frac{\Delta t}{3}\int_{t_n}^{t_{n+1}} \|\frac{\partial^2 \phi_m}{\partial t^2}\|_0^2 \mathrm{d}t,\\
\|\Psi_{m,t,2}^{n+1}\|_0^2 &\leq \frac{1}{\Delta t}\int_{t_n}^{t_{n+1}} \|(P^h-I)\frac{\partial \phi_m}{\partial t}(t)\|_0^2 \mathrm{d}t,\\
\|\Psi_{c,t}^{n+1}\|_0^2 &\leq \frac{1}{\Delta t}\int_{t_n}^{t_{n+1}} \|(P^h-I)\frac{\partial \boldsymbol{u}_c}{\partial t}(t)\|_0^2 \mathrm{d}t.
\end{split}
\end{equation*}

\subsection{Convergence analysis}
Some notations are defineded by
\begin{equation*}
\begin{split}
e_c^n=P^h \boldsymbol{u}_c(t_n) -\boldsymbol{u}_c^n;~~~~
\delta_c^n=P^h p_c(t_n) - p_c^n;~~~~
e_d^n=P^h \boldsymbol{u}_d(t_n)-\boldsymbol{u}_d^n;~~~~
\theta_d^n=P^h \phi_d(t_n)-\phi_d^n.
\end{split}
\end{equation*}
Then
\begin{equation*}
\begin{split}
\boldsymbol{u}_c(t_n)-\boldsymbol{u}_c^n:= \boldsymbol{u}_c(t_n)-P^h \boldsymbol{u}_c(t_n) + e_c^n,\\
p_c(t_n)-p_c^n := p_c(t_n)-P^h p_c(t_n) + \delta_c^n,\\
\boldsymbol{u}_d(t_n)-\boldsymbol{u}_d^n:=\boldsymbol{u}_d(t_n)-P^h \boldsymbol{u}_d(t_n) + e_d^n,\\
\phi_d(t_n)-\phi_d^n := \phi_d(t_n)-P^h \phi_d(t_n)+ \theta_d^n,
\end{split}
\end{equation*}
where~$d=f$~or~$m$.

Taking~$(\boldsymbol{v}_c,q;\boldsymbol{v}_f,\psi_f;\boldsymbol{v}_m,\psi_m)
=(\boldsymbol{v}_c^h,q^h;\boldsymbol{v}_f^h,\psi_f^h;\boldsymbol{v}_m^h,\psi_m^h)$~in~\eqref{S1}, \eqref{S2}~and~\eqref{S3}. \eqref{Step1}~\eqref{Step2}\\
and~\eqref{Step3}~are subtracted from the corresponding term, respectively. Which leads to
\begin{equation}\label{SS1}
\begin{split}
&\bigg( \frac{e_c^{n+1}-e_c^n}{\Delta t},\boldsymbol{v}_c^h\bigg)_{\Omega_c}+a_{\Omega_c}(e_c^{n+1},\boldsymbol{v}_c^h)
-b(\boldsymbol{v}_h,\delta_c^{n+1})+b(e_c^{n+1},q^h)
-\frac{1}{\rho} \int_{\mathbb{I}} (P^h \phi_f(t_{n+1})-\phi_f^{n_k}) \boldsymbol{v}_c^h \cdot \boldsymbol{n}_d \mathrm{d}s\\
&+ \frac{\gamma}{\rho h}\int_{\mathbb{I}} \big(e_c^{n+1} - (P^h \boldsymbol{u}_f(t_{n+1})-\boldsymbol{u}_f^{n_k})\big) \cdot \boldsymbol{n}_d \boldsymbol{v}_c^h \cdot \boldsymbol{n}_d \mathrm{d}s\\
&= \bigg(d_t(P^h \boldsymbol{u}_c(t_{n+1})-\boldsymbol{u}_c(t_{n+1}))
+\frac{\boldsymbol{u}_c^n - \hat{\boldsymbol{u}}_c^n-\boldsymbol{u}_c(t_n)+\bar{\boldsymbol{u}}_c(t_n)}{\Delta t}+R_{tr}^{n+1},\boldsymbol{v}_c^h \bigg)_{\Omega_c},
\end{split}
\end{equation}

\begin{equation}\label{SS2}
\begin{split}
&\frac{\eta_m C_{mt}}{\rho} \bigg( \frac{\theta_m^{n_{k+1}}-\theta_m^{n_k}}{\Delta s},\psi_m^h\bigg)_{\Omega_d}
+\frac{1}{\rho} (\nabla \cdot e_m^{n_{k+1}},\psi_m^h)_{\Omega_d}
+\frac{1}{\rho} (\mu k_m^{-1} e_m^{n_{k+1}},\boldsymbol{v}_m^h)_{\Omega_d}\\
&-\frac{1}{\rho} (\theta_m^{n_{k+1}},\nabla \cdot \boldsymbol{v}_m^h)_{\Omega_d}
+\frac{\sigma k_m}{\rho \mu} \big(\theta_m^{n_{k+1}}-(P^h \phi_f(t_{n_{k+1}})-\phi_f^{n_k}),\psi_m^h \big)_{\Omega_d}\\
&=\frac{\eta_m C_{mt}}{\rho}\bigg( \frac{P^h \phi_m(t_{n_{k+1}})- P^h \phi_m(t_{n_k})}{\Delta s}
-\frac{\partial \phi_m}{\partial s}(t_{n_{k+1}}), \psi_m^h \bigg)_{\Omega_d},
\end{split}
\end{equation}
and
\begin{equation}\label{SS3}
\begin{split}
&\frac{\eta_f C_{ft}}{\rho} \bigg( \frac{\theta_f^{n_{k+1}}-\theta_f^{n_k}}{\Delta s},\psi_f^h\bigg)_{\Omega_d}
+\frac{1}{\rho} (\nabla \cdot e_f^{n_{k+1}},\psi_f^h)_{\Omega_d}
+\frac{1}{\rho} (\mu k_f^{-1} e_f^{n_{k+1}},\boldsymbol{v}_f^h)_{\Omega_d}\\
&-\frac{1}{\rho} (\theta_f^{n_{k+1}},\nabla \cdot \boldsymbol{v}_f^h)_{\Omega_d}\
+\frac{\sigma k_m}{\rho \mu} \big( \theta_f^{n_{k+1}}-(P^h \phi_m(t_{n_{k+1}})-\phi_m^{n_k}),\psi_f^h \big)_{\Omega_d}\\
&+ \frac{1}{\rho} \int_{\mathbb{I}} (P^h \phi_f(t_{n_{k+1}})-\phi_f^{n_k}) \boldsymbol{v}_f^h \cdot \boldsymbol{n}_d \mathrm{d}s
-\frac{\gamma}{\rho h} \int_{\mathbb{I}} \big( (P^h \boldsymbol{u}_c(t_{n_{k+1}})-S^{n_{k+1}})-e_f^{n_{k+1}}\big) \cdot \boldsymbol{n}_d \boldsymbol{v}_f^h \cdot \boldsymbol{n}_d \mathrm{d}s\\
&=\frac{\eta_f C_{ft}}{\rho} \bigg( \frac{P^h \phi_f(t_{n_{k+1}})- P_h \phi_f(t_{n_k})}{\Delta s}-
\frac{\partial \phi_f}{\partial s}(t_{n_{k+1}}),\psi_f^h \bigg)_{\Omega_d}.
\end{split}
\end{equation}

We below prove error convergence of solutions in sense of~$L^2$-norm and~$H^1$-seminorm~for different time steps in different region. The key of successful proof is to obtain the uniform~$L^{\infty}$-boundedness of~$\boldsymbol{u}_h^n$~at the assumption step.

\begin{theorem}\label{thm1}
Assume that~\eqref{CBF}~has a unique solution~$(\boldsymbol{u}_c(t_n),p_c(t_n);\boldsymbol{u}_f(t_n),\phi_f(t_n);\boldsymbol{u}_m(t_n),$\\
$\phi_m(t_n))$~satisfying the boundedness of assumption. There exists some positive constants~$\tau_1$\\
and~$h_1$~such that when~$\Delta t < \tau_1, h < h_1$~and~$\Delta t =\mathcal{O}(h^2)$, the solution of fully discrete decoupled modified characteristic scheme~\eqref{Step1},\eqref{Step2}~and~\eqref{Step3}~in the first large time interval $[0,t_{n_1}]$, for any~$0 \leq J \leq r$~satisfies
\begin{equation}\label{THMP1}
\begin{split}
&\max_{0 \leq J \leq r}\|e_c^{J}\|_0^2
+\frac{\nu \Delta t}{2} \sum_{n=0}^{J} \|\nabla e_c^{n}\|_0^2
\leq C(h^4 + \Delta t^2 + \Delta t^2 h^{-1}).
\end{split}
\end{equation}
On the other hand, for~$[0,T], 0 \leq l \leq M-1$, we obtain
\begin{equation}\label{THMP}
\begin{split}
&\max_{0 \leq l \leq M-1}\big( \|e_c^{n_{l+1}}\|_0^2 + \frac{\eta_f C_{ft}}{\rho} \|\theta_f^{n_{l+1}}\|_0^2
+ \frac{\eta_m C_{mt}}{\rho} \|\theta_m^{n_{l+1}}\|_0^2
\big)
+\nu \Delta t \sum_{k=0}^l \sum_{n=n_k}^{n_{k+1}-1}\|\nabla e_c^{n+1}\|_0^2\\
&+ \frac{2 \mu \Delta s}{\rho k_f}\sum_{k=0}^{l}\|e_f^{n_{k+1}}\|_0^2
+ \frac{2 \mu \Delta s}{\rho k_m}\sum_{k=0}^{l}\|e_m^{n_{k+1}}\|_0^2
\leq C(h^4 + \Delta t^2 + \Delta t^2 h^{-1}),
\end{split}
\end{equation}
and
\begin{equation}
\begin{split}
\Delta t  \sum_{k=0}^{M-1}\sum_{n=n_k+1}^{n_{k+1}}\|\boldsymbol{u}_c^n\|_{W^{1,\infty}}^2
+\max_{ n \in \{n_{k}+1,...,n_{k+1}\} , 0 \leq k \leq M-1} \|\boldsymbol{u}_c^n\|_{L^{\infty}}
\leq C_B.
\end{split}
\end{equation}

\end{theorem}

\begin{proof}
We give the proof of~\eqref{THMP1}~by mathematical induction. First of all, when~$J=0$, $e_c^0=0, \nabla e_c^0=0$. It's obvious that~\eqref{THMP1}~holds at the initial time step. When~$J=m, 1 \leq m \leq r-1$, assume that~\eqref{THMP1}~holds, i.e.,
\begin{equation}\label{THMP11}
\|e_c^{m}\|_0^2
+\frac{\nu \Delta t}{2} \sum_{n=0}^{m} \|\nabla e_c^{n}\|_0^2
\leq C(h^4 + \Delta t^2 + \Delta t^2 h^{-1}).
\end{equation}
Using the discrete Sobolev inequality~\eqref{DS}, the properties~\eqref{Assume}~and the regularity assumption in (??), when~$\Delta t < \tau_1, h<h_1$~and~$\Delta t=\mathcal{O}(h^2)$, if~$d=2$, we have
\begin{equation}\label{bound1}
\begin{split}
\|\boldsymbol{u}_{c}^m\|_{L^{\infty}} &\leq \|e_{c}^m\|_{L^{\infty}} + \|P^h \boldsymbol{u}_c(t_m)\|_{L^{\infty}}
\leq C_{DS}(1+|\ln(h)|)^{1/2} \|e_{c}^m\|_1 + C \| \boldsymbol{u}_c(t_m)\|_2\\
&\leq C(1+|\ln(h)|)^{1/2} (\Delta t^{-1/2} h^2+\Delta t^{1/2} + \Delta t^{1/2} h^{-1/2}) + C\\
&\leq C_B.
\end{split}
\end{equation}
Actually, $h |\ln(h)| \rightarrow 0(h\rightarrow 0)$. If~$d=3$, we obtain
\begin{equation}\label{bound2}
\begin{split}
\|\boldsymbol{u}_{c}^m\|_{L^{\infty}} &\leq \|e_{c}^m\|_{L^{\infty}} + \|P^h \boldsymbol{u}_c(t_m)\|_{L^{\infty}}
\leq C_{DS} h^{-1/2} \|e_{c}^m\|_1 + C\|\boldsymbol{u}_c(t_m)\|_2\\
&\leq C h^{-1/2} (\Delta t^{-1/2}h^2 + \Delta t^{1/2} + \Delta t^{1/2} h^{-1/2}) + C \leq C_B.
\end{split}
\end{equation}
In addition, by the properties of~\eqref{Assume}, inverse inequality~\eqref{inverse2}, imbedding theorem and the regularity assumption in~\eqref{suppose}, we obtain
\begin{equation}\label{Boundu}
\begin{split}
\Delta t \|\boldsymbol{u}_c^m\|_{W^{1,\infty}} &\leq \Delta t(\|e_c^m\|_{W^{1,\infty}}+\|P^h \boldsymbol{u}_c(t_m)\|_{W^{1,\infty}})\\
&\leq \Delta t (Ch^{-3/2}\|\nabla e_c^m\|_0 + C\|\boldsymbol{u}_c(t_m)\|_{W^{2,d^{\ast}}})\\
&\leq C(\Delta t^{1/2} h^{1/2} + \Delta t^{3/2} h^{-3/2} + \Delta t^{3/2} h^{-2} + \Delta t )\\
&\leq \frac{1}{4},
\end{split}
\end{equation}
where~$d^{\ast} > D$.

When~$J=m+1$, taking~$\boldsymbol{v}_c^h=2 \Delta t e_c^{n+1}$~and~$q_h=2 \Delta t \delta_c^{n+1}$~in~\eqref{SS1}, sum over~$n=0, 1,...,m(1 \leq m\leq r-1)$, we have
\begin{align*}
&\|e_c^{m+1}\|_0^2 - \|e_c^{0}\|_0^2 + \sum_{n=0}^{m}\|e_c^{n+1}-e_c^n\|_0^2
+2 \nu \Delta t \sum_{n=0}^{m} \|\nabla e_c^{n+1}\|_0^2
+\frac{2 \alpha \nu \Delta t}{\sqrt{k}} \sum_{n=0}^{m} \|P_{\tau} (e_c^{n+1})\|_{L^2(\mathbb{I})}^2\\
&= \sum_{n=0}^{m} \bigg( \Psi_{c,t}^{n+1}+
\frac{\boldsymbol{u}_c^n - \hat{\boldsymbol{u}}_c^n - \boldsymbol{u}_c(t_n) + \bar{\boldsymbol{u}}_c(t_n)}{\Delta t} + R_{tr}^{n+1},2\Delta t e_c^{n+1} \bigg)_{\Omega_c}\\
&+ \frac{2 \Delta t}{\rho} \sum_{n=0}^{m} \int_{\mathbb{I}} (P^h \phi_f(t_{n+1})-\phi_f^{0}) e_c^{n+1} \cdot \boldsymbol{n}_d \mathrm{d}s \\
&- \frac{2 \gamma \Delta t}{\rho h} \sum_{n=0}^{m} \int_{\mathbb{I}} \big( e_c^{n+1} - (P^h \boldsymbol{u}_f(t_{n+1}) - \boldsymbol{u}_f^{0})\big) \cdot \boldsymbol{n}_d e_c^{n+1} \cdot \boldsymbol{n}_d \mathrm{d}s \\
&:= A_1 + A_2 + A_3.
\end{align*}
Note that in the first large time interval, $n_k=n_0=0$.
The first term
\begin{equation*}
\begin{split}
A_1 &=2\Delta t \sum_{n=0}^{m}(\Psi_{c,t}^{n+1},e_c^{n+1})_{\Omega_c}
+ 2\Delta t \sum_{n=0}^{m} \bigg(\frac{\boldsymbol{u}_c^n - \hat{\boldsymbol{u}}_c^n-\boldsymbol{u}_c(t_n)+\bar{\boldsymbol{u}}_c(t_n)}{\Delta t},e_c^{n+1}\bigg)_{\Omega_c}\\
&~~~~+ 2\Delta t \sum_{n=0}^{m} (R_{tr}^{n+1}, e_c^{n+1})_{\Omega_c}\\
&:=A_{11}+A_{12}+A_{13}.
\end{split}
\end{equation*}
By the~Cauchy-Schwarz~inequality, Poincar$\acute{\mathrm{e}}$-Friedrichs~inequality, Young~inequality and some techniques in part B of 4.1, we estimate
\begin{equation*}
\begin{split}
A_{11} &\leq 2C_{PF} \Delta t \sum_{n=0}^{m} \|\Psi_{c,t}^{n+1}\|_0 \|\nabla e_c^{n+1}\|_0\\
&\leq \frac{6C_{PF}^2}{\nu}\int_{0}^{t_{n_1}} \bigg\|(P^h-I)\frac{\partial \boldsymbol{u}_c}{\partial t}(t)\bigg\|_0^2 \mathrm{d}t
+ \frac{\nu \Delta t}{6} \sum_{n=0}^{m} \|\nabla e_c^{n+1}\|_0^2.
\end{split}
\end{equation*}
At the same time, using the~H$\mathrm{\ddot{o}}$lder~inequality, Poincar$\acute{\mathrm{e}}$-Friedrichs~inequality, Lemma~\ref{Lem2}, the trace inequality, inverse inequality, Sobolev~interpolation formulas, imbedding theorem and the boundedness of~\eqref{bound1}~\eqref{bound2}~\eqref{Boundu}, we show that
\begin{align*}
&2\Delta t \bigg(\frac{\boldsymbol{u}_c^n - \hat{\boldsymbol{u}}_c^n-\boldsymbol{u}_c(t_n)+\bar{\boldsymbol{u}}_c(t_n)}{\Delta t},e_c^{n+1}\bigg)_{\Omega_c}\\
&\leq 2C_{PF} \|P^h \boldsymbol{u}_c(t_n)-\boldsymbol{u}_c(t_n)-(P^h \bar{\boldsymbol{u}}_c(t_n)-\bar{\boldsymbol{u}}_c(t_n))\|_{-1} \|\nabla e_c^{n+1}\|_0\\
&+2 \|P^h \bar{\boldsymbol{u}}_c(t_n)-\bar{\boldsymbol{u}}_c(t_n)-(P^h \hat{\boldsymbol{u}}_c(t_n)-\hat{\boldsymbol{u}}_c(t_n))\|_{L^1}\|e_c^{n+1}\|_{L^{\infty}}\\
&+2 \|\hat{e}_c^n-e_c^n\|_0 \|e_c^{n+1}\|_0
+2 \|\hat{\boldsymbol{u}}_c(t_n)-\bar{\boldsymbol{u}}_c(t_n)\|_{L^{6/5}} \|e_c^{n+1}\|_{L^6}\\
&\leq 2 C_{PF} C_n \Delta t \|P^h \boldsymbol{u}_c(t_n)-\boldsymbol{u}_c(t_n)\|_0 \|\boldsymbol{u}_c(t_n)\|_{W^{1,4}}\|\nabla e_c^{n+1}\|_0\\
&+ 2C_N C_{inv} h^{-1/2} \Delta t \|P^h \boldsymbol{u}_c(t_n)-\boldsymbol{u}_c(t_n)\|_1
\|\boldsymbol{u}_c(t_n)-\boldsymbol{u}_c^n\|_0 \|e_c^{n+1}\|_{L^6}\\
&+ 2C_N C_{PF} \Delta t \|\nabla e_c^n\|_0 \|\boldsymbol{u}_c^n\|_{L^{\infty}} \|e_c^{n+1}\|_0\\
&+ 2C_N C_S C_{PF} \Delta t \|\boldsymbol{u}_c(t_n)\|_{W^{1,3}} \|\boldsymbol{u}_c^n - \boldsymbol{u}_c(t_n)\|_0 \|\nabla e_c^{n+1}\|_0\\
&\leq 2 C_{PF} C_n C_{pc} C_H h^2 \Delta t \|\nabla e_c^{n+1}\|_0\\
&+ 2 C_N C_S C_{inv} C_{pc} C_{PF} h^{1/2} \Delta t (C_{pc} h^2 + \|e_c^n\|_0)\|\nabla e_c^{n+1}\|_0\\
&+ 2 C_N C_{PF} C_B \Delta t  \|\nabla e_c^n\|_0 \|e_c^{n+1}\|_0
+ 2 C_N C_S C_{PF} C_H \Delta t (\|e_c^n\|_0 + C_{pc}h^2) \|\nabla e_c^{n+1}\|_0\\
&\leq \frac{\nu \Delta t}{6}\|\nabla e_c^{n+1}\|_0^2 + \frac{\nu \Delta t}{2}\|\nabla e_c^n\|_0^2
+ \frac{2 C_N^2 C_{PF}^2 C_B^2 \Delta t }{\nu} \|e_c^{n+1}\|_0^2
+ \frac{30 C_N^2 C_S^2 C_{PF}^2 C_H^2 \Delta t}{\nu} \|e_c^n\|_0^2\\
&+ Ch^4 \Delta t.
\end{align*}
Therefore
\begin{align*}
A_{12} &\leq \frac{\nu \Delta t}{6}\sum_{n=0}^{m} \|\nabla e_c^{n+1}\|_0^2 + \frac{\nu \Delta t}{2} \sum_{n=0}^{m} \|\nabla e_c^n\|_0^2
+ \frac{2 C_N^2 C_{PF}^2 C_B^2 \Delta t }{\nu} \sum_{n=0}^{m}\|e_c^{n+1}\|_0^2\\
&+ \frac{30 C_N^2 C_S^2 C_{PF}^2 C_H^2 \Delta t}{\nu} \sum_{n=0}^{m} \|e_c^n\|_0^2
+ Ch^4 \Delta t.
\end{align*}
Analogous to~$A_{11}$, using Lemma~\ref{Lem1}, we deduce that
\begin{equation*}
A_{13} \leq \frac{6C_{PF}^2 \Delta t}{\nu} \sum_{n=0}^{m} \|R_{tr}^{n+1}\|_0^2 + \frac{\nu \Delta t}{6} \sum_{n=0}^{m} \|\nabla e_c^{n+1}\|_0^2.
\end{equation*}
For the first term
\begin{equation*}
\begin{split}
A_1 &\leq \frac{\nu \Delta t}{2}\sum_{n=0}^{m} (\|\nabla e_c^{n+1}\|_0^2+\|\nabla e_c^n\|_0^2)
+ \frac{2 C_N^2 C_{PF}^2 C_B^2 \Delta t}{\nu} \sum_{n=0}^{m} \|e_c^{n+1}\|_0^2\\
&+ \frac{30 C_N^2 C_S^2 C_{PF}^2 C_H^2 \Delta t}{\nu} \sum_{n=0}^{m} \|e_c^n\|_0^2
+\frac{6C_{PF}^2 \Delta t}{\nu} \sum_{n=0}^{m} \|R_{tr}^{n+1}\|_0^2\\
&+\frac{6C_{PF}^2}{\nu}\int_{0}^{t_{n_{1}}} \bigg\|(P^h-I)\frac{\partial \boldsymbol{u}_c}{\partial t}(t)\bigg\|_0^2 \mathrm{d}t
+Ch^4 \Delta t.
\end{split}
\end{equation*}
For the second term, $A_2$~is bounded by the trace inequality and trace inverse inequality,
\begin{equation*}
\begin{split}
A_2 &=\frac{2 \Delta t}{\rho} \sum_{n=0}^{m} \int_{\mathbb{I}} (P^h \phi_f(t_{n+1})-\phi_f^{0}) e_c^{n+1} \cdot \boldsymbol{n}_d \mathrm{d}s\\
&\leq \frac{2\Delta t}{\rho} \sum_{n=0}^{m} \|P^h \phi_f(t_{n+1})-P^h \phi_f(0)\|_{L^2(\mathbb{I})} \|e_c^{n+1} \cdot \boldsymbol{n}_d\|_{L^2(\mathbb{I})}
+\frac{2\Delta t}{\rho} \sum_{n=0}^{m} \|\theta_f^{0}\|_{L^2(\mathbb{I})} \|e_c^{n+1} \cdot \boldsymbol{n}_d\|_{L^2(\mathbb{I})}\\
&\leq \frac{2 r C_T \tilde{C}_{inv} C_{PF}^{1/2} \Delta t}{\rho h^{\frac{1}{2}}} \sum_{n=0}^{m} \|P^h \phi_f(t_{n+1}) - P^h \phi_f(t_n)\|_0 \|\nabla e_c^{n+1}\|_0
+ \frac{2 C_T \tilde{C}_{inv} C_{PF}^{1/2} \Delta t}{\rho h^{\frac{1}{2}}}\sum_{n=0}^{m} \|\theta_f^{0}\|_0 \|\nabla e_c^{n+1}\|_0\\
&\leq \frac{4 r^2 C_T^2 \tilde{C}_{inv}^2 C_{PF} \Delta t^2}{\rho^2 \nu h} \int_{0}^{t_{n_1}} \bigg\|\frac{\partial P^h \phi_f}{\partial t}(t) \bigg\|_0^2 \mathrm{d}t
+\frac{\nu \Delta t}{2}\sum_{n=0}^{m} \|\nabla e_c^{n+1}\|_0^2
+\frac{4C_T^2 \tilde{C}_{inv}^2 C_{PF} \Delta t}{\rho^2 \nu h} \sum_{n=0}^{m} \|\theta_f^{0}\|_0^2.
\end{split}
\end{equation*}
For the third term, $A_3$~does the identity transformation,
\begin{equation*}
\begin{split}
A_3 &= - \frac{2 \gamma \Delta t}{\rho h} \sum_{n=0}^{m} \int_{\mathbb{I}} \big( e_c^{n+1} - (P^h \boldsymbol{u}_f(t_{n+1}) - \boldsymbol{u}_f^{0})\big) \cdot \boldsymbol{n}_d e_c^{n+1} \cdot \boldsymbol{n}_d \mathrm{d}s\\
&= - \frac{2 \gamma \Delta t}{\rho h} \sum_{n=0}^{m} \int_{\mathbb{I}} \big( e_c^{n+1} - e_f^{0}
-(P^h \boldsymbol{u}_f(t_{n+1})-P^h \boldsymbol{u}_f(0)) \big) \cdot \boldsymbol{n}_d (e_c^{n+1}-e_f^{0}+e_f^{0}) \mathrm{d}s\\
&= - \frac{2 \gamma \Delta t}{\rho h} \sum_{n=0}^{m} \|e_c^{n+1}-e_f^{0}\|_{L^2(\mathbb{I})}^2
+\frac{2 \gamma \Delta t}{\rho h} \sum_{n=0}^{m} \int_{\mathbb{I}} \big( P^h \boldsymbol{u}_f(t_{n+1})-P^h \boldsymbol{u}_f(0)\big) \cdot \boldsymbol{n}_d e_c^{n+1} \cdot \boldsymbol{n}_d \mathrm{d}s\\
&+\frac{2 \gamma \Delta t}{\rho h} \sum_{n=0}^{m}\int_{\mathbb{I}} (e_f^{0}-e_c^{n+1}) \cdot \boldsymbol{n}_d e_f^{0} \cdot \boldsymbol{n}_d \mathrm{d}s,
\end{split}
\end{equation*}
where using~H$\mathrm{\ddot{o}}$lder inequality, the general trace inequality, Young inequality, the properties of~$H(\mathrm{div})$~space and the divergence free condition, we arrive at
\begin{align*}
&~~~~\frac{2 \gamma \Delta t}{\rho h} \sum_{n=0}^{m} \int_{\mathbb{I}} \big( P^h \boldsymbol{u}_f(t_{n+1})-P^h \boldsymbol{u}_f(0)\big) \cdot \boldsymbol{n}_d e_c^{n+1} \cdot \boldsymbol{n}_d \mathrm{d}s\\
&\leq \frac{2 \gamma \Delta t}{\rho h} \sum_{n=0}^{m} \|\big( P^h \boldsymbol{u}_f(t_{n+1}) - P^h \boldsymbol{u}_f(0)\big) \cdot \boldsymbol{n}_d\|_{H^{\frac{1}{2}}(\partial \Omega_c)} \|e_c^{n+1} \cdot \boldsymbol{n}_d\|_{H^{-\frac{1}{2}}(\partial \Omega_c)}\\
&\leq \frac{2r\gamma C_T C_{div}\Delta t}{\rho h}\sum_{n=0}^{m}\|P^h \boldsymbol{u}_f(t_{n+1})-P^h \boldsymbol{u}_f(t_n)\|_1 \|e_c^{n+1}\|_{H(\mathrm{div})}\\
&\leq \frac{r^2 \gamma^2 C_T^2  \Delta t^2}{\rho^2 h} \int_{0}^{t_{n_{1}}} \bigg\|\frac{\partial P^h \boldsymbol{u}_f}{\partial t}(t)\bigg\|_1^2 \mathrm{d}t
+ \frac{C_{div}^2 \Delta t}{h}\sum_{n=0}^{m}\|e_c^{n+1}\|_0^2.
\end{align*}
Therefore, $A_3$~follows that
\begin{align*}
A_3 &\leq - \frac{2 \gamma \Delta t}{\rho h} \sum_{n=0}^{m} \|e_c^{n+1}-e_f^{0}\|_{L^2(\mathbb{I})}^2
+\frac{r^2 \gamma^2 C_T^2  \Delta t^2}{\rho^2 h} \int_{0}^{t_{n_{1}}} \bigg\|\frac{\partial P^h \boldsymbol{u}_f}{\partial t}(t)\bigg\|_1^2 \mathrm{d}t
+ \frac{C_{div}^2 \Delta t}{h}\sum_{n=0}^{m}\|e_c^{n+1}\|_0^2\\
&+\frac{\gamma \Delta t}{\rho h} \sum_{n=0}^{m}(\|e_f^{0}\|_{L^2(\mathbb{I})}^2 - \|e_c^{n+1} \cdot \boldsymbol{n}_d\|_{L^2(\mathbb{I})}^2 + \|e_c^{n+1}-e_f^{0}\|_{L^2(\mathbb{I})}^2).
\end{align*}
Collect the estimates of~$A_1$, $A_2$~and~$A_3$~to obtain
\begin{align*}
&\|e_c^{m+1}\|_0^2 + \sum_{n=0}^{m}\|e_c^{n+1}-e_c^n\|_0^2
+\frac{\nu \Delta t}{2} \sum_{n=0}^{m} \|\nabla e_c^{n+1}\|_0^2
+\frac{2 \alpha \nu \Delta t}{\sqrt{k}} \sum_{n=0}^{m} \|P_{\tau} (e_c^{n+1})\|_{L^2(\mathbb{I})}^2\\
&+ \frac{\gamma \Delta t}{\rho h} \sum_{n=0}^{m} \|e_c^{n+1}-e_f^{0}\|_{L^2(\mathbb{I})}^2
+ \frac{\gamma \Delta t}{\rho h} \sum_{n=0}^{m} \|e_c^{n+1} \cdot \boldsymbol{n}_d\|_{L^2(\mathbb{I})}^2\\
&\leq \Delta t \sum_{n=0}^{m} \bigg( \frac{2 C_N^2 C_{PF}^2 C_B^2}{\nu}  \|e_c^{n+1}\|_0^2
+ \frac{30 C_N^2 C_S^2 C_{PF}^2 C_H^2}{\nu} \|e_c^n\|_0^2
+ \frac{C_{div}^2}{h}\|e_c^{n+1}\|_0^2 \bigg)\\
&+\frac{6C_{PF}^2 \Delta t}{\nu} \sum_{n=0}^{m} \|R_{tr}^{n+1}\|_0^2
+\frac{6C_{PF}^2}{\nu}\int_{0}^{t_{n_{1}}} \bigg\|(P^h-I)\frac{\partial \boldsymbol{u}_c}{\partial t}(t)\bigg\|_0^2 \mathrm{d}t +Ch^4 \Delta t\\
&+ \frac{4 r^2 C_T^2 \tilde{C}_{inv}^2 C_{PF} \Delta t^2}{\rho^2 \nu h} \int_{0}^{t_{n_1}} \bigg\|\frac{\partial P^h \phi_f}{\partial t}(t) \bigg\|_0^2 \mathrm{d}t
+\frac{4C_T^2 \tilde{C}_{inv}^2 C_{PF} \Delta t}{\rho^2 \nu h} \sum_{n=0}^{m} \|\theta_f^{0}\|_0^2\\
&+\frac{r^2 \gamma^2 C_T^2  \Delta t^2}{\rho^2 h} \int_{0}^{t_{n_{1}}} \bigg\|\frac{\partial P^h \boldsymbol{u}_f}{\partial t}(t)\bigg\|_1^2 \mathrm{d}t
+\frac{\gamma \Delta t}{\rho h} \sum_{n=0}^{m}\|e_f^{0}\|_{L^2(\mathbb{I})}^2
+\|e_c^{0}\|_0^2.
\end{align*}
By the discrete Growall lemma, Lemma 4.1 and~$\Delta t=\mathcal{O}(h^2)$, when~$\kappa_n \Delta t \leq \frac{1}{2}$, where~$\kappa_n=\frac{2 C_N^2 C_{PF}^2 C_B^2}{\nu}
+ \frac{30 C_N^2 C_S^2 C_{PF}^2 C_H^2}{\nu}+ \frac{C_{div}^2}{h}$, then
\begin{equation*}
\|e_c^{m+1}\|_0^2
+\frac{\nu \Delta t}{2} \sum_{n=0}^{m} \|\nabla e_c^{n+1}\|_0^2
\leq C(h^4+\Delta t^2 + \Delta t^2 h^{-1}).
\end{equation*}
In conclusion, in the first large time interval~$[0,t_{n_1}]$, when~$0 \leq J \leq r$, it yields that
\begin{equation*}
\begin{split}
\max_{0 \leq J \leq r}\|e_c^{J}\|_0^2
+\frac{\nu \Delta t}{2} \sum_{n=0}^{J} \|\nabla e_c^{n}\|_0^2
\leq C(h^4 + \Delta t^2 + \Delta t^2 h^{-1}).
\end{split}
\end{equation*}

Taking~$\boldsymbol{v}_m^h=2\Delta s e_m^{n_{k+1}}, \psi_m^h=2 \Delta s \theta_m^{n_{k+1}}, \boldsymbol{v}_f^h=2 \Delta s e_f^{n_{k+1}}$~and~$\psi_f^h=2\Delta s \theta_f^{n_{k+1}}$~in~\eqref{SS2}~and ~\eqref{SS3}, sum over two equations
\begin{align}
&\frac{\eta_d C_{dt}}{\rho} (\|\theta_d^{n_{k+1}}\|_0^2 - \|\theta_d^{n_k}\|_0^2 + \|\theta_d^{n_{k+1}}-\theta_d^{n_k}\|_0^2)
+\frac{2\mu \Delta s}{\rho k_m}\|e_m^{n_{k+1}}\|_0^2
+\frac{2\mu \Delta s}{\rho k_f}\|e_f^{n_{k+1}}\|_0^2 \nonumber\\
&+\frac{\sigma k_m \Delta s}{\rho \mu}(\|\theta_m^{n_{k+1}}\|_0^2 - \|\theta_f^{n_k}\|_0^2 + \|\theta_m^{n_{k+1}} - \theta_f^{n_k}\|_0^2)
+\frac{\sigma k_m \Delta s}{\rho \mu}(\|\theta_f^{n_{k+1}}\|_0^2 - \|\theta_m^{n_k}\|_0^2 +
\|\theta_f^{n_{k+1}}-\theta_m^{n_k}\|_0^2)\nonumber\\
&=\frac{\eta_d C_{dt}}{\rho} (\psi_{d,s}^{n_{k+1}},2\Delta s \theta_d^{n_{k+1}})_{\Omega_d}
+\frac{2 \gamma \Delta s}{\rho h} \int_{\mathbb{I}} \big( (P^h \boldsymbol{u}_c(t_{n_{k+1}})-S^{n_{k+1}})-e_f^{n_{k+1}} \big) \cdot \boldsymbol{n}_d e_f^{n_{k+1}} \cdot \boldsymbol{n}_d \mathrm{d}s \nonumber\\
&+\bigg[ \frac{2 \sigma k_m \Delta s}{\rho \mu}(P^h \phi_m(t_{n_{k+1}})-P^h \phi_m(t_{n_k}),\theta_f^{n_{k+1}})_{\Omega_d}
+\frac{2\sigma k_m \Delta s}{\rho \mu}(P^h \phi_f(t_{n_{k+1}})-P^h \phi_f(t_{n_k}),\theta_m^{n_{k+1}})_{\Omega_d}\bigg] \nonumber\\
&-\frac{2\Delta s}{\rho} \int_{\mathbb{I}} (P^h \phi_f(t_{n_{k+1}})-\phi_f^{n_k}) e_f^{n_{k+1}} \cdot \boldsymbol{n}_d \mathrm{d}s \nonumber\\
&:=\sum_{i=1}^{4} B_i, \label{B}
\end{align}
where~$\psi_{d,s}^{n_{k+1}}=\psi_{d,s,1}^{n_{k+1}}+\psi_{d,s,2}^{n_{k+1}}$.

For the first term, use the previous techniques
\begin{equation*}
\begin{split}
B_1&=\frac{\eta_d C_{dt}}{\rho} (\psi_{d,s}^{n_{k+1}},2\Delta s \theta_d^{n_{k+1}})_{\Omega_d}\\
&=\frac{\eta_f C_{ft}}{\rho} \bigg(\Psi_{f,s}^{n_{k+1}},2\Delta s \theta_f^{n_{k+1}}\bigg)_{\Omega_d}
+\frac{\eta_m C_{mt}}{\rho} \bigg(\Psi_{m,s}^{n_{k+1}},2\Delta s \theta_m^{n_{k+1}}\bigg)_{\Omega_d}\\
&\leq \frac{2\eta_f C_{ft}\Delta s}{\rho} \|\Psi_{f,s}^{n_{k+1}}\|_0^2 +\frac{\eta_f C_{ft}\Delta s}{2\rho} \|\theta_f^{n_{k+1}}\|_0^2
+ \frac{2\eta_m C_{mt}\Delta s}{\rho} \|\Psi_{m,s}^{n_{k+1}}\|_0^2 +\frac{\eta_m C_{mt}\Delta s}{2\rho} \|\theta_m^{n_{k+1}}\|_0^2\\
&\leq \frac{2\eta_f C_{ft}\Delta s}{\rho}(\|\Psi_{f,s,1}^{n_{k+1}}\|_0^2+\|\Psi_{f,s,2}^{n_{k+1}}\|_0^2 ) +\frac{\eta_f C_{ft}\Delta s}{2\rho} \|\theta_f^{n_{k+1}}\|_0^2\\
&~~~~+ \frac{2\eta_m C_{mt}\Delta s}{\rho} (\|\Psi_{m,s,1}^{n_{k+1}}\|_0^2+\|\Psi_{m,s,2}^{n_{k+1}}\|_0^2 ) +\frac{\eta_m C_{mt}\Delta s}{2\rho} \|\theta_m^{n_{k+1}}\|_0^2\\
&\leq \frac{2\eta_f C_{ft}\Delta s^2}{3\rho}\int_{t_{n_k}}^{t_{n_{k+1}}} \bigg\|\frac{\partial^2 \phi_f}{\partial t^2} (t) \bigg\|_0^2 \mathrm{d}t
+\frac{2\eta_f C_{ft}}{\rho} \int_{t_{n_k}}^{t_{n_{k+1}}} \bigg\|(P^h - I) \frac{\partial \phi_f}{\partial t}(t)\bigg\|_0^2 \mathrm{d}t \\
&+ \frac{2\eta_m C_{mt}\Delta s^2}{3\rho}\int_{t_{n_k}}^{t_{n_{k+1}}} \bigg\|\frac{\partial^2 \phi_m}{\partial t^2}(t) \bigg\|_0^2 \mathrm{d}t
+\frac{2\eta_m C_{mt}}{\rho} \int_{t_{n_k}}^{t_{n_{k+1}}} \bigg\|(P^h - I) \frac{\partial \phi_m}{\partial t}(t)\bigg\|_0^2 \mathrm{d}t \\
&+\frac{\eta_f C_{ft} \Delta s}{2\rho} \|\theta_f^{n_{k+1}}\|_0^2
+\frac{\eta_m C_{mt} \Delta s}{2\rho} \|\theta_m^{n_{k+1}}\|_0^2.
\end{split}
\end{equation*}
For the second term
\begin{align*}
B_2&=\frac{2 \gamma \Delta s}{\rho h} \int_{\mathbb{I}} \big( (P^h \boldsymbol{u}_c(t_{n_{k+1}})-S^{n_{k+1}})-e_f^{n_{k+1}} \big) \cdot \boldsymbol{n}_d e_f^{n_{k+1}} \cdot \boldsymbol{n}_d \mathrm{d}s\\
&=\frac{2\gamma \Delta t}{\rho h} \sum_{n=n_k}^{n_{k+1}-1}\int_{\mathbb{I}} (P^h \boldsymbol{u}_c(t_{n_{k+1}})-P^h \boldsymbol{u}_c(t_{n+1}))\cdot \boldsymbol{n}_d e_f^{n_{k+1}} \cdot \boldsymbol{n}_d \mathrm{d}s\\
&+ \frac{2\gamma \Delta t}{\rho h}\sum_{n=n_k}^{n_{k+1}-1}\int_{\mathbb{I}} (e_c^{n+1}-e_f^{n_{k+1}})\cdot \boldsymbol{n}_d e_f^{n_{k+1}}\cdot \boldsymbol{n}_d \mathrm{d}s\\
&:= B_{21} + B_{22}.
\end{align*}
Using Cauchy-Schwarz inequality and trace inequality,
\begin{align*}
B_{21}&=\frac{2\gamma \Delta t}{\rho h} \sum_{n=n_k}^{n_{k+1}-1}\int_{\mathbb{I}} (P^h \boldsymbol{u}_c(t_{n_{k+1}})-P^h \boldsymbol{u}_c(t_{n+1}))\cdot \boldsymbol{n}_d e_f^{n_{k+1}} \cdot \boldsymbol{n}_d \mathrm{d}s\\
&\leq \frac{4 \gamma r^2 C_T^2 \Delta t^2}{\rho h} \int_{t_{n_k}}^{t_{n_{k+1}}} \bigg\| \frac{\partial P^h \boldsymbol{u}_c}{\partial t}\bigg\|_1^2 \mathrm{d}t
+\frac{\gamma \Delta t}{4\rho h} \sum_{n=n_k}^{n_{k+1}-1} \|e_f^{n_{k+1}} \cdot \boldsymbol{n}_d\|_{L^2(\mathbb{I})}^2.
\end{align*}
It is easy to see~$2(a-b,b) \leq (a,a) - (b,b)$, we have
\begin{equation*}
\begin{split}
B_{22}&=\frac{2\gamma \Delta t}{\rho h}\sum_{n=n_k}^{n_{k+1}-1}\int_{\mathbb{I}} (e_c^{n+1}-e_f^{n_{k+1}})\cdot \boldsymbol{n}_d e_f^{n_{k+1}}\cdot \boldsymbol{n}_d \mathrm{d}s\\
&\leq \frac{\gamma \Delta t}{\rho h} \sum_{n=n_k}^{n_{k+1}-1} (\|e_c^{n+1}\|_{L^2(\mathbb{I})}^2-\|e_f^{n_{k+1}} \cdot \boldsymbol{n}_d\|_{L^2(\mathbb{I})}^2).
\end{split}
\end{equation*}
Therefore,
\begin{equation*}
\begin{split}
B_{2} &\leq \frac{4 \gamma r^2 C_T^2 \Delta t^2}{\rho h} \int_{t_{n_k}}^{t_{n_{k+1}}} \bigg\| \frac{\partial P^h \boldsymbol{u}_c}{\partial t}\bigg\|_1^2 \mathrm{d}t
+\frac{\gamma \Delta t}{4\rho h} \sum_{n=n_k}^{n_{k+1}-1} \|e_f^{n_{k+1}} \cdot \boldsymbol{n}_d\|_{L^2(\mathbb{I})}^2\\
&+\frac{\gamma \Delta t}{\rho h} \sum_{n=n_k}^{n_{k+1}-1} (\|e_c^{n+1}\|_{L^2(\mathbb{I})}^2-\|e_f^{n_{k+1}} \cdot \boldsymbol{n}_d\|_{L^2(\mathbb{I})}^2).
\end{split}
\end{equation*}
For the third term
\begin{equation*}
\begin{split}
B_3&=\frac{2 \sigma k_m \Delta s}{\rho \mu}(P^h \phi_m(t_{n_{k+1}})-P^h \phi_m(t_{n_k}),\theta_f^{n_{k+1}})_{\Omega_d}
+\frac{2\sigma k_m \Delta s}{\rho \mu}(P^h \phi_f(t_{n_{k+1}})-P^h \phi_f(t_{n_k}),\theta_m^{n_{k+1}})_{\Omega_d}\\
&:= B_{31} + B_{32}.
\end{split}
\end{equation*}
By~Cauchy-Schwarz~inequality and the~Young~inequality, we have
\begin{equation*}
\begin{split}
B_{31}&=\frac{2\sigma k_m \Delta s}{\rho \mu}(P^h \phi_m(t_{n_{k+1}})-P^h \phi_m(t_{n_k}),\theta_f^{n_{k+1}})_{\Omega_d}\\
&\leq \frac{2 \sigma^2 k_m^2 \Delta s^2}{\eta_{f} C_{ft} \rho \mu^2}\int_{t_{n_k}}^{t_{n_{k+1}}}\bigg\|\frac{\partial P^h \phi_m}{\partial t}(t)\bigg\|_0^2 \mathrm{d}t
+ \frac{\eta_{f} C_{ft} \Delta s}{2 \rho}\|\theta_f^{n_{k+1}}\|_0^2.
\end{split}
\end{equation*}
Similarly available,
\begin{equation*}
\begin{split}
B_{32}&=\frac{2\sigma k_m \Delta s}{\rho \mu}(P^h \phi_f(t_{n_{k+1}})-P^h \phi_f(t_{n_k}),\theta_m^{n_{k+1}})_{\Omega_d}\\
&\leq \frac{2 \sigma^2 k_m^2 \Delta s^2}{\eta_{m} C_{mt}\rho \mu^2} \int_{t_{n_k}}^{t_{n_{k+1}}}\bigg\| \frac{\partial P^h \phi_{f}}{\partial t}(t) \bigg\|_0^2 \mathrm{d}t
+\frac{\eta_{m} C_{mt} \Delta s}{2 \rho}\|\theta_m^{n_{k+1}}\|_0^2.
\end{split}
\end{equation*}
We can conclude that
\begin{equation*}
\begin{split}
B_3 &\leq \frac{2 \sigma^2 k_m^2 \Delta s^2}{\eta_{f} C_{ft} \rho \mu^2}\int_{t_{n_k}}^{t_{n_{k+1}}}\bigg\|\frac{\partial P^h \phi_m}{\partial t}(t)\bigg\|_0^2 \mathrm{d}t
+ \frac{\eta_{f} C_{ft}\Delta s}{2 \rho}\|\theta_f^{n_{k+1}}\|_0^2\\
&+\frac{2 \sigma^2 k_m^2 \Delta s^2}{\eta_{m} C_{mt}\rho \mu^2} \int_{t_{n_k}}^{t_{n_{k+1}}}\bigg\| \frac{\partial P^h \phi_{f}}{\partial t}(t) \bigg\|_0^2 \mathrm{d}t
+\frac{\eta_{m} C_{mt} \Delta s}{2 \rho}\|\theta_m^{n_{k+1}}\|_0^2.
\end{split}
\end{equation*}
For the fourth term
\begin{equation*}
\begin{split}
B_4 &=-\frac{2\Delta s}{\rho} \int_{\mathbb{I}} (P^h \phi_f(t_{n_{k+1}})-\phi_f^{n_k}) e_f^{n_{k+1}} \cdot \boldsymbol{n}_d \mathrm{d}s\\
&=-\frac{2\Delta s}{\rho} \int_{\mathbb{I}}(P^h \phi_f(t_{n_{k+1}})-P^h \phi_f(t_{n_k})) e_f^{n_{k+1}} \cdot \boldsymbol{n}_d \mathrm{d}s\\
&-\frac{2 \Delta s}{\rho} \int_{\mathbb{I}}\theta_f^{n_k} e_f^{n_{k+1}} \cdot \boldsymbol{n}_d \mathrm{d}s\\
&:= B_{41} + B_{42}.
\end{split}
\end{equation*}
$B_{41}$~is bounded by the trace inequality and the Young inequality,
\begin{equation*}
\begin{split}
B_{41}&=-\frac{2\Delta s}{\rho} \int_{\mathbb{I}}(P^h \phi_f(t_{n_{k+1}})-P^h \phi_f(t_{n_k})) e_f^{n_{k+1}} \cdot \boldsymbol{n}_d \mathrm{d}s\\
&\leq \frac{8C_T^2 h \Delta s^2}{\rho \gamma}\int_{t_{n_k}}^{t_{n_{k+1}}} \bigg\| \frac{\partial P^h \phi_f}{\partial t}\bigg\|_1^2 \mathrm{d}t
+\frac{\gamma \Delta s}{8\rho h}\sum_{n=n_k}^{n_{k+1}-1}\|e_f^{n_{k+1}} \cdot \boldsymbol{n}_d\|_{L^2(\mathbb{I})}^2.
\end{split}
\end{equation*}
And~$B_{42}$~is bounded by the trace inverse inequality and the Young inequality,
\begin{equation*}
\begin{split}
B_{42}&= -\frac{2 \Delta s}{\rho} \int_{\mathbb{I}}\theta_f^{n_k} e_f^{n_{k+1}} \cdot \boldsymbol{n}_d \mathrm{d}s\\
&\leq \frac{2 \tilde{C}_{inv}h^{-1/2}\Delta s}{\rho}\sum_{n=n_k}^{n_{k+1}-1}\|\theta_f^{n_k}\|_0 \|e_f^{n_{k+1}} \cdot \boldsymbol{n}_d\|_{L^2(\mathbb{I})}\\
&\leq \frac{8 \tilde{C}_{inv}^2 \Delta s}{\rho \gamma}\sum_{n=n_k}^{n_{k+1}-1}\|\theta_f^{n_k}\|_0^2 + \frac{\gamma \Delta s}{8 \rho h}\sum_{n=n_k}^{n_{k+1}-1} \|e_f^{n_{k+1}} \cdot \boldsymbol{n}_d\|_{L^2(\mathbb{I})}^2.
\end{split}
\end{equation*}
Therefore,
\begin{equation*}
\begin{split}
B_4 &\leq \frac{8C_T^2 h \Delta s^2}{\rho \gamma}\int_{t_{n_k}}^{t_{n_{k+1}}} \bigg\| \frac{\partial P^h \phi_f}{\partial t}\bigg\|_1^2 \mathrm{d}t
+\frac{\gamma \Delta s}{4\rho h}\sum_{n=n_k}^{n_{k+1}-1}\|e_f^{n_{k+1}} \cdot \boldsymbol{n}_d\|_{L^2(\mathbb{I})}^2
+\frac{8 \tilde{C}_{inv}^2 \Delta s}{\rho \gamma}\sum_{n=n_k}^{n_{k+1}-1}\|\theta_f^{n_k}\|_0^2.
\end{split}
\end{equation*}
Substitute~$B_1-B_4$~into~\eqref{B}~and sum over~$k=0,...,M-1$,
\begin{align*}
&\frac{\eta_d C_{dt}}{\rho}\sum_{k=0}^{M-1} (\|\theta_d^{n_{k+1}}\|_0^2 - \|\theta_d^{n_k}\|_0^2 + \|\theta_d^{n_{k+1}}-\theta_d^{n_k}\|_0^2)
+\frac{2\mu \Delta s}{\rho k_m}\sum_{k=0}^{M-1}\|e_m^{n_{k+1}}\|_0^2
+\frac{2\mu \Delta s}{\rho k_f}\sum_{k=0}^{M-1}\|e_f^{n_{k+1}}\|_0^2\\
&\frac{\sigma k_m \Delta s}{\rho \mu}\sum_{k=0}^{M-1}(\|\theta_m^{n_{k+1}}\|_0^2 - \|\theta_f^{n_k}\|_0^2 + \|\theta_m^{n_{k+1}} - \theta_f^{n_k}\|_0^2)
+\frac{\sigma k_m \Delta s}{\rho \mu}\sum_{k=0}^{M-1}(\|\theta_f^{n_{k+1}}\|_0^2 - \|\theta_m^{n_k}\|_0^2 +
\|\theta_f^{n_{k+1}}-\theta_m^{n_k}\|_0^2)\\
&+\frac{\gamma \Delta t}{2\rho h}\sum_{k=0}^{M-1}\sum_{n=n_k}^{n_{k+1}-1}\|e_f^{n_{k+1}} \cdot \boldsymbol{n}_d\|_{L^2(\mathbb{I})}^2\\
&\leq
\Delta s \sum_{k=0}^{M-1}(\frac{\eta_f C_{ft}}{\rho}\|\theta_f^{n_{k+1}}\|_0^2
+ \frac{\eta_m C_{mt}}{\rho}\|\theta_m^{n_{k+1}}\|_0^2
+ \frac{8 \tilde{C}_{inv}^2}{\rho \gamma}\|\theta_f^{n_k}\|_0^2)
+ \frac{\gamma \Delta t}{\rho h}\sum_{k=0}^{M-1}\sum_{n=n_k}^{n_{k+1}-1}\|e_c^{n+1}\|_{L^2(\mathbb{I})}^2\\
&+\frac{2\eta_f C_{ft}\Delta s^2}{3\rho}\int_{0}^{T} \bigg\|\frac{\partial^2 \phi_f}{\partial t^2} (t) \bigg\|_0^2 \mathrm{d}t
+\frac{2\eta_f C_{ft}}{\rho} \int_{0}^{T} \bigg\|(P^h - I) \frac{\partial \phi_f}{\partial t}(t)\bigg\|_0^2 \mathrm{d}t \\
&+ \frac{2\eta_m C_{mt}\Delta s^2}{3\rho}\int_{0}^{T} \bigg\|\frac{\partial^2 \phi_m}{\partial t^2}(t) \bigg\|_0^2 \mathrm{d}t
+\frac{2\eta_m C_{mt}}{\rho} \int_{0}^{T} \bigg\|(P^h - I) \frac{\partial \phi_m}{\partial t}(t)\bigg\|_0^2 \mathrm{d}t \\
&+\frac{4 \gamma r^2 C_T^2 \Delta t^2}{\rho h} \int_{0}^{T} \bigg\| \frac{\partial P^h \boldsymbol{u}_c}{\partial t}\bigg\|_1^2 \mathrm{d}t
+\frac{2 \sigma^2 k_m^2 \Delta s^2}{\eta_{f} C_{ft} \rho \mu^2}\int_{0}^{T}\bigg\|\frac{\partial P^h \phi_m}{\partial t}(t)\bigg\|_0^2 \mathrm{d}t
+\frac{2 \sigma^2 k_m^2 \Delta s^2}{\eta_{m} C_{mt}\rho \mu^2} \int_{0}^{T}\bigg\| \frac{\partial P^h \phi_{f}}{\partial t}(t) \bigg\|_0^2 \mathrm{d}t\\
&+\frac{8C_T^2 h \Delta s^2}{\rho \gamma}\int_{0}^{T} \bigg\| \frac{\partial P^h \phi_f}{\partial t}\bigg\|_1^2 \mathrm{d}t.
\end{align*}
By the discrete~Gronwall~lemma, when~$\kappa_n \Delta t \leq \frac{1}{2}, \kappa_n=\kappa_1+\kappa_2+\kappa_3, \kappa_1=\frac{\gamma}{\rho h}, \kappa_2=\frac{\eta_f C_{ft}}{\rho}+\frac{8 C_{inv}^2}{\rho \gamma}, \kappa_3=\frac{\eta_m C_{mt}}{\rho}$, we have
\begin{equation*}
\begin{split}
&\frac{\eta_d C_{dt}}{\rho}\sum_{k=0}^{M-1} (\|\theta_d^{n_{k+1}}\|_0^2 - \|\theta_d^{n_k}\|_0^2 + \|\theta_d^{n_{k+1}}-\theta_d^{n_k}\|_0^2)
+\frac{2\mu \Delta s}{\rho k_m}\sum_{k=0}^{M-1}\|e_m^{n_{k+1}}\|_0^2
+\frac{2\mu \Delta s}{\rho k_f}\sum_{k=0}^{M-1}\|e_f^{n_{k+1}}\|_0^2\\
&\frac{\sigma k_m \Delta s}{\rho \mu}\sum_{k=0}^{M-1}(\|\theta_m^{n_{k+1}}\|_0^2 - \|\theta_f^{n_k}\|_0^2 + \|\theta_m^{n_{k+1}} - \theta_f^{n_k}\|_0^2)
+\frac{\sigma k_m \Delta s}{\rho \mu}\sum_{k=0}^{M-1}(\|\theta_f^{n_{k+1}}\|_0^2 - \|\theta_m^{n_k}\|_0^2 +
\|\theta_f^{n_{k+1}}-\theta_m^{n_k}\|_0^2)\\
&+\frac{\gamma \Delta t}{2\rho h}\sum_{k=0}^{M-1}\sum_{n=n_k}^{n_{k+1}-1}\|e_f^{n_{k+1}} \cdot \boldsymbol{n}_d \|_{L^2(\mathbb{I})}^2\\
&\leq C(h^4 + \Delta t^2 + \Delta t^2 h^{-1}).
\end{split}
\end{equation*}
Especially, when~$k=0$, $\phi_f^0 =P^h \phi_{f0}, \phi_m^0 = P^h \phi_{m0}$,
\begin{equation}\label{tebie}
\begin{split}
&\frac{2\eta_d C_{dt}}{\rho}\|\theta_d^{n_{1}}\|_0^2
+\frac{2\mu \Delta s}{\rho k_m}\|e_m^{n_{1}}\|_0^2
+\frac{2\mu \Delta s}{\rho k_f}\|e_f^{n_{1}}\|_0^2
+\frac{\gamma \Delta t}{2\rho h}\|e_f^{n_{1}}\|_{L^2(\mathbb{I})}^2
+\frac{2\sigma k_m \Delta s}{\rho \mu}(\|\theta_m^{n_{1}}\|_0^2 + \|\theta_f^{n_{1}}\|_0^2)\\
&\leq C(h^4 + \Delta t^2 + \Delta t^2 h^{-1}).
\end{split}
\end{equation}

Next, we give the proof of~\eqref{THMP}~by mathematical induction in time interval~$[0,T]$. When~$l=0$,
by~\eqref{tebie}~and~\eqref{THMP1}~to know
\begin{equation*}
\begin{split}
&\|e_c^{n_{1}}\|_0^2 + \frac{\eta_f C_{ft}}{\rho} \|\theta_f^{n_{1}}\|_0^2
+ \frac{\eta_m C_{mt}}{\rho} \|\theta_m^{n_{1}}\|_0^2
+\nu \Delta t \sum_{n=1}^{n_{1}}\|\nabla e_c^{n}\|_0^2
+ \frac{2 \mu \Delta s}{\rho k_f}\|e_f^{n_{1}}\|_0^2
+ \frac{2 \mu \Delta s}{\rho k_m}\|e_m^{n_{1}}\|_0^2\\
&\leq C(h^4 + \Delta t^2 + \Delta t^2 h^{-1}).
\end{split}
\end{equation*}
Assume that~\eqref{THMP}~holds for~$l=j(1 \leq j \leq M-2)$, i.e.,
\begin{align}
&\max_{1 \leq j \leq M-2} \bigg\{ \|e_c^{n_{j+1}}\|_0^2 + \frac{\eta_f C_{ft}}{\rho} \|\theta_f^{n_{j+1}}\|_0^2
+ \frac{\eta_m C_{mt}}{\rho} \|\theta_m^{n_{j+1}}\|_0^2 \bigg\}
+\nu \Delta t \sum_{k=0}^j \sum_{n=n_k}^{n_{k+1}-1}\|\nabla e_c^{n+1}\|_0^2 \nonumber\\
&+ \frac{2\mu \Delta s}{\rho k_f}\sum_{k=0}^j \|e_f^{n_{k+1}}\|_0^2
+ \frac{2\mu \Delta s}{\rho k_m}\sum_{k=0}^j \|e_m^{n_{k+1}}\|_0^2
\leq C(h^4 + \Delta t^2 + \Delta t^2 h^{-1}).\label{THMP2}
\end{align}
Using inverse inequality~\eqref{inverse2}, the properties of~\eqref{Assume}, \eqref{THMP1}, the regularity of (??)and imbedding theorem, when~$\Delta t < \tau_1, h < h_1$~and~$\Delta t =\mathcal{O}(h^2)$, we have~$n=n_k+1,...,n_{k+1}, 1 \leq k \leq j$,
\begin{equation}\label{Bound}
\begin{split}
\|\boldsymbol{u}_c^n\|_{L^{\infty}} &\leq \|e_c^n\|_{L^{\infty}} + \|P^h \boldsymbol{u}_c(t_n)\|_{L^{\infty}}
\leq C h^{-3/2} \|e_c^n\|_0 + C \|\boldsymbol{u}_c(t_n)\|_2\\
&\leq C(h^{1/2} + h^{-3/2}\Delta t + h^{-2} \Delta t) + C \leq C_B.
\end{split}
\end{equation}
and
\begin{equation}\label{Boundd}
\begin{split}
\Delta t \|\boldsymbol{u}_c^n\|_{W^{1,\infty}} &\leq \Delta t(\|e_c^n\|_{W^{1,\infty}} + \|P^h \boldsymbol{u}_c(t_n)\|_{W^{1,\infty}})\\
&\leq \Delta t (Ch^{-3/2}\|\nabla e_c^n\|_0 + C \|\boldsymbol{u}(t_n)\|_{W^{2,d^{\ast}}})\\
&\leq C (\Delta t^{1/2} h^{1/2} + \Delta t^{3/2} h^{-3/2} + \Delta t^{3/2} h^{-2} + \Delta t)\\
&\leq \frac{1}{4},
\end{split}
\end{equation}

When~$l=j+1$, taking~$(\boldsymbol{v}_c^h,q^h;\boldsymbol{v}_f^h,\psi_f^h;\boldsymbol{v}_m^h,\psi_m^h)=(2\Delta t e_c^{n+1},2\Delta t \delta_c^{n+1};2 \Delta s e_f^{n_{k+1}}, 2\Delta s \theta_f^{n_{k+1}};$\\
$ 2\Delta s e_m^{n_{k+1}}, 2\Delta s \theta_m^{n_{k+1}})$, combining~\eqref{SS1}, \eqref{SS2}~and~\eqref{SS3}, and summing over~$n=n_k, n_k+1,...,n_{k+1}-1$, we have
\begin{align}\label{T}
&\|e_c^{n_{k+1}}\|_0^2 - \|e_c^{n_k}\|_0^2 + \sum_{n=n_k}^{n_{k+1}-1} \|e_c^{n+1}-e_c^n\|_0^2
+\frac{\eta_d C_{dt}}{\rho}(\|\theta_d^{n_{k+1}}\|_0^2-\|\theta_d^{n_k}\|_0^2+\|\theta_d^{n_{k+1}}-\theta_d^{n_k}\|_0^2) \nonumber\\
&+\frac{\sigma k_m \Delta s}{\rho \mu}(\|\theta_f^{n_{k+1}}\|_0^2-\|\theta_m^{n_k}\|_0^2+\|\theta_f^{n_{k+1}}-\theta_m^{n_k}\|_0^2)
+\frac{\sigma k_m \Delta s}{\rho \mu}(\|\theta_m^{n_{k+1}}\|_0^2-\|\theta_f^{n_k}\|_0^2+\|\theta_m^{n_{k+1}}-\theta_f^{n_k}\|_0^2) \nonumber\\
&+2\nu \Delta t \sum_{n=n_k}^{n_{k+1}-1}\|\nabla e_c^{n+1}\|_0^2 + \frac{2\alpha \nu \Delta t}{\sqrt{k}} \sum_{n=n_k}^{n_{k+1}-1}\|P_{\tau}(e_c^{n+1})\|_{L^2(\mathbb{I})}^2
+\frac{2\mu \Delta s}{\rho k_f}\|e_f^{n_{k+1}}\|_0^2 + \frac{2\mu \Delta s}{\rho k_m}\|e_m^{n_{k+1}}\|_0^2 \nonumber\\
&=2\Delta t \sum_{n=n_k}^{n_{k+1}-1}\bigg( \Psi_{c,t}^{n+1}
+\frac{\boldsymbol{u}_c^n - \hat{\boldsymbol{u}}_c^n-\boldsymbol{u}_c(t_n)+\bar{\boldsymbol{u}}_c(t_n)}{\Delta t}+R_{tr}^{n+1}, e_c^{n+1}\bigg)_{\Omega_c}
+\frac{\eta_d C_{dt}}{\rho}\bigg(\Psi_{d,s}^{n_{k+1}},2\Delta s \theta_d^{n_{k+1}}\bigg)_{\Omega_d} \nonumber\\
&+\bigg[\frac{2\Delta t}{\rho}\sum_{n=n_k}^{n_{k+1}-1}\int_{\mathbb{I}} (P^h \phi_f(t_{n+1})-\phi_f^{n_k} ) e_c^{n+1} \cdot \boldsymbol{n}_d \mathrm{d}s
-\frac{2\Delta s}{\rho}\int_{\mathbb{I}} (P^h \phi_f(t_{n_{k+1}})-\phi_f^{n_k})e_f^{n_{k+1}}\cdot \boldsymbol{n}_d \mathrm{d}s \bigg] \nonumber\\
&+\bigg[ \frac{2\gamma \Delta s}{\rho h}\int_{\mathbb{I}}\big( (P^h \boldsymbol{u}_c(t_{n_{k+1}})-S^{n_{k+1}})-e_f^{n_{k+1}}\big)\cdot \boldsymbol{n}_d e_f^{n_{k+1}} \cdot \boldsymbol{n}_d \mathrm{d}s \nonumber\\
&~~~~-\frac{2\gamma \Delta t}{\rho h}\sum_{n=n_k}^{n_{k+1}-1}\int_{\mathbb{I}}\big( e_c^{n+1}-(P^h \boldsymbol{u}_f(t_{n+1})-\boldsymbol{u}_f^{n_k}) \big)\cdot \boldsymbol{n}_d e_c^{n+1}\cdot \boldsymbol{n}_d \mathrm{d}s \bigg] \nonumber\\
&+\bigg[ \frac{2\sigma k_m \Delta s}{\rho \mu}(P^h \phi_m(t_{n_{k+1}})-P^h \phi_m(t_{n_k}),\theta_f^{n_{k+1}})_{\Omega_d}
+\frac{2\sigma k_m \Delta s}{\rho \mu}(P^h \phi_f(t_{n_{k+1}})-P^h \phi_f(t_{n_k}),\theta_m^{n_{k+1}})_{\Omega_d} \bigg] \nonumber\\
&:=\sum_{i=1}^5 T_i.
\end{align}
For the first term
\begin{equation*}
\begin{split}
T_1&=2\Delta t \sum_{n=n_k}^{n_{k+1}-1}(\Psi_{c,t}^{n+1},e_c^{n+1})_{\Omega_c}
+ 2\Delta t \sum_{n=n_k}^{n_{k+1}-1} \bigg(\frac{\boldsymbol{u}_c^n - \hat{\boldsymbol{u}}_c^n-\boldsymbol{u}_c(t_n)+\bar{\boldsymbol{u}}_c(t_n)}{\Delta t},e_c^{n+1}\bigg)_{\Omega_c}\\
&~~~~+ 2\Delta t \sum_{n=n_k}^{n_{k+1}-1} (R_{tr}^{n+1}, e_c^{n+1})_{\Omega_c}\\
&:=T_{11}+T_{12}+T_{13}.
\end{split}
\end{equation*}
We repeat the same procedure in~$A_1$, by the boundedness of~\eqref{Bound}~and~\eqref{Boundd},
\begin{equation*}
\begin{split}
T_1 &\leq \frac{\nu \Delta t}{2}\sum_{n=n_k}^{n_{k+1}-1} (\|\nabla e_c^{n+1}\|_0^2+\|\nabla e_c^n\|_0^2)
+ \frac{2 C_N^2 C_{PF}^2 C_B^2 \Delta t}{\nu} \sum_{n=n_k}^{n_{k+1}-1} \|e_c^{n+1}\|_0^2\\
&+ \frac{30 C_N^2 C_S^2 C_{PF}^2 C_H^2 \Delta t}{\nu} \sum_{n=n_k}^{n_{k+1}-1} \|e_c^n\|_0^2
+\frac{6C_{PF}^2 \Delta t}{\nu} \sum_{n=n_k}^{n_{k+1}-1} \|R_{tr}^{n+1}\|_0^2\\
&+\frac{6C_{PF}^2}{\nu}\int_{t_{n_k}}^{t_{n_{k+1}}} \bigg\|(P^h-I)\frac{\partial \boldsymbol{u}_c}{\partial t}(t)\bigg\|_0^2 \mathrm{d}t
+Ch^4 \Delta t.
\end{split}
\end{equation*}
For the second term, similar with the deduce of~$B_1$, by~Cauchy-Schwarz~inequality and Young~inequality
\begin{equation*}
\begin{split}
T_2 &\leq \frac{2\eta_f C_{ft}\Delta s^2}{3\rho}\int_{t_{n_k}}^{t_{n_{k+1}}} \bigg\|\frac{\partial^2 \phi_f}{\partial t^2} (t) \bigg\|_0^2 \mathrm{d}t
+\frac{2\eta_f C_{ft}}{\rho} \int_{t_{n_k}}^{t_{n_{k+1}}} \bigg\|(P^h - I) \frac{\partial \phi_f}{\partial t}(t)\bigg\|_0^2 \mathrm{d}t \\
&+ \frac{2\eta_m C_{mt}\Delta s^2}{3\rho}\int_{t_{n_k}}^{t_{n_{k+1}}} \bigg\|\frac{\partial^2 \phi_m}{\partial t^2}(t) \bigg\|_0^2 \mathrm{d}t
+\frac{2\eta_m C_{mt}}{\rho} \int_{t_{n_k}}^{t_{n_{k+1}}} \bigg\|(P^h - I) \frac{\partial \phi_m}{\partial t}(t)\bigg\|_0^2 \mathrm{d}t \\
&+\frac{\eta_f C_{ft} \Delta s}{2\rho} \|\theta_f^{n_{k+1}}\|_0^2
+\frac{\eta_m C_{mt} \Delta s}{2\rho} \|\theta_m^{n_{k+1}}\|_0^2.
\end{split}
\end{equation*}
For the third term
\begin{align*}
T_3 &= \frac{2\Delta t}{\rho}\sum_{n=n_k}^{n_{k+1}-1}\int_{\mathbb{I}} (P^h \phi_f(t_{n+1})-\phi_f^{n_k} ) e_c^{n+1} \cdot \boldsymbol{n}_d \mathrm{d}s
-\frac{2\Delta s}{\rho}\int_{\mathbb{I}} (P^h \phi_f(t_{n_{k+1}})-\phi_f^{n_k})e_f^{n_{k+1}}\cdot \boldsymbol{n}_d \mathrm{d}s\\
&= \frac{2 \Delta t}{\rho} \sum_{n=n_k}^{n_{k+1}-1} \int_{\mathbb{I}} \big( P^h \phi_f(t_{n+1})- P^h \phi_f(t_{n_{k+1}}) \big) e_c^{n+1} \cdot \boldsymbol{n}_d \mathrm{d}s \\
&+ \frac{2\Delta t}{\rho} \sum_{n=n_k}^{n_{k+1}-1} \int_{\mathbb{I}} \theta_f^{n_k} (e_c^{n+1}-e_f^{n_{k+1}}) \cdot \boldsymbol{n}_d \mathrm{d}s\\
&+ \frac{2\Delta t}{\rho} \sum_{n=n_k}^{n_{k+1}-1} \int_{\mathbb{I}} \big( P^h \phi_f(t_{n_{k+1}})-P^h \phi_f(t_{n_k}) \big) (e_c^{n+1} - e_f^{n_{k+1}}) \cdot \boldsymbol{n}_d \mathrm{d}s\\
&:= T_{31} + T_{32} + T_{33}.
\end{align*}
Using~Cauchy-Schwarz inequality, the trace inverse inequality and the~Young~inequality, we arrive at
\begin{equation*}
\begin{split}
T_{31} &=\frac{2 \Delta t}{\rho} \sum_{n=n_k}^{n_{k+1}-1} \int_{\mathbb{I}} \big( P^h \phi_f(t_{n+1})- P^h \phi_f(t_{n_{k+1}}) \big) e_c^{n+1} \cdot \boldsymbol{n}_d \mathrm{d}s\\
&\leq \frac{2 \tilde{C}_{inv}^2 r h^{-1}\Delta t}{\rho} \sum_{n=n_k}^{n_{k+1}-1} \|P^h \phi_f(t_{n+1})- P^h \phi_f(t_{n})\|_0 \|e_c^{n+1}\|_0\\
&\leq \frac{\tilde{C}_{inv}^2 \Delta t}{h}\sum_{n=n_k}^{n_{k+1}-1} \|e_c^{n+1}\|_0^2
+ \frac{\tilde{C}_{inv}^2 r^2 \Delta t^2}{\rho^2 h} \int_{t_{n_k}}^{t_{n_{k+1}}} \bigg\| \frac{\partial P^h \phi_f}{ \partial t}(t) \bigg\|_0^2 \mathrm{d}t,
\end{split}
\end{equation*}

\begin{equation*}
\begin{split}
T_{32} &= \frac{2\Delta t}{\rho} \sum_{n=n_k}^{n_{k+1}-1} \int_{\mathbb{I}} \theta_f^{n_k} (e_c^{n+1}-e_f^{n_{k+1}}) \cdot \boldsymbol{n}_d \mathrm{d}s \\
&\leq \frac{2 \tilde{C}_{inv} h^{-1/2} \Delta t}{\rho} \sum_{n=n_k}^{n_{k+1}-1} \|\theta_f^{n_k}\|_0 \|(e_c^{n+1}-e_f^{n_{k+1}}) \cdot \boldsymbol{n}_d\|_{L^2(\mathbb{I})} \\
&\leq \frac{4 \tilde{C}_{inv}^2 \Delta s}{\rho \gamma} \|\theta_f^{n_k}\|_0^2
+ \frac{\gamma \Delta t}{4 \rho h} \sum_{n=n_k}^{n_{k+1}-1} \|(e_c^{n+1}-e_f^{n_{k+1}}) \cdot \boldsymbol{n}_d\|_{L^2(\mathbb{I})}^2,
\end{split}
\end{equation*}
and
\begin{equation*}
\begin{split}
T_{33} &= \frac{2\Delta t}{\rho} \sum_{n=n_k}^{n_{k+1}-1} \int_{\mathbb{I}} \big( P^h \phi_f(t_{n_{k+1}})-P^h \phi_f(t_{n_k}) \big) (e_c^{n+1} - e_f^{n_{k+1}}) \cdot \boldsymbol{n}_d \mathrm{d}s\\
&\leq \frac{2 \tilde{C}_{inv} h^{-1/2} \Delta t}{\rho} \sum_{n=n_k}^{n_{k+1}-1} \|P^h \phi_f(t_{n_{k+1}})-P^h \phi_f(t_{n_k})\|_0 \|(e_c^{n+1} - e_f^{n_{k+1}}) \cdot \boldsymbol{n}_d\|_{L^2(\mathbb{I})}\\
&\leq \frac{4 r \tilde{C}_{inv}^2 \Delta t^2}{\rho \gamma}\int_{t_{n_k}}^{t_{n_{k+1}}} \bigg\| \frac{\partial P^h \phi_f}{\partial t}(t)\bigg\|_0^2 \mathrm{d}t
+\frac{\gamma \Delta t}{4\rho h}\sum_{n=n_k}^{n_{k+1}-1}\|(e_c^{n+1}-e_f^{n_{k+1}}) \cdot \boldsymbol{n}_d\|_{L^2(\mathbb{I})}^2.
\end{split}
\end{equation*}
Therefore,
\begin{equation*}
\begin{split}
T_3 &\leq \frac{\tilde{C}_{inv}^2 \Delta t}{h}\sum_{n=n_k}^{n_{k+1}-1} \|e_c^{n+1}\|_0^2
+\frac{4 \tilde{C}_{inv}^2 \Delta s}{\rho \gamma}\|\theta_f^{n_k}\|_0^2
+\frac{\gamma \Delta t}{2\rho h}\sum_{n=n_k}^{n_{k+1}-1}\|(e_c^{n+1}-e_f^{n_{k+1}}) \cdot \boldsymbol{n}_d\|_{L^2(\mathbb{I})}^2\\
&+\frac{4 r \tilde{C}_{inv}^2 \Delta t^2}{\rho \gamma}\int_{t_{n_k}}^{t_{n_{k+1}}} \bigg\| \frac{\partial P^h \phi_f}{\partial t}(t)\bigg\|_0^2 \mathrm{d}t
+ \frac{\tilde{C}_{inv}^2 r^2 \Delta t^2}{\rho^2 h} \int_{t_{n_k}}^{t_{n_{k+1}}} \bigg\| \frac{\partial P^h \phi_f}{ \partial t}(t) \bigg\|_0^2 \mathrm{d}t.
\end{split}
\end{equation*}
For the fourth term, since~$S^{n_{k+1}}=\frac{1}{r}\sum_{n=n_k}^{n_{k+1}-1} \boldsymbol{u}_c^{n+1}$, then
\begin{equation*}
\begin{split}
T_4 &=\frac{2\gamma \Delta s}{\rho h}\int_{\mathbb{I}}\big( (P^h \boldsymbol{u}_c(t_{n_{k+1}})-S^{n_{k+1}})-e_f^{n_{k+1}}\big)\cdot \boldsymbol{n}_d e_f^{n_{k+1}} \cdot \boldsymbol{n}_d \mathrm{d}s\\
&~~~~-\frac{2\gamma \Delta t}{\rho h}\sum_{n=n_k}^{n_{k+1}-1}\int_{\mathbb{I}}\big( e_c^{n+1}-(P^h \boldsymbol{u}_f(t_{n+1})-\boldsymbol{u}_f^{n_k}) \big)\cdot \boldsymbol{n}_d e_c^{n+1}\cdot \boldsymbol{n}_d \mathrm{d}s \\
&=\frac{2 \gamma \Delta t}{\rho h} \sum_{n=n_k}^{n_{k+1}-1} \int_{\mathbb{I}} \big( P^h \boldsymbol{u}_c(t_{n_{k+1}}) - P^h \boldsymbol{u}_c(t_{n+1}) + e_c^{n+1} - e_f^{n_{k+1}}\big) \cdot \boldsymbol{n}_d e_f^{n_{k+1}} \cdot \boldsymbol{n}_d \mathrm{d}s \\
&~~~~-\frac{2\gamma \Delta t}{\rho h} \sum_{n=n_k}^{n_{k+1}-1} \int_{\mathbb{I}} (e_c^{n+1} - e_f^{n_{k+1}}) \cdot \boldsymbol{n}_d e_c^{n+1} \cdot \boldsymbol{n}_d \mathrm{d}s \\
&~~~~-\frac{2 \gamma \Delta t}{\rho h} \sum_{n=n_k}^{n_{k+1}-1} \int_{\mathbb{I}} \big( e_f^{n_{k+1}} -
(P^h \boldsymbol{u}_f(t_{n+1}) - \boldsymbol{u}_f^{n_k})\big) \cdot \boldsymbol{n}_d e_c^{n+1} \cdot \boldsymbol{n}_d \mathrm{d}s \\
&= \frac{2 \gamma \Delta t}{\rho h} \sum_{n=n_k}^{n_{k+1}-1} \int_{\mathbb{I}} \big( P^h \boldsymbol{u}_c(t_{n_{k+1}}) - P^h \boldsymbol{u}_c(t_{n+1}) \big) \cdot \boldsymbol{n}_d e_f^{n_{k+1}} \cdot \boldsymbol{n}_d \mathrm{d}s \\
&~~~~+ \frac{2 \gamma \Delta t}{\rho h} \sum_{n=n_k}^{n_{k+1}-1} \int_{\mathbb{I}} \big( P^h \boldsymbol{u}_f(t_{n+1}) - P^h \boldsymbol{u}_f(t_{n_k})\big) \cdot \boldsymbol{n}_d e_c^{n+1} \cdot \boldsymbol{n}_d \mathrm{d}s \\
&~~~~-\bigg[ \frac{2 \gamma \Delta t}{\rho h}\sum_{n=n_k}^{n_{k+1}-1} \|(e_c^{n+1}-e_f^{n_{k+1}}) \cdot \boldsymbol{n}_d\|_{L^2(\mathbb{I})}^2
+ \frac{2 \gamma \Delta t}{\rho h} \sum_{n=n_k}^{n_{k+1}-1} \int_{\mathbb{I}} (e_f^{n_{k+1}} - e_f^{n_k}) \cdot \boldsymbol{n}_d e_c^{n+1} \cdot \boldsymbol{n}_d \mathrm{d}s \bigg] \\
&:= T_{41} + T_{42} + T_{43}.
\end{split}
\end{equation*}
Applying~Cauchy-Schwarz~inequality and Young inequality, we show that
\begin{align*}
T_{41} &= \frac{2 \gamma \Delta t}{\rho h} \sum_{n=n_k}^{n_{k+1}-1} \int_{\mathbb{I}} \big( P^h \boldsymbol{u}_c(t_{n_{k+1}}) - P^h \boldsymbol{u}_c(t_{n+1}) \big) \cdot \boldsymbol{n}_d e_f^{n_{k+1}} \cdot \boldsymbol{n}_d \mathrm{d}s \\
&\leq \frac{2 \gamma C_T \Delta t}{\rho h} \sum_{n=n_k}^{n_{k+1}-1} \|P^h \boldsymbol{u}_c(t_{n_{k+1}})- P^h \boldsymbol{u}_c(t_{n+1})\|_1 \|(e_c^{n+1}-e_f^{n_{k+1}}) \cdot \boldsymbol{n}_d\|_{L^2(\mathbb{I})}\\
&~~~~+\frac{r^2 \gamma^2 C_T^2  \Delta t^2}{\rho^2 h} \int_{t_{n_k}}^{t_{n_{k+1}}} \bigg\|\frac{\partial P^h \boldsymbol{u}_c}{\partial t}(t)\bigg\|_1^2 \mathrm{d}t
+ \frac{C_{div}^2 \Delta t}{h}\sum_{n=n_k}^{n_{k+1}-1}\|e_c^{n+1}\|_0^2\\
&\leq \frac{4 \gamma C_T^2 \Delta t^2}{\rho h} \int_{t_{n_k}}^{t_{n_{k+1}}} \bigg\|\frac{\partial P^h \boldsymbol{u}_c}{\partial t}(t)\bigg\|_1^2 \mathrm{d}t
+ \frac{\gamma \Delta t}{4 \rho h} \sum_{n=n_k}^{n_{k+1}-1}\|(e_c^{n+1}-e_f^{n_{k+1}}) \cdot \boldsymbol{n}_d\|_{L^2(\mathbb{I})}^2 \\
&~~~~+\frac{r^2 \gamma^2 C_T^2  \Delta t^2}{\rho^2 h} \int_{t_{n_k}}^{t_{n_{k+1}}} \bigg\|\frac{\partial P^h \boldsymbol{u}_c}{\partial t}(t)\bigg\|_1^2 \mathrm{d}t
+ \frac{C_{div}^2 \Delta t}{h}\sum_{n=n_k}^{n_{k+1}-1}\|e_c^{n+1}\|_0^2.
\end{align*}
By~H$\mathrm{\ddot{o}}$lder~inequality, the general trace inequality, the Young inequality, the property of~$H(\mathrm{div})$~space and the divergence free condition,
\begin{equation*}
\begin{split}
T_{42} &= \frac{2 \gamma \Delta t}{\rho h} \sum_{n=n_k}^{n_{k+1}-1} \int_{\mathbb{I}} \big( P^h \boldsymbol{u}_f(t_{n+1}) - P^h \boldsymbol{u}_f(t_{n_k})\big) \cdot \boldsymbol{n}_d e_c^{n+1} \cdot \boldsymbol{n}_d \mathrm{d}s \\
&= \frac{2 \gamma \Delta t}{\rho h} \sum_{n=n_k}^{n_{k+1}-1} \int_{\partial\Omega_c} \big( P^h \boldsymbol{u}_f(t_{n+1}) - P^h \boldsymbol{u}_f(t_{n_k})\big) \cdot \boldsymbol{n}_d e_c^{n+1} \cdot \boldsymbol{n}_d \mathrm{d}s \\
&\leq \frac{2 \gamma \Delta t}{\rho h} \sum_{n=n_k}^{n_{k+1}-1} \|\big( P^h \boldsymbol{u}_f(t_{n+1}) - P^h \boldsymbol{u}_f(t_{n_k})\big) \cdot \boldsymbol{n}_d\|_{H^{\frac{1}{2}}(\partial \Omega_c)} \|e_c^{n+1} \cdot \boldsymbol{n}_d\|_{H^{-\frac{1}{2}}(\partial \Omega_c)}\\
&\leq \frac{2r\gamma C_T C_{div}\Delta t}{\rho h}\sum_{n=n_k}^{n_{k+1}-1}\|P^h \boldsymbol{u}_f(t_{n+1})-P^h \boldsymbol{u}_f(t_n)\|_1 \|e_c^{n+1}\|_{H(\mathrm{div})}\\
&\leq \frac{r^2 \gamma^2 C_T^2  \Delta t^2}{\rho^2 h} \int_{t_{n_k}}^{t_{n_{k+1}}} \bigg\|\frac{\partial P^h \boldsymbol{u}_f}{\partial t}(t)\bigg\|_1^2 \mathrm{d}t
+ \frac{C_{div}^2 \Delta t}{h}\sum_{n=n_k}^{n_{k+1}-1}\|e_c^{n+1}\|_0^2.
\end{split}
\end{equation*}
And
\begin{equation*}
\begin{split}
T_{43} &=- \frac{2 \gamma \Delta t}{\rho h}\sum_{n=n_k}^{n_{k+1}-1} \|(e_c^{n+1}-e_f^{n_{k+1}}) \cdot \boldsymbol{n}_d\|_{L^2(\mathbb{I})}^2
- \frac{2 \gamma \Delta t}{\rho h} \sum_{n=n_k}^{n_{k+1}-1} \int_{\mathbb{I}} (e_f^{n_{k+1}} - e_f^{n_k}) \cdot \boldsymbol{n}_d e_c^{n+1} \cdot \boldsymbol{n}_d \mathrm{d}s\\
&= - \frac{2 \gamma \Delta t}{\rho h}\sum_{n=n_k}^{n_{k+1}-1} \|(e_c^{n+1}-e_f^{n_{k+1}}) \cdot \boldsymbol{n}_d\|_{L^2(\mathbb{I})}^2
-\frac{2\gamma \Delta s}{\rho h} \|e_f^{n_{k+1}} \cdot \boldsymbol{n}_d\|_{L^2(\mathbb{I})}^2 \\
&~~~~+ \frac{2 \gamma \Delta t}{\rho h} \sum_{n=n_k}^{n_{k+1}-1} \int_{\mathbb{I}} e_f^{n_k} \cdot \boldsymbol{n}_d e_f^{n_{k+1}} \cdot \boldsymbol{n}_d \mathrm{d}s
+ \frac{2 \gamma \Delta t}{\rho h} \sum_{n=n_k}^{n_{k+1}-1} \int_{\mathbb{I}} (e_f^{n_{k+1}}- e_f^{n_k}) \cdot \boldsymbol{n}_d (e_f^{n_{k+1}}-e_c^{n+1}) \cdot \boldsymbol{n}_d \mathrm{d}s \\
&\leq -\frac{\gamma \Delta t}{\rho h} \sum_{n=n_k}^{n_{k+1}-1} \|(e_c^{n+1} - e_f^{n_{k+1}})\cdot \boldsymbol{n}_d\|_{L^2(\mathbb{I})}^2 - \frac{\gamma \Delta s}{\rho h} \bigg[ \|e_f^{n_{k+1}} \cdot \boldsymbol{n}_d\|_{L^2(\mathbb{I})}^2 - \|e_f^{n_k} \cdot \boldsymbol{n}_d\|_{L^2(\mathbb{I})}^2 \bigg].
\end{split}
\end{equation*}
Therefore,
\begin{align*}
T_4 &\leq - \frac{ \gamma \Delta s}{ \rho h} \bigg[ \|e_f^{n_{k+1}} \cdot \boldsymbol{n}_d\|_{L^2(\mathbb{I})}^2 - \|e_f^{n_k} \cdot \boldsymbol{n}_d\|_{L^2(\mathbb{I})}^2 \bigg]
-\frac{3 \gamma \Delta t}{4 \rho h}\sum_{n=n_k}^{n_{k+1}-1}\|(e_c^{n+1}-e_f^{n_{k+1}})\cdot \boldsymbol{n}_d\|_{L^2(\mathbb{I})}^2\\
&+ \frac{4 \gamma C_T^2 \Delta t^2}{\rho h} \int_{t_{n_k}}^{t_{n_{k+1}}} \bigg\|\frac{\partial P^h \boldsymbol{u}_c}{\partial t}(t)\bigg\|_1^2 \mathrm{d}t
+ \frac{r^2 \gamma^2 C_T^2  \Delta t^2}{\rho^2 h} \int_{t_{n_k}}^{t_{n_{k+1}}} \bigg\|\frac{\partial P^h \boldsymbol{u}_c}{\partial t}(t)\bigg\|_1^2 \mathrm{d}t\\
&+ \frac{r^2 \gamma^2 C_T^2  \Delta t^2}{\rho^2 h} \int_{t_{n_k}}^{t_{n_{k+1}}} \bigg\|\frac{\partial P^h \boldsymbol{u}_f}{\partial t}(t)\bigg\|_1^2 \mathrm{d}t
+ \frac{2 C_{div}^2 \Delta t}{h}\sum_{n=n_k}^{n_{k+1}-1}\|e_c^{n+1}\|_0^2.
\end{align*}
For the fifth term
\begin{equation*}
\begin{split}
T_5 &= \frac{2\sigma k_m \Delta s}{\rho \mu}(P^h \phi_m(t_{n_{k+1}})-P^h \phi_m(t_{n_k}),\theta_f^{n_{k+1}})_{\Omega_d}
+\frac{2\sigma k_m \Delta s}{\rho \mu}(P^h \phi_f(t_{n_{k+1}})-P^h \phi_f(t_{n_k}),\theta_m^{n_{k+1}})_{\Omega_d}\\
&:= T_{51} + T_{52}.
\end{split}
\end{equation*}
$T_{51}$~is bounded by Cauchy-Schwarz~inequality and the Young inequality,
\begin{equation*}
\begin{split}
T_{51}&=\frac{2\sigma k_m \Delta s}{\rho \mu}(P^h \phi_m(t_{n_{k+1}})-P^h \phi_m(t_{n_k}),\theta_f^{n_{k+1}})_{\Omega_d}\\
&\leq \frac{2 \sigma^2 k_m^2 \Delta s^2}{\eta_{f} C_{ft} \rho \mu^2}\int_{t_{n_k}}^{t_{n_{k+1}}}\bigg\|\frac{\partial P^h \phi_m}{\partial t}(t)\bigg\|_0^2 \mathrm{d}t
+ \frac{\eta_{f} C_{ft} \Delta s}{2 \rho}\|\theta_f^{n_{k+1}}\|_0^2.
\end{split}
\end{equation*}
In the same spirit,
\begin{equation*}
\begin{split}
T_{52}&=\frac{2\sigma k_m \Delta s}{\rho \mu}(P^h \phi_f(t_{n_{k+1}})-P^h \phi_f(t_{n_k}),\theta_m^{n_{k+1}})_{\Omega_d}\\
&\leq \frac{2 \sigma^2 k_m^2 \Delta s^2}{\eta_{m} C_{mt}\rho \mu^2} \int_{t_{n_k}}^{t_{n_{k+1}}}\bigg\| \frac{\partial P^h \phi_{f}}{\partial t}(t) \bigg\|_0^2 \mathrm{d}t
+\frac{\eta_{m} C_{mt} \Delta s}{2 \rho}\|\theta_m^{n_{k+1}}\|_0^2.
\end{split}
\end{equation*}
Therefore,
\begin{equation*}
\begin{split}
T_5 &\leq \frac{2 \sigma^2 k_m^2 \Delta s^2}{\eta_{f} C_{ft} \rho \mu^2}\int_{t_{n_k}}^{t_{n_{k+1}}}\bigg\|\frac{\partial P^h \phi_m}{\partial t}(t)\bigg\|_0^2 \mathrm{d}t
+ \frac{\eta_{f} C_{ft}\Delta s}{2 \rho}\|\theta_f^{n_{k+1}}\|_0^2\\
&+\frac{2 \sigma^2 k_m^2 \Delta s^2}{\eta_{m} C_{mt}\rho \mu^2} \int_{t_{n_k}}^{t_{n_{k+1}}}\bigg\| \frac{\partial P^h \phi_{f}}{\partial t}(t) \bigg\|_0^2 \mathrm{d}t
+\frac{\eta_{m} C_{mt} \Delta s}{2 \rho}\|\theta_m^{n_{k+1}}\|_0^2.
\end{split}
\end{equation*}
Combining~$T_1, T_2, T_3, T_4$~and~$T_5$, summing over~$k=0,...,j+1(1\leq j \leq M-2)$, we arrive at
\begin{align*}
&\|e_c^{n_{j+2}}\|_0^2 + \sum_{k=0}^{j+1} \sum_{n=n_k}^{n_{k+1}-1}\|e_c^{n+1}-e_c^n\|_0^2 + \frac{\eta_f C_{ft}}{\rho} \|\theta_f^{n_{j+2}}\|_0^2
+ \frac{\eta_m C_{mt}}{\rho} \|\theta_m^{n_{j+2}}\|_0^2
+\frac{\eta_f C_{ft}}{\rho}\sum_{k=0}^{j+1}\|\theta_f^{n_{k+1}} - \theta_f^{n_k}\|_0^2\\
&+\frac{\eta_m C_{mt}}{\rho}\sum_{k=0}^{j+1}\|\theta_m^{n_{k+1}} - \theta_m^{n_k}\|_0^2
+\frac{\sigma k_m \Delta s}{\rho \mu}\sum_{k=0}^{j+1}(\|\theta_f^{n_{k+1}}-\theta_m^{n_k}\|_0^2 + \|\theta_m^{n_{k+1}}-\theta_f^{n_k}\|_0^2)\\
&+\nu \Delta t \sum_{k=0}^{j+1} \sum_{n=n_k}^{n_{k+1}-1}\|\nabla e_c^{n+1}\|_0^2
+ \frac{2\alpha \nu \Delta t}{\sqrt{k}}\sum_{k=0}^{j+1} \sum_{n=n_k}^{n_{k+1}-1}\|P_{\tau}(e_c^{n+1})\|_{L^2(\mathbb{I})}^2
+ \frac{2 \mu \Delta s}{\rho k_f}\sum_{k=0}^{j+1}\|e_f^{n_{k+1}}\|_0^2\\
&+ \frac{2 \mu \Delta s}{\rho k_m}\sum_{k=0}^{j+1}\|e_m^{n_{k+1}}\|_0^2
+\frac{\gamma \Delta t}{4 \rho h}\sum_{k=0}^{j+1} \sum_{n=n_k}^{n_{k+1}-1} \|(e_c^{n+1}-e_f^{n_{k+1}})\cdot \boldsymbol{n}_d\|_{L^2(\mathbb{I})}^2
+\frac{\gamma \Delta s}{\rho h} \|e_f^{n_{j+2}} \cdot \boldsymbol{n}_d\|_{L^2(\mathbb{I})}^2\\
&\leq \frac{6C_{PF}^2 \Delta t}{\nu} \sum_{k=0}^{j+1} \sum_{n=n_k}^{n_{k+1}-1}\|R_{tr}^{n+1}\|_0^2
+ \frac{6C_{PF}^2}{\nu}\int_{0}^{T} \bigg\|(P^h-I)\frac{\partial \boldsymbol{u}_c}{\partial t}(t)\bigg\|_0^2 \mathrm{d}t
+\frac{r^2 \gamma^2 C_T^2  \Delta t^2}{\rho^2 h} \int_{t_{n_k}}^{t_{n_{k+1}}} \bigg\|\frac{\partial P^h \boldsymbol{u}_c}{\partial t}(t)\bigg\|_1^2 \mathrm{d}t\\
&+ \frac{2\eta_f C_{ft}\Delta s^2}{3\rho}\int_{0}^{T} \bigg\|\frac{\partial^2 \phi_f}{\partial t^2} (t)\bigg\|_0^2 \mathrm{d}t
+\frac{2\eta_f C_{ft}}{\rho} \int_{0}^{T} \bigg\|(P^h - I) \frac{\partial \phi_f}{\partial t}(t)\bigg\|_0^2 \mathrm{d}t
+ \frac{2\eta_m C_{mt}\Delta s^2}{3\rho}\int_{0}^{T} \bigg\|\frac{\partial^2 \phi_m}{\partial t^2} (t)\bigg\|_0^2 \mathrm{d}t\\
&+\frac{2\eta_m C_{mt}}{\rho} \int_{0}^{T} \bigg\|(P^h - I) \frac{\partial \phi_m}{\partial t}(t)\bigg\|_0^2 \mathrm{d}t
+\frac{4 r \tilde{C}_{inv}^2 \Delta t^2}{\rho \gamma}\int_{0}^{T} \bigg\| \frac{\partial P^h \phi_f}{\partial t}(t)\bigg\|_0^2 \mathrm{d}t\\
&+\frac{r^2 \gamma^2 C_T^2  \Delta t^2}{\rho^2 h} \int_{0}^{T} \bigg\|\frac{\partial P^h \boldsymbol{u}_f}{\partial t}(t)\bigg\|_1^2 \mathrm{d}t
+\frac{2 \sigma^2 k_m^2 \Delta s^2}{\eta_{f} C_{ft} \rho \mu^2}\int_{0}^{T}\bigg\|\frac{\partial P^h \phi_m}{\partial t}(t)\bigg\|_0^2 \mathrm{d}t\\
&+\frac{2 \sigma^2 k_m^2 \Delta s^2}{\eta_{m} C_{mt} \rho \mu^2} \int_{0}^{T}\bigg\| \frac{\partial P^h \phi_{f}}{\partial t}(t) \bigg\|_0^2 \mathrm{d}t
+ \frac{\tilde{C}_{inv}^2 r^2 \Delta t^2}{\rho^2 h} \int_{t_{n_k}}^{t_{n_{k+1}}} \bigg\| \frac{\partial P^h \phi_f}{ \partial t}(t) \bigg\|_0^2 \mathrm{d}t
+\frac{4 \gamma C_T^2 \Delta t^2}{\rho h}\int_{0}^{T}\bigg\| \frac{\partial P^h \boldsymbol{u}_c}{\partial t}(t)\bigg\|_1^2 \mathrm{d}t \\
&+ Ch^4 \Delta t
+ \Delta t \sum_{k=0}^{j+1}\sum_{n=n_k}^{n_{k+1}-1}(\frac{2 C_N^2 C_{PF}^2 C_B^2 }{\nu}\|e_c^{n+1}\|_0^2 + \frac{\tilde{C}_{inv}^2+2 C_{div}^2}{h}\|e_c^{n+1}\|_0^2
+ \frac{30 C_N^2 C_S^2 C_{PF}^2 C_H^2}{\nu} \|e_c^n\|_0^2)\\
&+ \Delta t \sum_{k=0}^{j+1} \sum_{n=n_k}^{n_{k+1}-1} (\frac{\eta_f C_{ft} }{\rho} \|\theta_f^{n_{k+1}}\|_0^2
+\frac{\eta_m C_{mt} }{\rho} \|\theta_m^{n_{k+1}}\|_0^2
+\frac{4 \tilde{C}_{inv}^2}{\rho \gamma}\|\theta_f^{n_k}\|_0^2)\\
&+ \|e_c^0\|_0^2 + \frac{\eta_f C_{ft}}{\rho}\|\theta_f^0\|_0^2
+ \frac{\eta_m C_{mt}}{\rho}\|\theta_m^0\|_0^2
+ \frac{\sigma k_m \Delta s}{\rho \mu}(\|\theta_f^0\|_0^2 + \|\theta_m^0\|_0^2)
+ \frac{\gamma \Delta s}{\rho h} \|e_f^{0} \cdot \boldsymbol{n}_d\|_{L^2(\mathbb{I})}^2.
\end{align*}
By discrete Gronwall inequality, when~$\kappa_n \Delta t \leq \frac{1}{2}$, where~$\kappa_n=\kappa_1+\kappa_2+\kappa_3, \kappa_1=\frac{2 C_N^2 C_{PF}^2 C_B^2 }{\nu}+\frac{\tilde{C}_{inv}+2 C_{div}^2}{h}+\frac{30 C_N^2 C_S^2 C_{PF}^2 C_H^2}{\nu}, \kappa_2=\frac{\eta_f C_{ft}}{\rho}+\frac{4 C_{inv}^2}{\rho \gamma},
\kappa_3=\frac{\eta_m C_{mt}}{\rho}$~and~$\Delta t =\mathcal{O} (h^2)$,
we conclude that
\begin{equation*}
\begin{split}
&\max_{1 \leq j \leq M-2}\bigg\{\|e_c^{n_{j+2}}\|_0^2 + \frac{\eta_f C_{ft}}{\rho} \|\theta_f^{n_{j+2}}\|_0^2
+ \frac{\eta_m C_{mt}}{\rho} \|\theta_m^{n_{j+2}}\|_0^2 \bigg\}
+\nu \Delta t \sum_{k=0}^{j+1} \sum_{n=n_k}^{n_{k+1}-1}\|\nabla e_c^{n+1}\|_0^2\\
&+ \frac{2 \mu \Delta s}{\rho k_f}\sum_{k=0}^{j+1}\|e_f^{n_{k+1}}\|_0^2
+ \frac{2 \mu \Delta s}{\rho k_m}\sum_{k=0}^{j+1}\|e_m^{n_{k+1}}\|_0^2
\leq C(h^4 + \Delta t^2 + \Delta t^2 h^{-1}).
\end{split}
\end{equation*}
Further more, by imbedding theorem~and~inverse inequality, $n=n_k+1,...,n_{k+1}, 0 \leq k \leq M-1$,
\begin{equation*}
\begin{split}
\max_{\substack{ 0 \leq k \leq M-1 \\ n \in \{n_{k}+1,...,n_{k+1}\}}} \|\boldsymbol{u}_c^n\|_{L^{\infty}} &\leq \max_{\substack{ 0 \leq k \leq M-1 \\ n \in \{n_{k}+1,...,n_{k+1}\}}}(\|e_c^n\|_{L^{\infty}}+\|P^h \boldsymbol{u}_c(t_n)\|_{L^{\infty}})\\
&\leq Ch^{-3/2}  \max_{\substack{ 0 \leq k \leq M-1 \\ n \in \{n_{k}+1,...,n_{k+1}\}}}\|e_c^n\|_0
+ C\max_{\substack{ 0 \leq k \leq M-1 \\ n \in \{n_{k}+1,...,n_{k+1}\}}} \|\boldsymbol{u}_c(t_n)\|_2\\
&\leq C_B,
\end{split}
\end{equation*}
and
\begin{equation*}
\begin{split}
\Delta t  \sum_{k=0}^{M-1}\sum_{n=n_k+1}^{n_{k+1}}\|\boldsymbol{u}_c^n\|_{W^{1,\infty}}^2 &\leq \Delta t \sum_{k=0}^{M-1}\sum_{n=n_k+1}^{n_{k+1}} (\|e_c^n\|_{W^{1,\infty}}^2 + \|P^h \boldsymbol{u}(t_n)\|_{W^{1,\infty}}^2)\\
&\leq C h^{-3} \Delta t \sum_{k=0}^{M-1}\sum_{n=n_k+1}^{n_{k+1}} \|\nabla e_c^n\|_0^2 + C \Delta t \sum_{k=0}^{M-1}\sum_{n=n_k+1}^{n_{k+1}} \|\boldsymbol{u}_c(t_n)\|_{W^{2,d^{\ast}}}^2\\
&\leq C_B,
\end{split}
\end{equation*}
where~$d^{\ast} > D$.

\end{proof}

From the triangle inequality and the approximation properties~\eqref{pro1}\eqref{pro2}\eqref{pro3}, we show the following corollary.
\begin{corollary}[Error convergence]\label{cor1}
Let assumptions of Theorem~$\ref{thm1}$~hold, we have the following error convergence
\begin{align}
\max_{0 \leq m \leq M}(\|\boldsymbol{u}_c(t_m)-\boldsymbol{u}_c^m\|_0^2 + \|\phi_d(t_m)-\phi_d^{m}\|_0^2)&\leq C(h^4 + \Delta t^2 + \Delta t^2 h^{-1}),
\end{align}
where
\begin{align*}
\|\phi_d(t_m)-\phi_d^{m}\|_0^2:=\|\phi_f(t_m)-\phi_f^{m}\|_0^2+\|\phi_m(t_m)-\phi_m^{m}\|_0^2.
\end{align*}
\end{corollary}
~

\section{Numerical examples}

In this section, some numerical examples are used to show the convergence performance and efficiency of decoupled modified characteristic finite element method with different subdomain time steps for the mixed stabilized formulation. The first example's results validate the optimal convergence order with different time steps. Furthermore, by changing the penalty parameter to search the impact for the convergence performance. Finally, we use 2-cores CPU to solve this problem. The second example is used to investigate the performance of practical issue by adjusting some different physical parameters. Such as the velocity of injection well, the matrix velocity on
the boundaries and the deep relationship between injection wellborn and horizontal open-hole completion wellborn. In this way, we can get some comparative experimental phenomena. For the finite element space constructed in Section 2, we use MINI elements (P1b-P1) for the Navier-Stokes equations and the Brezzi-Douglas-Marini (BMD1) for the microfracture flow velocity~$\boldsymbol{u}_f$~and matrix flow velocity~$\boldsymbol{u}_m$~in two different Darcy equations. The corresponding pressure~$\phi_f$~and~$\phi_m$~use piecewise constant elements P0. The FreeFEM++ package~\cite{Hecht2012New} is used to carry out experiments and all results are generated by the same computer.

\subsection{Example~1}
The solving area of model problem is~$\Omega=\Omega_c \cup \Omega_d$, where the conduct domain is~$\Omega_c = [0,1] \times [1,2]$~and the dual-porosity domain is~$\Omega_d = [0,1] \times [0,1]$. The non-slip interface is~$\mathbb{I} = (0,1) \times \{1\}$. The analytical solutions satisfying the  dual-porosity-Navier-Stokes model is as follows:
\begin{align*}
\boldsymbol{u}_c &=\bigg [  (x^2(y-1)^2+y)\cos(t),~ -\frac{2}{3}x(y-1)^3\cos(t)+(2-\pi \sin(\pi x)) \cos(t) \bigg ]^{T},\\
p_c &= (2-\pi \sin(\pi x) )  \sin(0.5 \pi y) \cos(t),\\
\phi_f &= (2-\pi \sin (\pi x) ) (1-y-\cos(\pi y)) \cos(t),\\
\phi_m &= (2-\pi \sin(\pi x)) (\cos(\pi (1-y))) \cos(t),\\
\boldsymbol{u}_f &= -\frac{k_f}{ \mu} \nabla \phi_f,\\
\boldsymbol{u}_m &= -\frac{k_m}{ \mu} \nabla \phi_m.
\end{align*}
In addition, the initial condition, boundary conditions and the forcing terms follow the analytical solutions.
For simplicity, take the parameters~$\nu, \mu, \sigma, \alpha, \rho, \eta_f, \eta_m, C_{ft}, C_{mt}, k_f, k_m$~are all~1.0~in numerical example, penalty parameter~$\gamma=0.1$, and the final time~$T=0.5$.

\subsubsection{Convergence performance and CPU comparation with different subdomain time steps}

We take varying space steps~$h=1/4, 1/8, 1/16, 1/32, 1/64$, different time step ratios~$r=1, 2, 4, 8$~and choose the corresponding time step~$\Delta t =h^2$. Calculate the errors between the exact solution and numerical solution of~$L^2$-norm and~$H^1$-seminorm for the velocity separately. And we compute the pressure and its~$L^2$-norm. In addition, the corresponding rate of convergence are obtained.

From Table 1-4, we see that $L^2$-norm convergence rate of~$\boldsymbol{u}_c, \boldsymbol{u}_f$~and~$\boldsymbol{u}_m$~is~2.0. The~$H^1$-seminorm convergence rate of~$\boldsymbol{u}_c$, the $L^2$-norm convergence rate of~$\phi_f$~and~$\phi_m$~is~1.0. In other words, the error in sense of~$L^2$-norm~obtains the optimal convergence rate~$\mathcal{O}(h^2)$, and the error in sense of semi-$H^1$~norm obtains the optimal rate~$\mathcal{O}(h)$.

\begin{table}[H]
\caption{The convergence performance for~$\Delta t=h^2$~at~$r=1$}
\centering
\begin{tabular}{ccccccc}
\hline
$h$ &  $\|\boldsymbol{u}_c -\boldsymbol{u}_c^h\|_0$ &
Rate &
 $\|\boldsymbol{u}_f -\boldsymbol{u}_f^h\|_0$&
Rate &
$\|\boldsymbol{u}_m -\boldsymbol{u}_m^h\|_0$&
Rate \\
\hline
$\frac{1}{4}$  & 0.129317    & --      & 0.536254   & --     & 0.482127    &--    \\
$\frac{1}{8}$  & 0.029965    & 2.11    & 0.115880   & 2.21   & 0.111867    &2.11  \\
$\frac{1}{16}$ & 0.007332    & 2.03    & 0.023838   & 2.28   & 0.024654    &2.18  \\
$\frac{1}{32}$ & 0.001793    & 2.03    & 0.005283   & 2.17   & 0.005993    &2.04  \\
$\frac{1}{64}$ & 0.000447    & 2.00    & 0.001336   & 1.98   & 0.001532    &1.97  \\
\hline
$h$ & $\|\nabla(\boldsymbol{u}_c -\boldsymbol{u}_c^h )\|_0$&
Rate &
$\|\phi_f-\phi_f^h\|_0$ &
Rate &
$\|\phi_m-\phi_m^h\|_0$ &
Rate \\
\hline
$\frac{1}{4}$  & 1.660150   & --       & 0.240974  & --       & 0.264592  &--     \\
$\frac{1}{8}$  & 0.703788   & 1.24     & 0.122430  & 0.98     & 0.135008  & 0.97  \\
$\frac{1}{16}$ & 0.333161   & 1.08     & 0.054168  & 1.18     & 0.061704  & 1.13  \\
$\frac{1}{32}$ & 0.158962   & 1.07     & 0.027427  & 0.98     & 0.031148  & 0.99  \\
$\frac{1}{64}$ & 0.078633   & 1.02     & 0.014070  & 0.96     & 0.015801  & 0.98  \\
\hline
\end{tabular}
\end{table}

\begin{table}[H]
\caption{The convergence performance for~$\Delta t=h^2$~at~$r=2$}
\centering
\begin{tabular}{ccccccc}
\hline
$h$ &  $\|\boldsymbol{u}_c -\boldsymbol{u}_c^h\|_0$ &
Rate &
 $\|\boldsymbol{u}_f -\boldsymbol{u}_f^h\|_0$&
Rate &
$\|\boldsymbol{u}_m -\boldsymbol{u}_m^h\|_0$&
Rate \\
\hline
$\frac{1}{4}$  & 0.129303    & --      & 0.530867   & --     & 0.480194    &--    \\
$\frac{1}{8}$  & 0.029945    & 2.11    & 0.115879   & 2.20   & 0.111907    &2.10  \\
$\frac{1}{16}$ & 0.007323    & 2.03    & 0.023828   & 2.28   & 0.024687    &2.18  \\
$\frac{1}{32}$ & 0.001790    & 2.03    & 0.005282   & 2.17   & 0.006003    &2.04  \\
$\frac{1}{64}$ & 0.000446    & 2.01    & 0.001336   & 1.98   & 0.001535    &1.97  \\
\hline
$h$ & $\|\nabla(\boldsymbol{u}_c -\boldsymbol{u}_c^h )\|_0$&
Rate &
$\|\phi_f-\phi_f^h\|_0$ &
Rate &
$\|\phi_m-\phi_m^h\|_0$ &
Rate \\
\hline
$\frac{1}{4}$  & 1.660140   & --       & 0.235392  & --       & 0.263993  &--     \\
$\frac{1}{8}$  & 0.703787   & 1.24     & 0.122329  & 0.94     & 0.134997  & 0.97  \\
$\frac{1}{16}$ & 0.333162   & 1.08     & 0.054165  & 1.18     & 0.061704  & 1.13  \\
$\frac{1}{32}$ & 0.158962   & 1.07     & 0.027427  & 0.98     & 0.031148  & 0.99  \\
$\frac{1}{64}$ & 0.078633   & 1.02     & 0.014070  & 0.96     & 0.015801  & 0.98  \\
\hline
\end{tabular}
\end{table}

\begin{table}[H]
\caption{The convergence performance for~$\Delta t=h^2$~at~$r=4$}
\centering
\begin{tabular}{ccccccc}
\hline
$h$ &  $\|\boldsymbol{u}_c -\boldsymbol{u}_c^h\|_0$ &
Rate &
 $\|\boldsymbol{u}_f -\boldsymbol{u}_f^h\|_0$&
Rate &
$\|\boldsymbol{u}_m -\boldsymbol{u}_m^h\|_0$&
Rate \\
\hline
$\frac{1}{4}$  & 0.129279    & --      & 0.522902   & --     & 0.475665    &--    \\
$\frac{1}{8}$  & 0.029894    & 2.11    & 0.115875   & 2.17   & 0.111991    &2.09  \\
$\frac{1}{16}$ & 0.007305    & 2.03    & 0.023810   & 2.28   & 0.024757    &2.18  \\
$\frac{1}{32}$ & 0.001783    & 2.03    & 0.005282   & 2.17   & 0.006024    &2.04  \\
$\frac{1}{64}$ & 0.000443    & 2.01    & 0.001337   & 1.98   & 0.001540    &1.97  \\
\hline
$h$ & $\|\nabla(\boldsymbol{u}_c -\boldsymbol{u}_c^h )\|_0$&
Rate &
$\|\phi_f-\phi_f^h\|_0$ &
Rate &
$\|\phi_m-\phi_m^h\|_0$ &
Rate \\
\hline
$\frac{1}{4}$  & 1.660130   & --       & 0.229101  & --       & 0.263075  &--     \\
$\frac{1}{8}$  & 0.703790   & 1.24     & 0.122142  & 0.91     & 0.134975  & 0.96  \\
$\frac{1}{16}$ & 0.333165   & 1.08     & 0.054160  & 1.17     & 0.061704  & 1.13  \\
$\frac{1}{32}$ & 0.158965   & 1.07     & 0.027427  & 0.98     & 0.031148  & 0.99  \\
$\frac{1}{64}$ & 0.078634   & 1.02     & 0.014070  & 0.96     & 0.015801  & 0.98  \\
\hline
\end{tabular}
\end{table}

\begin{table}[H]
\caption{The convergence performance for~$\Delta t=h^2$~at~$r=8$}
\centering
\begin{tabular}{ccccccc}
\hline
$h$ &  $\|\boldsymbol{u}_c -\boldsymbol{u}_c^h\|_0$ &
Rate &
 $\|\boldsymbol{u}_f -\boldsymbol{u}_f^h\|_0$&
Rate &
$\|\boldsymbol{u}_m -\boldsymbol{u}_m^h\|_0$&
Rate \\
\hline
$\frac{1}{4}$  & 0.129111    & --      & 0.517255   & --     & 0.463726    &--    \\
$\frac{1}{8}$  & 0.029782    & 2.12    & 0.115885   & 2.16   & 0.112212    &2.05  \\
$\frac{1}{16}$ & 0.007267    & 2.03    & 0.023784   & 2.28   & 0.024910    &2.17  \\
$\frac{1}{32}$ & 0.001774    & 2.03    & 0.005292   & 2.17   & 0.006069    &2.04  \\
$\frac{1}{64}$ & 0.000444    & 2.00    & 0.001349   & 1.97   & 0.001551    &1.97  \\
\hline
$h$ & $\|\nabla(\boldsymbol{u}_c -\boldsymbol{u}_c^h )\|_0$&
Rate &
$\|\phi_f-\phi_f^h\|_0$ &
Rate &
$\|\phi_m-\phi_m^h\|_0$ &
Rate \\
\hline
$\frac{1}{4}$  & 1.660340   & --       & 0.224156  & --       & 0.261685  &--     \\
$\frac{1}{8}$  & 0.703837   & 1.24     & 0.121832  & 0.88     & 0.134936  & 0.96  \\
$\frac{1}{16}$ & 0.333189   & 1.08     & 0.054151  & 1.17     & 0.061704  & 1.13  \\
$\frac{1}{32}$ & 0.158976   & 1.07     & 0.027427  & 0.98     & 0.031148  & 0.99  \\
$\frac{1}{64}$ & 0.078640   & 1.02     & 0.014070  & 0.96     & 0.015801  & 0.98  \\
\hline
\end{tabular}
\end{table}

Also, we record the corresponding time cost in different time steps.
\begin{table}[H]
\caption{The CPU cost performance with different time step}
\centering
\begin{tabular}{|c|ccccc|}
\hline
\diagbox{$r$}{CPU Time(s)}{$h$}& \Large$\frac{1}{4}$ & \Large$\frac{1}{8}$ & \Large$\frac{1}{16}$ & \Large$\frac{1}{32}$ & \Large$\frac{1}{64}$\\
\hline
1 & 0.14 & 1.29 & 13.05 & 190.89 & 3284.09\\
\hline
2 & 0.06 & 0.60 & 8.43 & 134.89 & 2356.03\\
\hline
4 & 0.05 & 0.49 & 6.76 & 107.78 & 1885.20\\
\hline
8 & 0.04 & 0.44 & 5.89 & 94.31 & 1649.38\\
\hline
\end{tabular}
\end{table}
When taking different~$r$, we obtain similar errors. From Table~5, as time step ratio~$r$~grows, the CPU time decreases. When~$h$~is smaller, the CPU time is so much less. Especially, $r=2$~is obvious.

Next, fixing the time step~$\Delta t=0.001$, we give the relative errors and the CPU time at~$r=1$~as follows.
\begin{table}[H]
\caption{The relative errors and time cost of the characteristic FEMs}
\centering
\begin{tabular}{cccccccc}
\hline
$h$ &  $\frac{\|\boldsymbol{u}_c-\boldsymbol{u}_c^h\|_0}{\|\boldsymbol{u}_c\|_0}$ &
$\frac{\|\nabla(\boldsymbol{u}_c-\boldsymbol{u}_c^h)\|_0}{\|\nabla \boldsymbol{u}_c\|_0}$ &
$\frac{\|\boldsymbol{u}_f-\boldsymbol{u}_f^h\|_0}{\|\boldsymbol{u}_f\|_0}$ &
$\frac{\|\phi_f - \phi_f^h \|_0}{\|\phi_f\|_0} $ &
$\frac{\|\boldsymbol{u}_m - \boldsymbol{u}_m^h \|_0}{\|\boldsymbol{u}_m\|_0}$ &
$\frac{\|\phi_m - \phi_m^h \|_0}{\|\phi_m\|_0} $ &
$\text{CPU(s)}$\\
\hline
$\frac{1}{4}$  & 0.075926  & 0.264307   & 0.129379    & 0.448490  &  0.102381  & 0.442385   &  8.50\\
$\frac{1}{8}$  & 0.017523  & 0.112033   & 0.027754    & 0.220726  &  0.023679  & 0.225042   &  17.96\\
$\frac{1}{16}$ & 0.004300  & 0.053037   & 0.005709    & 0.097581  &  0.005215  & 0.102845   &  47.70\\
$\frac{1}{32}$ & 0.001060  & 0.025306   & 0.001265    & 0.049407  &  0.001269  & 0.051915   &  184.04\\
$\frac{1}{64}$ & 0.000275  & 0.012519   & 0.000319    & 0.025345  &  0.000326  & 0.026335   &  835.67\\
\hline
\end{tabular}
\end{table}

At the same time, we apply the Newton iteration method to solve the Navier-Stokes equations as for a contrast of using modified characteristic finite element method. Also, we solve the problem in the same mesh.

\begin{table}[H]
\caption{The relative errors and time cost of the Newton iteration method}
\centering
\begin{tabular}{cccccccc}
\hline
$h$ &  $\frac{\|\boldsymbol{u}_c-\boldsymbol{u}_c^h\|_0}{\|\boldsymbol{u}_c\|_0}$ &
$\frac{\|\nabla(\boldsymbol{u}_c-\boldsymbol{u}_c^h)\|_0}{\|\nabla \boldsymbol{u}_c\|_0}$ &
$\frac{\|\boldsymbol{u}_f-\boldsymbol{u}_f^h\|_0}{\|\boldsymbol{u}_f\|_0}$ &
$\frac{\|\phi_f - \phi_f^h \|_0}{\|\phi_f\|_0} $ &
$\frac{\|\boldsymbol{u}_m - \boldsymbol{u}_m^h \|_0}{\|\boldsymbol{u}_m\|_0}$ &
$\frac{\|\phi_m - \phi_m^h \|_0}{\|\phi_m\|_0} $ &
$\text{CPU(s)}$\\
\hline
$\frac{1}{4}$  & 0.075927   &  0.264311    & 0.129379   & 0.448490  &  0.102381 &   0.442385  &  20.72\\
$\frac{1}{8}$  & 0.017515   &  0.112035    & 0.027753   & 0.220726  &  0.023679 &   0.225042  &  32.47\\
$\frac{1}{16}$ & 0.004293   &  0.053038    & 0.005710   & 0.097581  &  0.005215 &   0.102845  &  110.48\\
$\frac{1}{32}$ & 0.001051   &  0.025306    & 0.001266   & 0.049407  &  0.001269 &   0.051915  &  433.19\\
$\frac{1}{64}$ & 0.000260   &  0.012518    & 0.000319   & 0.025345  &  0.000326 &   0.026335  &  1442.66\\
\hline
\end{tabular}
\end{table}

Obviously, the two method obtain similar accuracy. Note that the time to solve is reduced greatly with modified characteristic finite element method. Therefore, when we increase the mesh size, the modified characteristic finite element method performances efficiently  in same accuracy. Also, when take different~$r$, we can get the same conclusion.

\subsubsection{The stability performance with different penalty parameters}
Next, we use the method to test the convergence performance in different penalty parameter~$\gamma=0, 0.0001, 0.001, 0.01, 1$. Show the log-log plot of the errors as follows.
\begin{figure}[H]
\begin{centering}
\begin{subfigure}[t]{0.3\textwidth}
\centering
\includegraphics[width=1.05\textwidth]{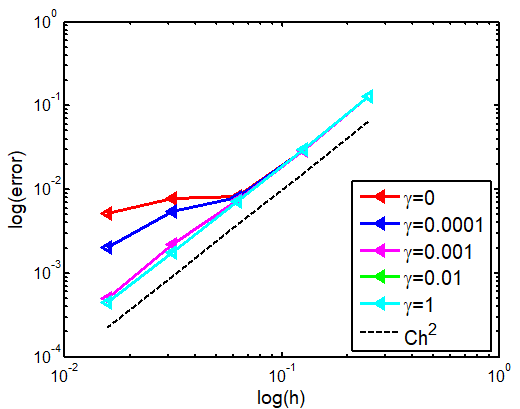}
\subcaption*{~~~~\footnotesize$\|\boldsymbol{u}_c-\boldsymbol{u}_c^h\|_0$}
\end{subfigure}
\quad
\begin{subfigure}[t]{0.3\textwidth}
\centering
\includegraphics[width=1.05\textwidth]{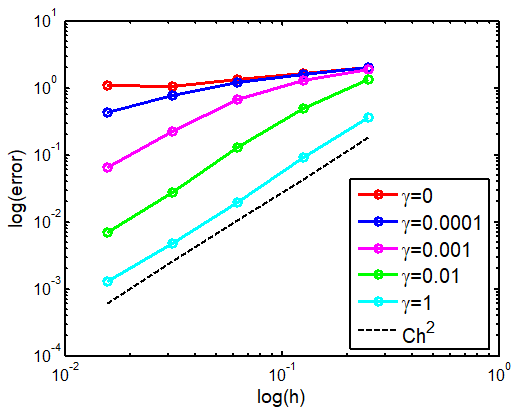}
\subcaption*{~~~~~~$\footnotesize\|\boldsymbol{u}_f-\boldsymbol{u}_f^h\|_0$}
\end{subfigure}
\quad
\begin{subfigure}[t]{0.3\textwidth}
\centering
\includegraphics[width=1.05\textwidth]{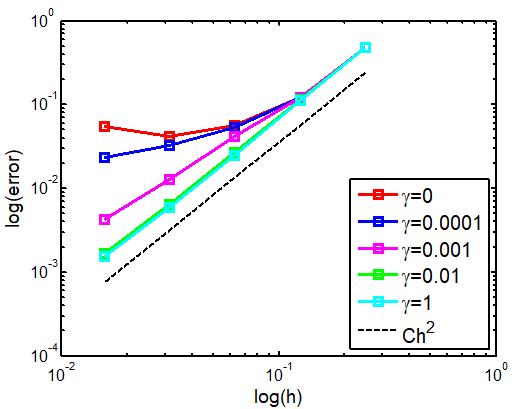}
\subcaption*{~~~~~~$\footnotesize\|\boldsymbol{u}_m-\boldsymbol{u}_m^h\|_0$}
\end{subfigure}
\end{centering}
\caption{\label{Fig1}The effect of the different values of the penalty parameter on the order of convergence for velocity.}
\end{figure}

\begin{figure}[H]
\begin{centering}
\begin{subfigure}[t]{0.3\textwidth}
\centering
\includegraphics[width=1.05\textwidth]{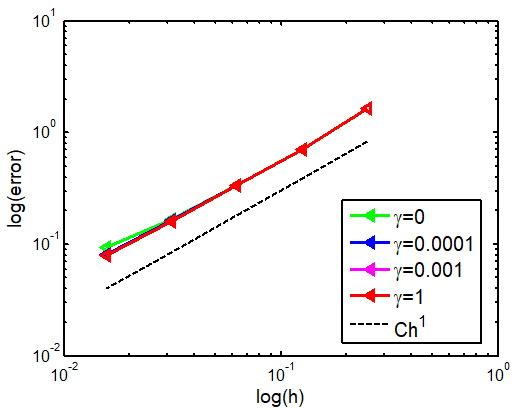}
\subcaption*{~~~~\footnotesize$\normalsize \|\nabla (\boldsymbol{u}_c-\boldsymbol{u}_c^h)\|_0$}
\end{subfigure}
\quad
\begin{subfigure}[t]{0.3\textwidth}
\centering
\includegraphics[width=1.05\textwidth]{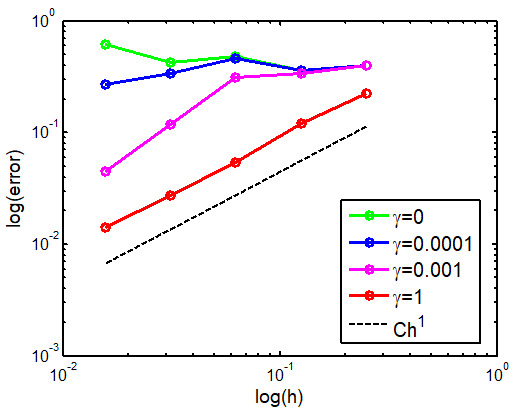}
\subcaption*{~~~~$\footnotesize\|\phi_f-\phi_f^h\|_0$}
\end{subfigure}
\quad
\begin{subfigure}[t]{0.3\textwidth}
\centering
\includegraphics[width=1.05\textwidth]{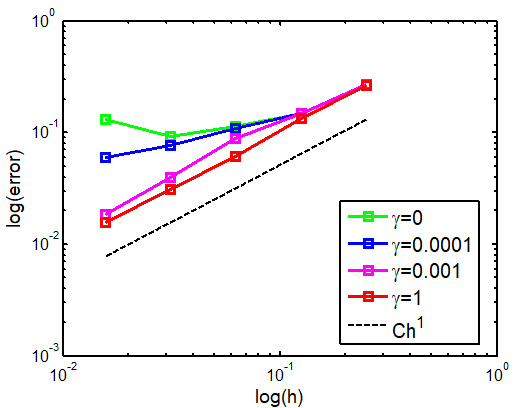}
\subcaption*{~~~~~~$\footnotesize\|\phi_m-\phi_m^h\|_0$}
\end{subfigure}
\end{centering}
\caption{\label{Fig2}The effect of the different values of the penalty parameter on the order of convergence for pressure.}
\end{figure}

\subsubsection{Parallel computing performance}
In order to improve the computational efficiency, we are thinking about the parallel algorithm in different domains based on modified characteristic finite element method.
\begin{table}[H]
\caption{The relative errors and time cost of the characteristic FEMs by MPI=2}
\centering
\begin{tabular}{cccccccc}
\hline
$h$ &  $\frac{\|\boldsymbol{u}_c-\boldsymbol{u}_c^h\|_0}{\|\boldsymbol{u}_c\|_0}$ &
$\frac{\|\nabla(\boldsymbol{u}_c-\boldsymbol{u}_c^h)\|_0}{\|\nabla \boldsymbol{u}_c\|_0}$ &
$\frac{\|\boldsymbol{u}_f-\boldsymbol{u}_f^h\|_0}{\|\boldsymbol{u}_f\|_0}$ &
$\frac{\|\phi_f - \phi_f^h \|_0}{\|\phi_f\|_0} $ &
$\frac{\|\boldsymbol{u}_m - \boldsymbol{u}_m^h \|_0}{\|\boldsymbol{u}_m\|_0}$ &
$\frac{\|\phi_m - \phi_m^h \|_0}{\|\phi_m\|_0} $ &
$\text{CPU(s)}$\\
\hline
$\frac{1}{4}$  & 0.075926  & 0.264307   & 0.129379    & 0.448490  &  0.102381  & 0.442385   &  3.07\\
$\frac{1}{8}$  & 0.017523  & 0.112033   & 0.027754    & 0.220726  &  0.023679  & 0.225042   &  9.69\\
$\frac{1}{16}$ & 0.004300  & 0.053037   & 0.005709    & 0.097581  &  0.005215  & 0.102845   &  36.67\\
$\frac{1}{32}$ & 0.001060  & 0.025306   & 0.001265    & 0.049407  &  0.001269  & 0.051915   &  143.76\\
$\frac{1}{64}$ & 0.000275  & 0.012519   & 0.000319    & 0.025345  &  0.000326  & 0.026335   &  644.51\\
\hline
\end{tabular}
\end{table}

Under the same precision, the CPU time are shorten. The computational efficiency get promoted.

\subsection{Example~2}
In this numerical example, we give the numerical simulation for the horizontal open-hole completion wellborn with a vertical production wellborn and a vertical injection wellborn. In petroleum engineering, in order to improve oil productivity, the injection wellborn is one of the most skills. In general, we will take some measures such as putting some salty water, the gas of CO2 and superheated steam and so on.
In the following, we adjust the pressure of wellborn to investigate the change of velocity and pressure in wellborn.

\begin{figure}[htbp]
  \centering
  \includegraphics[width=6.7cm]{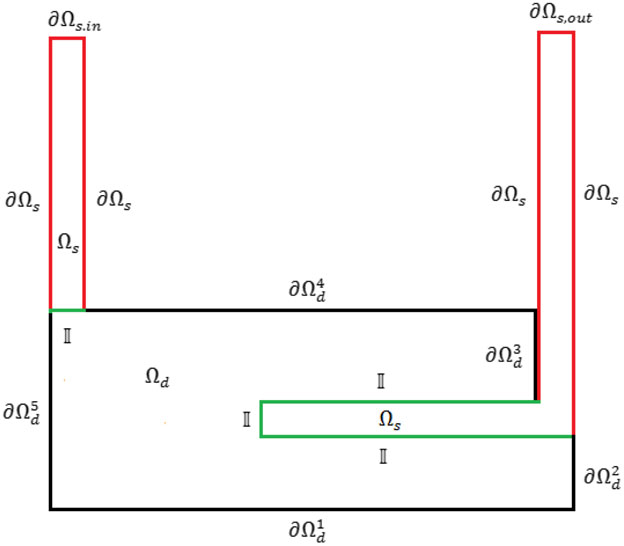}\\
  \caption{A sketch of the conduit region~$\Omega_c$, the dual-porosity region~$\Omega_d$~and the interface~$\mathbb{I}$.}
\end{figure}
The geometrical shape of numerical simulation as is shown in above picture. The interface between dual-porosity domain and conduit domain is~$\mathbb{I}=\{(x,y): y=3, 0 \leq x \leq 0.25 \}  \cup \{(x,y): y=1.38, 2\leq x \leq 6\} \cup \{(x,y): x=2, 1.38 \leq y \leq 1.63\} \cup \{(x,y): y=1.63, 2 \leq x \leq 5.75\}$. 

The dual-porosity area~$\Omega_d$~is made up microfracture area and matrix area. The boundary is~$\partial \Omega_d^1 =\{(x,y): y=0, 0 \leq x \leq 6\}, \partial \Omega_d^2=\{(x,y): x=6, 0 \leq y \leq 1.38\}, \partial \Omega_d^3 =\{(x,y): x=5.75, 1.63 \leq y \leq 3\}, \partial \Omega_d^4=\{(x,y): y=3, 0.25 \leq x \leq 5.75\}, \partial \Omega_d^5 =\{(x,y): x=0, 0 \leq y \leq 3\}$. In~$\partial \Omega_d^1, \partial \Omega_d^2, \partial \Omega_d^3, \partial \Omega_d^4, \partial \Omega_d^5$, the velocity in microfracture area is~$\boldsymbol{u}_f$, and the velocity in matrix area is~$\boldsymbol{u}_m$. Here, The velocity in the two domains are supposed in~$(0,0.5), (-0.5,0), (-0.5,0), (0,-0.5), (0.5,0)$~and~$(0,0.001), (-0.001,0), (-0.001,0),(0,-0.001),$\\
$(0.001,0)$ separately.  Furthermore, in contrast,  we will increase the velocity value~$(0,\Theta), (-\Theta,0),$\\
$(-\Theta,0), (0,-\Theta), (\Theta,0)$~on the boundary of matrix~$\boldsymbol{u}_m$, where the~$\Theta=0.05, 0.1$. At the same time, other parameters are not changed.

The domain of free flow is made up with a vertical injection and production wellborn with horizontal open-hole completion. The inflow boundary condition is applied on the top of the left vertical well~$\partial \Omega_{c,in}=\{(x,y): y=7, 0 \leq x \leq 0.25\}$, $U_{c1}=0, U_{c2}=-64x(0.25-x)$. Accordingly, on the top of the right vertical wellborn~$\partial \Omega_{c,out}=\{(x,y): y=7, 5.75 \leq x \leq 6\}$, we
apply the Neumann boundary condition~$(-p_c \mathrm{I} + \nu \nabla \boldsymbol{u}_c) \cdot \boldsymbol{n}_c = 0$. The height of injection wellborn and production wellborn is~$\partial \Omega_c = \{(x,y): x=0, 3 \leq y \leq 7\} \cup \{(x,y): x=0.25, 3 \leq y \leq 7\} \cup \{(x,y): x=6, 1.38 \leq y \leq 7\} \cup \{(x,y): x=5.75, 1.63 \leq y \leq 7\}$. In here, we impose the non-slip boundary condition~$\boldsymbol{u}_c=(0,0)$.

The model parameters are chosen as follows,  $\eta_f=10^{-4}, \eta_m=10^{-2}, C_{ft}=10^{-5}, C_{mt}=10^{-5}, k_f=10^{-4}, k_m=10^{-8}, \mu=10^{-2}, \nu=10^{-2}, \sigma=0.9, \alpha=1.0, \rho=1.0, \boldsymbol{f}_c=0, \boldsymbol{f}_d=0$, and~$\gamma=10$.  The mesh size and time step are taken~$h=1/32, \Delta t =0.01$. And the finial time~$T=5.0$.

\begin{figure}[H]
\begin{centering}
\begin{subfigure}[t]{0.3\textwidth}
\centering
\includegraphics[width=1.5\textwidth]{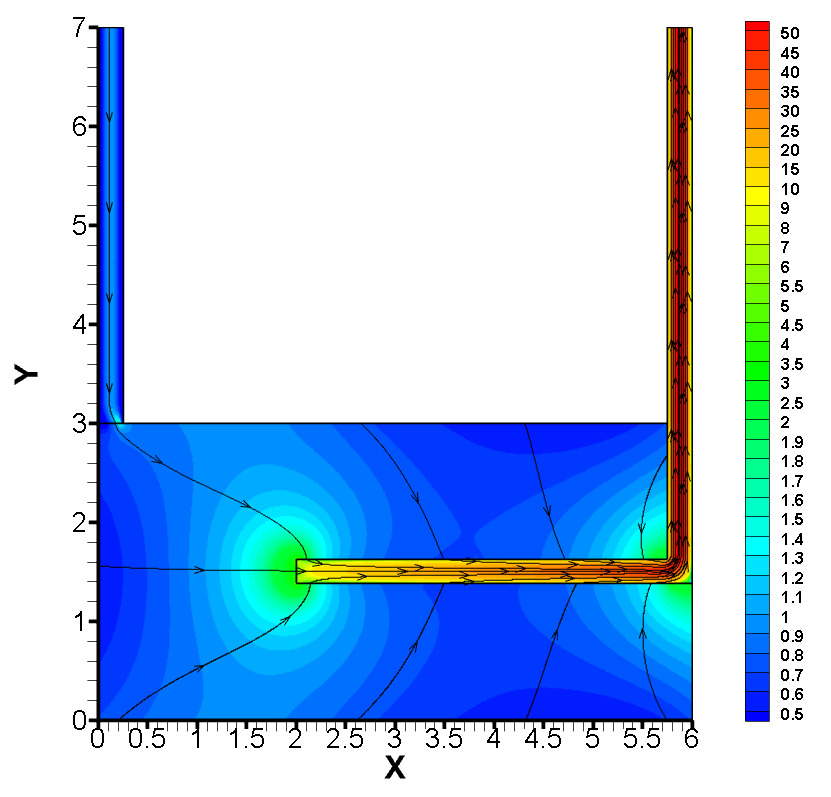}
\end{subfigure}
\hspace{30mm}
\begin{subfigure}[t]{0.3\textwidth}
\centering
\includegraphics[width=1.5\textwidth]{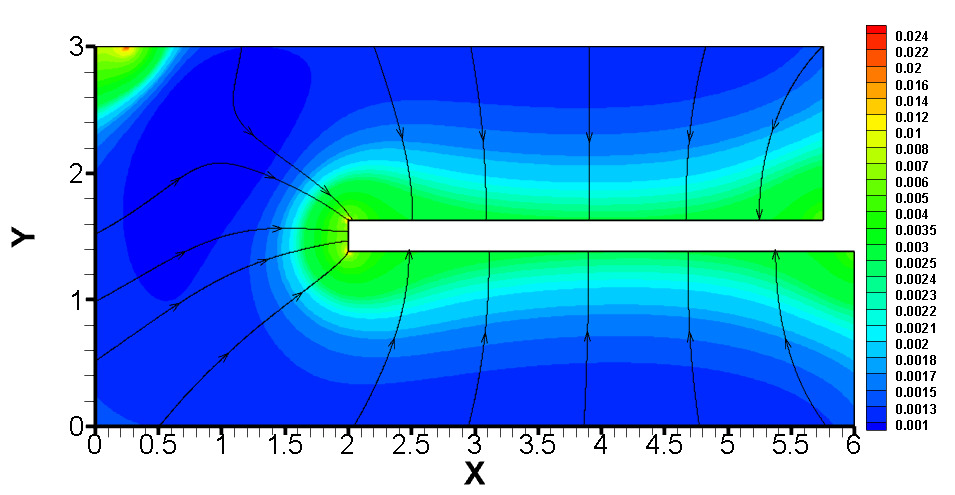}
\end{subfigure}
\end{centering}
\caption{\label{Fig3}The flow speed and streamlines around a imjection well and horizontal open-hole attached with vertical production wellbore completion. Left: the flow in the microfractures and conduits; Right: the flow in the matrix.}
\end{figure}

\begin{figure}[H]
\begin{centering}
\begin{subfigure}[t]{0.3\textwidth}
\centering
\includegraphics[width=1.5\textwidth]{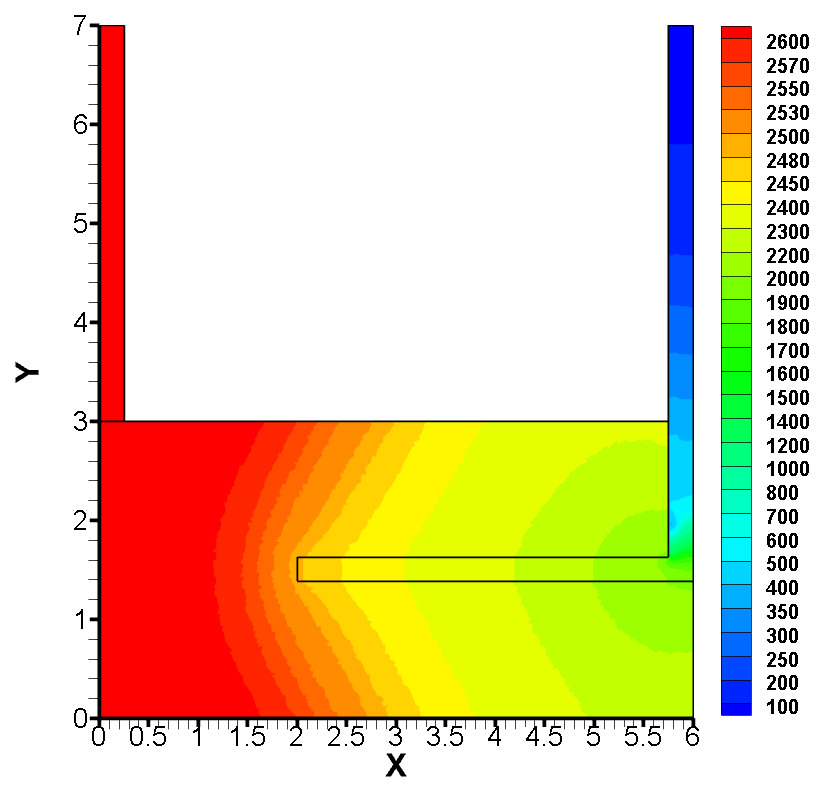}
\end{subfigure}
\hspace{30mm}
\begin{subfigure}[t]{0.3\textwidth}
\centering
\includegraphics[width=1.5\textwidth]{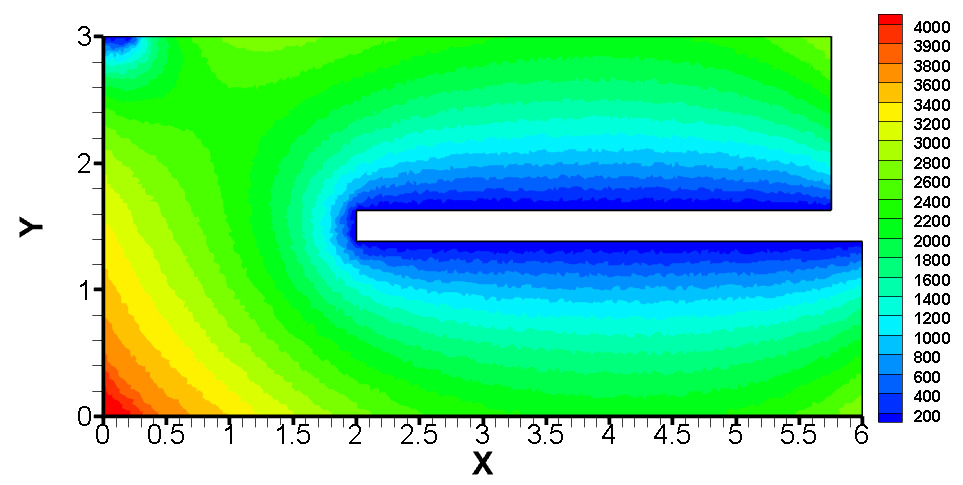}
\end{subfigure}
\end{centering}
\caption{\label{Fig4}The pressure around a injection well and horizontal open-hole attached with vertical production wellborn completion. Left: the pressure in the microfractures and conduits; Right: the pressure in the matrix.}
\end{figure}

\section{Conclusions}

In this paper, we develop and analyze the decoupled modified characteristic finite element method with different subdomain time steps for the mixed stabilized formulation of nonstationary dual-porosity-Navier-Stokes model.
Its main feature is a combination of characteristic methods and finite element methods, which leads to a decoupled and
fully discrete scheme without nonlinear terms. Under certain assumptions the $L^\infty$-norm of
the fully discrete velocity solution is uniformly bounded, and then we prove the
error convergence of the velocity and the corresponding pressures in sense of the $L^2$-norm and the $H^1$-seminorm. Further the numerical tests show the validity of the methods we develop. In order to improve our methods, several open problems remain to be solved, e.g., how to relax assumptions but still keep the uniform~$L^\infty$~boundedness of~$\boldsymbol{u}_c^n$; whether there exist other solutions of dealing with the interface terms in the fully discrete scheme for better convergence results with $H^1$-seminorm; the modified characteristic FEM mixed with other numerical methods such as stabilization methods and high-order time discretization methods are worthy of research items.

\section*{Acknowledgements}
This work is partially supported by the Major Research and Development Program of China under grant No. 2016YFB0200901 and the NSF of China under grant No. 11771348 and No. 11771259. Also grant for the Special support program to develop innovative talents in the region of Shaanxi province and the youth innovation team on computationally efficient numerical methods based on new energy problems in Shannxi province.

\newpage
\section*{References}
\bibliographystyle{plain}
\bibliography{bibfile}
\end{document}